\documentclass[11pt,reqno]{amsart}
\usepackage{eucal,fullpage,times,amsmath,amsthm,amssymb,mathrsfs,stmaryrd,color,enumerate,accents,enumitem}
\usepackage[all]{xy}
\usepackage{url}
\usepackage[colorlinks=true,hyperindex]{hyperref}

\newcommand{\calO}{{\mathcal{O}}}
\newcommand{\calA}{{\mathcal{A}}}
\newcommand{\calB}{{\mathcal{B}}}
\newcommand{\calC}{\mathcal{C}}

\newcommand{\calD}{\mathcal{D}}

\newcommand{\calF}{\mathcal{F}}
\newcommand{\calE}{\mathcal{E}}
\newcommand{\calG}{\mathcal{G}}
\newcommand{\calL}{\mathcal{L}}
\newcommand{\calM}{\mathcal{M}}
\newcommand{\calQ}{\mathcal{Q}}
\newcommand{\calP}{\mathcal{P}}
\newcommand{\calS}{\mathcal{S}}
\newcommand{\calX}{\mathcal{X}}

\newcommand{\Z}{\mathbf{Z}}
\newcommand{\G}{\mathbf{G}}
\newcommand{\N}{\mathbf{N}}
\newcommand{\C}{\mathbf{C}}
\newcommand{\F}{\mathbf{F}}
\newcommand{\Q}{\mathbf{Q}}
\newcommand{\A}{\mathbf{A}}
\newcommand{\E}{\mathbf{E}}

\renewcommand{\P}{\mathbf{P}}

\newcommand{\fp}{\mathrm{fp}}

\newcommand{\fppf}{\mathrm{fppf}}

\newcommand{\Spec}{{\mathrm{Spec}}}

\newcommand{\CAlg}{\mathrm{CAlg}}

\newcommand{\Ab}{\mathrm{Ab}}
\newcommand{\Hom}{\mathrm{Hom}}
\newcommand{\Isom}{\mathrm{Isom}}

\newcommand{\Fil}{\mathrm{Fil}}
\newcommand{\Quot}{\mathrm{Quot}}
\newcommand{\Sym}{\mathrm{Sym}}
\newcommand{\gr}{\mathrm{gr}}
\newcommand{\Tot}{\mathrm{Tot}}

\newcommand{\calEnd}{\mathcal{E}\mathrm{nd}}
\newcommand{\Aut}{\mathrm{Aut}}
\newcommand{\Ext}{\mathrm{Ext}}
\newcommand{\Tor}{\mathrm{Tor}}

\newcommand{\perf}{\mathrm{perf}}

\newcommand{\Vect}{\mathrm{Vect}}
\newcommand{\gVect}{\mathcal{V}\mathrm{ect}}

\newcommand{\Frob}{\mathrm{Frob}}
\newcommand{\Mod}{\mathrm{Mod}}

\renewcommand{\L}{\mathrm{L}}
\newcommand{\Pic}{\mathrm{Pic}}
\newcommand{\gPic}{\mathcal{P}\mathrm{ic}}
\newcommand{\et}{\mathrm{\acute{e}t}}

\newcommand{\ad}{\mathrm{ad}}

\newcommand{\id}{\mathrm{id}}
\newcommand{\coker}{\mathrm{coker}}

\renewcommand{\ker}{\mathrm{ker}}

\newcommand{\opp}{\mathrm{opp}}
\newcommand{\Sch}{\mathrm{Sch}}
\newcommand{\Perf}{\mathrm{Perf}}
\newcommand{\Gr}{\mathrm{Gr}}
\newcommand{\GL}{\mathrm{GL}}
\newcommand{\rk}{\mathrm{rk}}
\newcommand{\Spa}{{\mathrm{Spa}}}

\newcommand{\Fin}{\mathrm{Fin}}
\newcommand{\Mon}{\mathrm{Mon}}
\newcommand{\gp}{\mathrm{gp}}
\newcommand{\Sp}{\mathrm{Sp}}
\newcommand{\tors}{\mathrm{tors}}
\newcommand{\Frac}{\mathrm{Frac}}
\renewcommand{\det}{\mathrm{det}}
\newcommand{\Dem}{\mathrm{Dem}}
\newcommand{\on}{\, \mathrm{on}\, }
\newcommand{\twoinjlim}{2\text{-}\varinjlim}
\newcommand{\op}{\mathrm{op}}
\newcommand{\aff}{\mathrm{aff}}
\newcommand{\Gal}{\mathrm{Gal}}
\newcommand{\SL}{\mathrm{SL}}
\newcommand{\Map}{\mathrm{Map}}
\newcommand{\red}{\mathrm{red}}

\newcommand{\fram}{\mathfrak{m}}

\newcommand{\cosimp}[3]{\xymatrix@1{#1 \ar@<.4ex>[r] \ar@<-.4ex>[r] & {\ }#2 \ar@<0.8ex>[r] \ar[r] \ar@<-.8ex>[r] & {\ } #3 \ar@<1.2ex>[r] \ar@<.4ex>[r] \ar@<-.4ex>[r] \ar@<-1.2ex>[r] & \cdots }}

\newcommand{\colim}{\mathop{\mathrm{colim}}}

\newcommand{\adjunction}[4]{\xymatrix@1{#1{\ } \ar@<0.3ex>[r]^{ {\scriptstyle #2}} & {\ } #3 \ar@<0.3ex>[l]^{ {\scriptstyle #4}}}}

\begin{document}

\bibliographystyle{alpha}

\newtheorem{theorem}{Theorem}[section]
\newtheorem*{theorem*}{Theorem}
\newtheorem*{definition*}{Definition}
\newtheorem{proposition}[theorem]{Proposition}
\newtheorem{lemma}[theorem]{Lemma}
\newtheorem{corollary}[theorem]{Corollary}

\theoremstyle{definition}
\newtheorem{definition}[theorem]{Definition}
\newtheorem{question}[theorem]{Question}
\newtheorem{remark}[theorem]{Remark}
\newtheorem{example}[theorem]{Example}
\newtheorem{notation}[theorem]{Notation}
\newtheorem{convention}[theorem]{Convention}
\newtheorem{construction}[theorem]{Construction}
\newtheorem{claim}[theorem]{Claim}

\title{Projectivity of the Witt vector affine Grassmannian}
\author{Bhargav Bhatt and Peter Scholze}

\begin{abstract} We prove that the Witt vector affine Grassmannian, which parametrizes $W(k)$-lattices in $W(k)[\frac{1}{p}]^n$ for a perfect field $k$ of charactristic $p$, is representable by an ind-(perfect scheme) over $k$. This improves on previous results of Zhu by constructing a natural ample line bundle. Along the way, we establish various foundational results on perfect schemes, notably $h$-descent results for vector bundles.
\end{abstract}

\maketitle


\section{Introduction}

\subsection{Motivation and goals}
Fix a perfect field $k$ of characteristic $p$, and let $W(k)$ be the ring of Witt vectors of $k$. This paper deals with the question of putting an algebro-geometric structure (over $k$) on the set 
\[ \GL_n\big(W(k)[\frac{1}{p}]\big)/\GL_n(W(k)),\] 
or, equivalently, on the set of $W(k)$-lattices in $W(k)[\frac{1}{p}]^n$. This is a mixed-characteristic version of the affine Grassmannian, which (for the group $\GL_n$) parametrizes $k[[t]]$-lattices in $k((t))^n$. For the introduction, let us call the usual affine Grassmannian $\Gr^\aff$, and the Witt vector affine Grassmannian $\Gr^{W\aff}$.

Some of the interest in $\Gr^{W\aff}$ comes from its relation with the special fibers of Rapoport-Zink spaces. These are, via Dieudonn\'e theory, naturally described by affine Deligne-Lusztig varieties in the Witt vector affine Grassmannian\footnote{For example, if $\beta:G \to G_0$ is an isogeny of $p$-divisible groups over $k$, and a trivialization $W(k)[\frac{1}{p}]^n \simeq D(G_0)[\frac{1}{p}]$ of the Dieudonn\'e module of $G_0$ has been fixed, then the induced map $D(\beta)$ on Dieudonn\'e modules defines a point of $\Gr^{W\aff}$. In the equal characteristic case, the relation has been obtained by Hartl and Viehmann, \cite[Theorem 6.3]{HartlViehmannNewton}.}. However, as the Witt vector affine Grassmannian was only known as a set, the affine Deligne-Lusztig varieties were also only known as sets; this was a hindrance to talking systematically about notions like connected components or dimensions (though ad hoc definitions could be given in such cases, cf. e.g. \cite{ChenKisinViehmann}, \cite[\S 10]{Hamacher}). Another motivation is that in the context of forthcoming work of the second author, $\Gr^{W\aff}$ is supposed to appear as the special fibre of an object over $\Z_p$ whose generic fibre is an affine Grassmannian related to Fontaine's ring $B_{\mathrm{dR}}^+$, which is only defined on perfectoid algebras. In fact, the results in this paper show that some strange features of perfectoid spaces have more elementary analogues for perfect schemes. Notably, the $v$-topology on perfect schemes introduced in this paper is an analogue of a similar topology in the context of perfectoid spaces.

\subsection{Results}
We first recall the relevant structure on $\Gr^\aff$ that we intend to transport to $\Gr^{W\aff}$. The usual affine Grassmannian $\Gr^\aff$ is known to be represented by an ind-projective ind-scheme, and was first considered in an algebraic-geometric context by Beauville-Laszlo, \cite{BeauvilleLaszlo}; earlier work, motivated by the Korteweg-de-Vries equation, includes \cite{SatoSato} and \cite{SegalWilson}. More precisely, for integers $a \leq b$, one has closed subfunctors $\Gr^{\aff,[a,b]}\subset \Gr^\aff$ parametrizing lattices $M\subset k((t))^n$ lying between $t^a k[[t]]^n$ and $t^b k[[t]]^n$; varying the parameters gives a filtering system that exhausts $\Gr^\aff$, i.e., 
\[ \Gr^\aff \simeq \varinjlim \Gr^{\aff,[a,b]}. \]
Now each $\Gr^{\aff,[a,b]}$ admits a closed embedding into a (finite disjoint union of) classical Grassmannian(s) $\Gr(d,t^a k[[t]]^n / t^b k[[t]]^n)$ as those $k$-subvectorspaces $V\subset t^a k[[t]]^n / t^b k[[t]]^n$ which are stable under multiplication by $t$. As such, they are projective $k$-schemes, and all transition maps are closed embeddings; this proves that $\Gr^\aff$ is an ind-(projective scheme). In fact, there is a natural line bundle $\mathcal{L}$ on $\Gr^\aff$, given by
\[
\mathcal{L} = \det_k (t^a k[[t]]^n / M)
\]
on $\Gr^{\aff,[a,b]}$. As this line bundle is already ample on the classical Grassmannian $\Gr(d, t^a k[[t]]^n / t^b k[[t]]^n)$, it stays so on each $\Gr^{\aff,[a,b]}$.

Our aim here is to establish similar results in the Witt vector case, $\Gr^{W\aff}$. This question has been considered previously, notably by Haboush, \cite{Haboush}, Kreidl, \cite{Kreidl}, and Zhu, \cite{ZhuMixedCharGeomSatake}. The primary issue is that for a general $\F_p$-algebra $R$, its ring of Witt vectors $W(R)$ is pathological: it may contain $p$-torsion, and the natural map $W(R)/p\to R$ may not be an isomorphism. However, if $R$ is perfect, i.e. the Frobenius map $\Phi: R\to R$ is an isomorphism, then
\[
W(R) = \{\sum_{n\geq 0} [a_n] p^n\mid a_n\in R\}
\]
is well-behaved. In fact, $W(R)$ may be characterized as the unique (up to unique isomorphism) $p$-adically complete flat $\Z_p$-algebra lifting $R$. As already observed by Kreidl, \cite[Theorem 5]{Kreidl}, if one restricts to perfect rings $R$, then the set of $W(R)$-lattices $M$ in $W(R)[\frac{1}{p}]^n$ is a reasonable object, e.g. one has faithfully flat descent. Kreidl's work also shows that it is difficult to get representability results if one does not restrict to perfect rings.

Zhu, \cite{ZhuMixedCharGeomSatake}, then proved a representability result for $\Gr^{W\aff}$ restricted to perfect rings. There is still a similar presentation of $\Gr^{W\aff}$ as an increasing union of $\Gr^{W\aff,[a,b]}$ parametrizing lattices $M$ between $p^a W(R)^n$ and $p^b W(R)^n$. Zhu's result is that each $\Gr^{W\aff,[a,b]}$ can be represented by the perfection of a proper \emph{algebraic space} over $\F_p$. Our main theorem improves on this result of Zhu.

\begin{theorem} 
	\label{thm:mainthmIntro}
	The functor $\Gr^{W\aff,[a,b]}$ on perfect rings $R$, parametrizing $W(R)$-lattices $M\subset W(R)[\frac{1}{p}]^n$ lying between $p^a W(R)^n$ and $p^b W(R)^n$, is representable by the perfection of a projective algebraic variety over $\F_p$. Consequently, $\Gr^{W\aff}$ is representable by an inductive limit of perfections of projective varieties.
\end{theorem}

Our proof is independent of the work of Zhu. The crucial step is the construction of a natural line bundle $\mathcal{L}$ on $\Gr^{W\aff,[a,b]}$, analogous to the line bundle
\[
\mathcal{L} = \det_R (t^a R[[t]]^n/M)
\]
on $\Gr^{\aff,[a,b]}$.\footnote{The existence of this line bundle was already conjectured by Zhu.} However, as $p^a W(R)^n / M$ does not carry the structure of an $R$-module, it is not clear how to make sense of $\det_R ( p^a W(R)^n / M )$.

We give two solutions to this problem. Both make use of strong \emph{non-flat} descent properties for line bundles on perfect schemes. In fact, it turns out that Voevodsky's $h$-topology is subcanonical on the category of perfect schemes, and supports the ``correct'' bundle theory (see \S \ref{VDescent} for precise statements):

\begin{theorem}
	\label{thm:hdescentIntro}
Any vector bundle $\calE$ on a perfect $\F_p$-scheme $X$ gives a sheaf for the $h$-topology on perfect schemes over $X$ via pullback, and one has $H^i_h(X,\calE) \simeq H^i(X,\calE)$ for all $i$. Moreover, one has effective descent for vector bundles along $h$-covers of perfect schemes.
\end{theorem}

The first part of this theorem is due to Gabber, cf. \cite[\S 3]{BhattSchwedeTakagi}; the second part, in fact, extends to the full derived category (see \S \ref{sec:hDescentDerived}). Using this descent result, we can informally describe our first construction of $\mathcal{L}$, which is $K$-theoretic. The idea here is simply that one can often (e.g., when $R$ is a field, or {\em always} after an $h$-cover) filter $p^a W(R)^n / M$ in such a way that all gradeds $Q_i$ are finite projective $R$-modules, and then define
\[
\mathcal{L} = \bigotimes_i \det_R (Q_i)\ .
\]
The problem is then showing that this construction is independent of the choice of filtration; once $\mathcal{L}$ is canonically independent of the choices, one can use $h$-descent (i.e., Theorem \ref{thm:hdescentIntro}) to define it in general. In $K$-theoretic language, this amounts to constructing a natural map
\[
\widetilde{\det}: K(W(R)\on R)\to \gPic^\Z(R)
\]
from a $K$-theory spectrum\footnote{We are ultimately only interested in the statement at the level of $\pi_0$. However, the language of spectra, or at least Picard groupoids, is critical to carry out the descent.} to the groupoid of graded line bundles $\gPic^\Z(R)$; here the $K$-theory spectrum parametrizes perfect complexes on $W(R)$ that are acyclic after inverting $p$ (the relevant example is $p^a W(R)^n/M$ in the above notation), while the $\Z$-grading on the target keeps track of the fibral rank (and can be largely ignored at first pass). As we recall in the Appendix, taking determinants of projective modules defines a natural map
\[
\det: K(R)\to \gPic^\Z(R)\ ,
\]
so our problem can be reformulated as that of extending the map $\det$ along the forgetful map $\alpha: K(R)\to K(W(R)\on R)$. When $R$ is the perfection of a regular $\F_p$-algebra, this problem is easy to solve: $\alpha$ is an equivalence, thanks essentially to Quillen's d\'evissage theorem. In general, we use de Jong's alterations and $h$-descent for line bundles to reduce the problem to the previous case. 

Our second construction is more geometric. Here, we observe that on a suitable Demazure resolution 
\[ \widetilde{\Gr}^{W\aff,[a,b]}\to \Gr^{W\aff,[a,b]}\]
parametrizing filtrations of $p^a W(R)^n / M$ with gradeds being finite projective $R$-modules, there is a line bundle $\widetilde{\mathcal{L}}$ as above by definition. The problem becomes that of descending $\widetilde{\mathcal{L}}$. To handle such descent questions, we give the following criterion (see Theorem \ref{ThmVectTrivial}, Remark~\ref{RemVectTrivialnew}):

\begin{theorem} 
	\label{thm:VectTrivialIntro}
		Let $f: X\to Y$ be the perfection of a proper surjective map of $\F_p$-schemes. Assume that all geometric fibres of $f$ are connected. Then a vector bundle $\calE$ on $X$ descends (necessarily uniquely) to a vector bundle on $Y$ if and only if $\calE$ is trivial on all geometric fibres of $f$.
\end{theorem}

\begin{remark} In a previous version this theorem was stated under the stronger hypothesis $Rf_\ast \calO_X = \calO_Y$.
\end{remark}

In finite type situations, such a theorem implies descent of vector bundles after some finite purely inseparable map, which may have interesting applications.

Having constructed the line bundle $\calL$, we prove that it is ample by using a fundamental result of Keel, \cite{KeelSemiAmple}, on semiample line bundles in positive characteristic. Unfortunately, contrary to the situation in equal characteristic, we are not able to give a direct construction of enough sections of $\calL$ which would give a projective embedding.

We remark that it is a very interesting question whether there is a natural ``finite-type" structure on the Witt vector affine Grassmannian. For example, all minuscule Schubert cells, which parametrize lattices $M \subset W(k)[\frac{1}{p}]^n$ that differ from the standard lattice by (strict) $p$-torsion, are canonically the perfections of classical Grassmannians; it is natural to wonder if such a story extends deeper into the stratification. Nevertheless, for all questions of a ``topological'' nature (such as connected components, dimensions, or \'etale cohomology), it suffices understand the perfection: the functor $X\mapsto X_\perf$ preserves this information.

\subsection{Outline} 
We begin in \S \ref{sec:hsheaves} by collecting some basic notions surrounding Voevodsky's $h$-topology on the category of schemes (and its non-noetherian analogue, the $v$-topology, which relies on results of Rydh \cite{Rydh}); the main result is a criterion for an fppf sheaf to be an $h$-sheaf in terms of ``abstract blowup squares,'' see Theorem \ref{CritHSheaf}. Next, in \S \ref{sec:perfect}, we study perfect schemes and the perfection functor, and record some surprising vanishing and base change results that will be useful later. All this material is put to use in \S \ref{VDescent} to prove Theorem \ref{thm:hdescentIntro}. Using the results of \S \ref{sec:perfect} and \S \ref{VDescent}, in \S \ref{sec:DetConsKtheory} and \S \ref{sec:DetConsGeom}, we give the two promised constructions of line bundles; we stress that the results in \S \ref{sec:DetConsGeom}, which include Theorem \ref{thm:VectTrivialIntro} above, go beyond the construction of line bundles, and will be used later. Specializing to the problem at hand, in \S \ref{sec:torsionfamilies}, we collect some geometric observations about certain universal families parametrizing finite torsion $\Z_p$-modules filtered in a certain way; these are exploited to prove Theorem \ref{thm:mainthmIntro} in \S \ref{sec:WittAffGrGLn}. We briefly discuss the extension of Theorem \ref{thm:mainthmIntro} to general groups $G$ in \S \ref{sec:WittAffGrG}. The special case of $G = \SL_n$ is then studied in \S \ref{sec:WittAffGrSL}: we identify a central extension of $\SL_n(\Q_p)$ resulting from our construction of the line bundle in classical terms, and raise several related questions, mostly motivated by the corresponding equicharacteristic story.

In \S \ref{sec:hDescentDerived}, which is not used in the rest of the paper, we extend Theorem \ref{thm:hdescentIntro} to the derived category of quasi-coherent sheaves; here we use the language of $\infty$-categories, and rely on a notion recently introduced by Mathew in \cite{MathewGaloisGroup}. This section also contains some other results of independent interest: $h$-descent for Witt vector cohomology after inverting Frobenius (in \S \ref{ss:WittVectorCoh}) extending \cite[\S 3]{BBE}, a conceptual proof of Kunz's theorem on regular noetherian rings (in Corollary \ref{cor:Kunz}), a characterization of $h$-covers of noetherian schemes in terms of the derived category (in Theorem \ref{thm:hCoverClassical}), and derived $h$-descent for the full quasi-coherent derived category of noetherian schemes (in Theorem \ref{thm:hCechDescentDerivedQCoh}) extending \cite{HalpernLeistnerPreygel}.

Finally, in the Appendix \S \ref{Appendix}, we briefly review the construction of the determinant map $\det$ mentioned above in the language of $\infty$-categories and spectra; our goal here is to give an intuitive picture of the construction, and detailed proofs are not given.

\subsection*{Acknowledgments} This work started after the authors listened to a talk of Xinwen Zhu on his work at the MSRI, and the authors would like to thank him for asking the question on the existence of $\calL$. They would also like to thank Akhil Mathew for enlightening conversations related to \S \ref{ss:DescendableMaps}. Moreover, they wish to thank all the participants of the ARGOS seminar in Bonn in the summer term 2015 for their careful reading of the manuscript, and the many suggestions for improvements and additions. The first version of this preprint contained an error in the proof of Lemma~\ref{lem:VectcdhStack}, tracing back to an error in \cite[Corollary 3.3.2]{EGAIII}; we thank Christopher Hacon, and Linquan Ma (via Karl Schwede) and the anonymous referee for pointing this out. The authors are also indebted to the referee for providing numerous other comments that improved the readability of this paper. Finally, they would like to thank the Clay Mathematics Institute, the University of California (Berkeley), and the MSRI for their support and hospitality.  This work was done while B. Bhatt was partially supported by NSF grant DMS 1340424 and P. Scholze was a Clay Research Fellow.

\section{$h$-sheaves}
\label{sec:hsheaves}

In this section, we recall some general facts about sheaves on the $h$-topology defined by Voevodsky, \cite[\S 3]{VoevodskyHTop}. We use results of Rydh, \cite{Rydh}, in the non-noetherian case. In the following, all schemes are assumed to be qcqs for simplicity. Let us start by recalling the notion of universally subtrusive morphisms.

\begin{definition} A morphism $f: X\to Y$ of qcqs schemes is called \emph{universally subtrusive} or a {\em $v$-cover} if for any map $\Spec(V)\to Y$, with $V$ a valuation ring, there is an extension $V\hookrightarrow W$ of valuation rings and a commutative diagram
\[\xymatrix{
\Spec(W)\ar@{->>}[d]\ar[r]&X\ar[d]^f\\
\Spec(V)\ar[r]&Y.
}\]
We say that $f^\prime: X^\prime\to Y$ is a refinement of $f$ if $f^\prime$ is still universally subtrusive, and factors through $f$.
\end{definition}

By \cite[Corollary 2.9]{Rydh}, this agrees with \cite[Definition 2.2]{Rydh}. We also remark that any $v$-cover is submersive (and thus universally submersive), meaning that the map $|X|\to |Y|$ is a quotient map. If $Y$ is noetherian, universally submersive maps are $v$-covers, cf. \cite[Theorem 2.8]{Rydh}.

\begin{remark} The name $v$-cover, besides the similarity with the existing notion of $h$-covers, is meant to suggest surjectivity at the level of valuations. In fact, this can be made precise as follows. Recall that there is a fully faithful functor $X\mapsto X^\ad$ from schemes to adic spaces sending $\Spec(R)$ to $\Spa(R,R)$, the space of (equivalence classes of) valuations on $R$ which only take values $\leq 1$, cf. \cite{HuberGeneralization}. Then a map $f: X\to Y$ is a $v$-cover if and only if $|f^\ad|: |X^\ad|\to |Y^\ad|$ is surjective.
\end{remark}

Before going on, let us discuss several examples.

\begin{example}[{\cite[Remark 2.5]{Rydh}}] Let $f: X\to Y$ be a map of qcqs schemes. Then $f$ is a $v$-cover in any of the following cases.
\begin{enumerate}
\item[{\rm (i)}] The map $f$ is faithfully flat.
\item[{\rm (ii)}] The map $f$ is proper and surjective.
\item[{\rm (iii)}] The map $f$ is an $h$-cover in the sense of Voevodsky.
\end{enumerate}
Indeed, for (i), one can first lift the special point of the valuation ring, and then lift generalizations. For (ii), one can first lift the generic point of the valuation ring, and then use the valuative criterion of properness. Finally, (iii) follows from (i) and (ii) (as $h$-covers are generated by proper surjective maps and fppf maps). As an example of a surjective map $f: X\to Y$ that is not a $v$-cover, consider the following example, also given by Voevodsky: Let $Y=\mathbb{A}^2_k$ over some field $k$, let $\tilde{X}\to Y$ be the blow-up of $Y$ at the origin $(0,0)$, and let $X\subset \tilde{X}$ be the complement of a point in the exceptional locus. Any valuation on $Y$ that specializes from the generic point to $(0,0)$ in the direction corresponding to the missing point of $X$ does not admit a lift to $\tilde{X}$.
\end{example}

We need the following structural result about $v$-covers.

\begin{theorem}[{\cite[Theorem 3.12]{Rydh}}]\label{StructureUnivSubtr} Let $f: X\to Y$ be a finitely presented $v$-cover, where $Y$ is affine. Then there is a refinement $f^\prime: X^\prime\to Y$ of $f$ which factors as a quasi-compact open covering $X^\prime\to Y^\prime$ and a proper surjective map $Y^\prime\to Y$ of finite presentation.
\end{theorem}

It follows from the flattening techniques of Raynaud-Gruson, \cite{RaynaudGruson}, that, up to refinements, one can break up general proper surjective maps into blow-ups, and finitely presented flat maps. This is formalized in the following definition.

\begin{definition} Let $f: X\to Y$ be a proper surjective map of finite presentation between qcqs schemes. We say that $f$ is of inductive level $0$ if it can be refined into a composition of a proper fppf map $f^\prime: X^\prime\to Y^\prime$ and a finitely presented nilimmersion $Y^\prime\hookrightarrow Y$. For $n>0$, we say that $f$ is of inductive level $\leq n$ if $f$ admits a refinement $f^\prime: X^\prime\to Y$ which has a factorization $X^\prime\to X_0\to Y_0\to Y$, where
\begin{enumerate}
\item[{\rm (o)}] the map $Y_0\to Y$ is a finitely presented nilimmersion,
\item[{\rm (i)}] the map $X_0\to Y_0$ is proper surjective of finite presentation, and an isomorphism outside a finitely presented closed subset $Z\subset Y_0$ such that $X_0\times_{Y_0} Z\to Z$ is of inductive level $\leq n-1$, and
\item[{\rm (ii)}] the map $X^\prime\to X_0$ is a proper fppf cover.
\end{enumerate}
\end{definition}

These notions are preserved under base change. The following lemma ensures that every relevant map is of inductive level $\leq n$ for some $n$.

\begin{lemma}\label{ProperGood} Let $f: X\to Y$ be a proper surjective map of finite presentation between qcqs schemes. Then $f$ is of inductive level $\leq n$, for some $n\geq 0$.
\end{lemma}

\begin{proof} By noetherian approximation, we can assume that $Y$ is of finite type over $\Z$, cf. \cite[App. C, Theorem C.9]{ThomasonKtheory}. We now prove the more precise result that $f: X\to Y$ is of inductive level $\leq \dim Y$, by induction on $\dim Y$. We can assume that $Y$ is reduced, as we allowed base-changes by nilimmersions. If $\dim Y=0$, then $Y$ is a disjoint union of spectra of finite fields; then $X\to Y$ admits a point over a finite \'etale extension of $Y$, which shows that $X\to Y$ is of inductive level $0$.

Now assume that $\dim Y>0$. By generic flatness and \cite[Th\'eor\`eme 5.2.2]{RaynaudGruson} (cf. \cite[Proposition 3.6]{Rydh}), one can refine $X\to Y$ by $X^\prime\to X_0\to Y$, where $X^\prime\to X_0$ is a proper fppf cover, and $X_0\to Y$ is a blow-up which is an isomorphism outside a closed subset $Z\subset Y$ of $\dim Z\leq n-1$. By induction, $X_0\times_Y Z\to Z$ is of inductive level $\leq n-1$, as desired.
\end{proof}

Now fix a qcqs base scheme $S$, and consider the category $\Sch^\fp_{/S}$ of finitely presented $S$-schemes.

\begin{definition} The $h$-topology on $\Sch^\fp_{/S}$ is generated by finitely presented $v$-covers.
\end{definition}

We will refer to finitely presented $v$-covers as $h$-covers in the sequel; in cases of overlap, this agrees with Voevodsky's definition. Our next goal is to characterize $h$-sheaves amongst all presheaves in terms of some easily geometrically testable properties. 

\begin{proposition} Let $F$ be a presheaf (of sets) on $\Sch^\fp_{/S}$. Then $F$ is an $h$-sheaf if and only if the following conditions are satisfied.
\begin{enumerate}
\item[{\rm (i)}] The presheaf $F$ is a sheaf for the fppf topology.
\item[{\rm (ii)}] Let $Y=\Spec(A)\in \Sch^\fp_{/S}$ be an affine scheme, and $X\to Y$ a proper surjective map of finite presentation, which is an isomorphism outside a finitely presented closed subset $Z\subset Y$ with preimage $E\subset X$. Then the diagram
\[\xymatrix{
F(Y)\ar[r]\ar[d]&F(X)\ar[d]\\
F(Z)\ar[r]&F(E)
}\]
is a pullback square.
\end{enumerate}
\end{proposition}

Note in particular that if $X\hookrightarrow Y$ is a finitely presented nilimmersion of affine schemes of finite presentation over $S$, then taking $Z=X$ in (ii) (with preimage $E=X$), one sees that $F(Y)=F(X)$.

\begin{proof} Clearly, if $F$ is an $h$-sheaf, then (i) is satisfied. To verify (ii), note that $X\to Y$ is an $h$-cover. The sheaf axiom implies that $F(Y)$ is the equalizer of $F(X)\rightrightarrows F(X\times_Y X)$. But $X\sqcup E\times_Z E\to X\times_Y X$ is a further $h$-cover, so that $F(Y)$ is the equalizer of $F(X)\rightrightarrows F(X\sqcup E\times_Z E) = F(X) \times F(E\times_Z E)$. They clearly agree in the first component, so $F(Y)$ is the equalizer of $F(X)\rightrightarrows F(E\times_Z E)$. Now, given elements $a_X\in F(X)$ and $a_Z\in F(Z)$ which agree on $E$, the two pullbacks of $a_X$ to $E\times_Z E$ are both given by the pullback of $a_Z$ along $E\times_Z E\to Z$, giving an element $a_Y\in F(Y)$ restricting to $a_X$ on $X$. As $F(Z)\hookrightarrow F(E)$, $E\to Z$ being an $h$-cover, it follows that $a_Y$ also restricts to $a_Z$, as desired.

For the converse, we first check that $F$ is separated. Thus, take an $h$-cover $f: X\to Y$ in $\Sch^\fp_{/S}$ and two sections $a,b\in F(Y)$ that become equal on $X$. As $F$ is an fppf sheaf (and in particular, a Zariski sheaf), we may assume that $Y = \Spec(A)$ is affine by (i). Applying Theorem \ref{StructureUnivSubtr}, we may then assume that $X$ is a quasicompact open cover of a proper surjective map $Y^\prime\to Y$. Using (i) again, we can reduce to the case that $X\to Y$ is proper surjective. By Lemma \ref{ProperGood}, $X\to Y$ is of inductive level $\leq n$ for some $n$. We argue by induction on $n$. If $n=0$, then, after further refinement, $X\to Y$ is a composition of an fppf cover and a nilimmersion; using (i) and (ii), we see $a=b$. If $n>0$, then after further refinement, $X\to Y$ factors as a composition of nilimmersions, fppf covers, and a map $X^\prime\to Y^\prime$ which is an isomorphism outside some finitely presented $Z\subset Y^\prime$, for which $X^\prime \times_{Y^\prime} Z\to Z$ is of inductive level $\leq n-1$. Using (ii) and induction, one again finds that $a=b$, as desired.

Now we check that $F$ is a sheaf. Thus, take an $h$-cover $f: X\to Y$ in $\Sch^\fp_{/S}$ and a section $a\in F(X)$ whose two pullbacks to $F(X\times_Y X)$ agree. By (i), we may assume that $Y$ is affine. As we already proved that $F$ is separated, we are also free to replace $f$ by a refinement. Thus, by Theorem \ref{StructureUnivSubtr}, we may assume that $X$ factors as a composite of an fppf cover, and a proper surjective map; using (i) again, we may then assume that $X\to Y$ is a proper surjective map. By Lemma \ref{ProperGood}, $X\to Y$ is of inductive level $\leq n$ for some $n$. Using induction on $n$ and properties (i) and (ii) once more, we can assume that $X\to Y$ is an isomorphism outside some finitely presented $Z\subset Y$ with preimage $E\subset X$, and that $a\in F(X)$ induces a section of $F(E)$ that comes via pullback from $a_Z\in F(Z)$. Property (ii) gives us an element $a_Y\in F(Y)$ whose pullback to $X$ is $a$, as desired.
\end{proof}

In fact, the same result holds true for sheaves of spaces, and in particular for stacks. In the following, all limits and colimits are taken in the $\infty$-categorical sense (i.e., they are homotopy limits and homotopy colimits). The reader unfamiliar with the language of sheaves of spaces as developed by Lurie in \cite{LurieHTT} may assume that $F$ is a prestack in the following theorem; this is the most important case for the sequel. For ease of reference, we prefer to state the theorem in its natural generality.

\begin{theorem}\label{CritHSheaf} Let $F$ be a presheaf of spaces on $\Sch^\fp_{/S}$. Then $F$ is an $h$-sheaf if and only if it satisfies the following properties.
\begin{enumerate}
\item[{\rm (i)}] The presheaf $F$ is a sheaf for the fppf topology.
\item[{\rm (ii)}] Let $Y=\Spec(A)\in \Sch^\fp_{/S}$ be an affine scheme of finite presentation over $S$, and $X\to Y$ a proper surjective map of finite presentation, which is an isomorphism outside a finitely presented closed subset $Z\subset Y$ with preimage $E\subset X$. Then the diagram
\[\xymatrix{
F(Y)\ar[r]\ar[d]&F(X)\ar[d]\\
F(Z)\ar[r]&F(E)
}\]
is a (homotopy) pullback square.
\end{enumerate}
\end{theorem}

\begin{proof} Assume first that $F$ is a sheaf. Clearly, (i) is then satisfied. For (ii), we follow the proof of \cite[Lemma 3.6]{VoevodskyCDStructures}. Let $\mathcal{F}_X$ be the sheaf associated with $X$ for any $X\in \Sch^\fp_{/S}$; this is a sheaf of sets.\footnote{The $h$-topology is not subcanonical (as e.g. nilimmersions are covers), so one needs to sheafify.} As $F(X) = \Hom(\mathcal{F}_X,F)$ is the space of maps in the $\infty$-category of sheaves of spaces in the $h$-topology, it is enough to prove that $\mathcal{F}_Y$ is the pushout of $\mathcal{F}_X$ and $\mathcal{F}_Z$ along $\mathcal{F}_E$ in the $\infty$-category of sheaves of spaces in the $h$-topology. From the previous proposition, we know that this is true as sheaves of sets. But $\mathcal{F}_E(U)\hookrightarrow \mathcal{F}_X(U)$ is injective for all $U\in \Sch^\fp_{/S}$ (i.e., cofibrant as a map of simplicial sets), so the set-theoretic pushout
\[
\mathcal{F}_Z(U)\bigsqcup_{\mathcal{F}_E(U)} \mathcal{F}_X(U)
\]
is also a homotopy pushout. This shows that the pushout of $\mathcal{F}_X$ and $\mathcal{F}_Z$ along $\mathcal{F}_E$ in the $\infty$-category of presheaves of spaces on $\Sch^\fp_{/S}$ is still discrete; thus, its sheafification is also discrete, and then agrees with the pushout $\mathcal{F}_Y$ in the category of sheaves of sets.

For the converse, we have to show that for any $h$-cover $f: X\to Y$ in $\Sch^\fp_{/S}$ with Cech nerve $X^{\bullet/Y}\to Y$ given by the $n$-fold fibre products $X^{n/Y}$ of $X$ over $Y$, the map
\[
F(Y)\to \lim F(X^{\bullet/Y})
\]
is a weak equivalence. If this condition is satisfied, we say that $f$ is of $F$-descent; if this condition is also satisfied for any base change of $f$, we say that $f$ is of universal $F$-descent. We need the following very general d\'evissage lemma.

\begin{lemma}[{\cite[Lemma 3.1.2]{LiuZhengSixOp}}]\label{DevissageLemma} Let $f: X\to Y$ and $g: Y\to Z$ be morphisms in $\Sch^\fp_{/S}$ with composition $h=g\circ f: X\to Z$.
\begin{enumerate}
\item[{\rm (a)}] If $h$ is of universal $F$-descent, then $g$ is of universal $F$-descent.
\item[{\rm (b)}] If $f$ is of universal $F$-descent and $g$ is of $F$-descent, then $h$ is of $F$-descent. In particular, if $f$ and $g$ are of universal $F$-descent, then $h$ is of universal $F$-descent.
\end{enumerate}
\end{lemma}

First, we prove that all proper surjective maps $f: X\to Y$ in $\Sch^\fp_{/S}$ with $Y$ affine are of universal $F$-descent, by induction on the inductive level. If $f: X\to Y$ is of inductive level $0$, then it can be refined by a composition of a finitely presented nilimmersion and a proper fppf cover; using assumptions (i) and (ii) along with Lemma \ref{DevissageLemma} (a) shows that $f$ is of universal $F$-descent. If $f: X\to Y$ is of inductive level $\leq n$, then after refinement it is of the form $X\to X_0\to Y_0\to Y$ as in the definition. Here, $X\to X_0$ is proper fppf, and thus of universal $F$-descent by assumption (i), and $Y_0\to Y$ is a finitely presented nilimmersion, and thus of universal $F$-descent by (ii). Using Lemma \ref{DevissageLemma} (b), it is enough to prove that $X_0\to Y_0$ is of universal $F$-descent. Recall that $X_0\to Y_0$ is an isomorphism outside a finitely presented closed subset $Z\subset Y_0$ such that $X_0\times_{Y_0} Z\to Z$ is of inductive level $\leq n-1$.

For convenience, let us rename $X_0\to Y_0$ as $X\to Y$. Let $X^{n/Y}$ be the $n$-fold fibre product of $X$ over $Y$. Applying (ii) to $X^{n/Y}\to Y$, we get a pullback square
\[\xymatrix{
F(Y)\ar[r]\ar[d]&F(X^{n/Y})\ar[d]\\
F(Z)\ar[r]&F(E^{n/Z}).
}\]
By induction, $E\to Z$ is of (universal) $F$-descent, and so $F(Z)\to \lim F(E^{\bullet/Z})$ is a weak equivalence. Taking the limit of the pullback squares gives a pullback square
\[\xymatrix{
F(Y)\ar[r]\ar[d]&\lim F(X^{\bullet/Y})\ar[d]\\
F(Z)\ar[r]&\lim F(E^{\bullet/Z}).
}\]
As the lower map is a weak equivalence, it follows that the upper map is a weak equivalence, showing that $X\to Y$ is of $F$-descent. By the same argument, it is of universal $F$-descent, as desired.

Now, given a general $h$-cover $f: X\to Y$ in $\Sch^\fp_{/S}$, we want to prove that $f$ is of universal $F$-descent. Take an open affine cover $Y^\prime\to Y$, and let $X^\prime = X\times_Y Y^\prime$. Using Lemma \ref{DevissageLemma}, it is enough to prove that $X^\prime\to Y^\prime$ is of universal $F$-descent, in other words, we may assume that $Y$ is affine. Then, using Theorem \ref{StructureUnivSubtr}, $f: X\to Y$ can be refined by a composition of an fppf cover and a proper surjective map. Both of these are of universal $F$-descent (by (i), resp. the above), finally proving that $f$ is of universal $F$-descent, using Lemma \ref{DevissageLemma} once more.
\end{proof}

This yields the following criterion for detecting $h$-cohomological descent of abelian fppf sheaves:

\begin{corollary}\label{HCohom} Let $F$ be a sheaf of abelian groups on $\Sch^\fp_{/S}$. Assume that for any affine scheme $Y=\Spec(A)\in \Sch^\fp_{/S}$ of finite presentation over $S$, and $X\to Y$ a proper surjective map of finite presentation, which is an isomorphism outside a finitely presented closed subset $Z\subset Y$ with preimage $E\subset X$, the triangle
\[
R\Gamma_{\fppf}(Y,F)\to R\Gamma_{\fppf}(X,F)\oplus R\Gamma_{\fppf}(Z,F)\to R\Gamma_{\fppf}(E,F)
\]
in the derived category of abelian groups is distinguished. Then, for all $X\in \Sch^\fp_{/S}$, $R\Gamma_h(X,F) = R\Gamma_{\fppf}(X,F)$.
\end{corollary}

\begin{proof} For any $n\geq 0$, apply the previous theorem to the presheaf of spaces sending $X$ to $\tau_{\geq 0} R\Gamma_{\fppf}(X,F)[n]$ (or rather the simplicial set corresponding to it under the Dold-Kan correspondence). Then (i) is satisfied by definition, and (ii) by assumption.
\end{proof}

Moreover, if the sheaves are finitely presented, one can remove finite presentation constraints on the schemes. For this, we record the following simple observation.

\begin{lemma}\label{ApproxUnivSubtr} Let $f: X=\Spec(A)\to Y=\Spec(B)$ be a $v$-cover in $\Sch_{/S}$. Then one can write $f$ as a cofiltered limit of $h$-covers $f_i: X_i=\Spec(A_i)\to Y_i=\Spec(B_i)$ of finitely presented $X_i,Y_i\in \Sch^\fp_{/S}$.
\end{lemma}

\begin{proof} One can write $A$ as a filtered colimit of finitely presented $B$-algebras $A_i$. As $\Spec(A)\to \Spec(B)$ factors over $\Spec(A_i)$, it follows that $\Spec(A_i)\to \Spec(B)$ is still a $v$-cover. This reduces us to the case that $X\to Y$ is finitely presented. In that case, it comes as a base-change of some finitely presented map $X^\prime=\Spec(A^\prime)\to Y^\prime=\Spec(B^\prime)$ of finitely presented $X^\prime,Y^\prime\in \Sch^\fp_{/S}$. From Theorem \ref{StructureUnivSubtr}, it follows that $X^\prime\to Y^\prime$ is also a $v$-cover, possibly after changing $Y^\prime$.
\end{proof}

In order to work with perfect schemes, we need the following variant of the $h$-topology, which doesn't impose any finiteness constraints: 

\begin{definition} For a qcqs base scheme $S$ as above, the $v$-topology on the category $\Sch_{/S}$ of qcqs schemes over $S$ is the topology generated by $v$-covers.
\end{definition}

Then we get the following version of Theorem \ref{CritHSheaf}.

\begin{corollary}\label{SheafUnivSubtr} Let $F$ be a presheaf of $n$-truncated spaces on $\Sch_{/S}$ for some $n\geq 0$. Assume that the following conditions are satisfied.
\begin{enumerate}
\item[{\rm (i)}] The presheaf $F$ is a sheaf for the fppf topology.
\item[{\rm (ii)}] Let $Y=\Spec(A)\in \Sch_{/S}$ be an affine scheme of finite presentation over $S$, and $X\to Y$ a proper surjective map of finite presentation, which is an isomorphism outside a finitely presented closed subset $Z\subset Y$ with preimage $E\subset X$. Then the diagram
\[\xymatrix{
F(Y)\ar[r]\ar[d]&F(X)\ar[d]\\
F(Z)\ar[r]&F(E)
}\]
is a homotopy pullback square.
\item[{\rm (iii)}] If $Y=\Spec(A)\in \Sch_{/S}$ is a cofiltered limit of finitely presented $Y_i = \Spec(A_i)\in \Sch^\fp_{/S}$, then
\[
\colim F(Y_i)\to F(Y)
\]
is a weak equivalence.
\end{enumerate}
Then $F$ is a sheaf for the $v$-topology.
\end{corollary}

\begin{proof} By (i) and (ii), the restriction of $F$ to $\Sch^\fp_{/S}$ is an $h$-sheaf by Theorem \ref{CritHSheaf}. In general, by fppf descent, it is enough to prove that if $f: X=\Spec(A)\to Y=\Spec(B)$ is a $v$-cover, then $f$ is of $F$-descent. By Lemma \ref{ApproxUnivSubtr}, we can write $f$ as a cofiltered limit of $h$-covers $f_i: X_i=\Spec(A_i)\to Y_i=\Spec(B_i)$ between finitely presented $S$-schemes. Then, for each $i$,
\[
F(Y_i)\to \lim F(X_i^{\bullet/Y_i})
\]
is a weak equivalence. Passing to the filtered colimit over $i$ gives a weak equivalence
\[
F(Y)\simeq \colim_i F(Y_i)\simeq \colim_i (\lim F(X_i^{\bullet/Y_i}))\ .
\]
As $F$ is $n$-truncated for some $n$, one can commute the filtered colimit with the limit (see \cite[Corollary 4.3.7]{LurieDAGVIII} for a proof in the context of spectra, which is the only relevant case for the sequel). This gives 
\[
F(Y)\simeq \lim (\colim_i F(X_i^{\bullet/Y_i}) )\ ,
\]
where each
\[
\colim_i F(X_i^{n/Y_i})\simeq F(X^{n/Y})\ ,
\]
proving the claim.
\end{proof}

In particular, this applies to usual sheaves (of sets) and stacks. Similarly, we also have an analogue of Corollary \ref{HCohom}.

\begin{corollary}\label{VCohom} Let $F$ be a sheaf of abelian groups on $\Sch_{/S}$. Assume that for any affine scheme $Y=\Spec(A)\in \Sch^\fp_{/S}$ of finite presentation over $S$, and $X\to Y$ a proper surjective map of finite presentation, which is an isomorphism outside a finitely presented closed subset $Z\subset Y$ with preimage $E\subset X$, the triangle
\[
R\Gamma_{\fppf}(Y,F)\to R\Gamma_{\fppf}(X,F)\oplus R\Gamma_{\fppf}(Z,F)\to R\Gamma_{\fppf}(E,F)
\]
in the derived category of abelian groups is distinguished. Moreover, assume that if $Y=\Spec(A)\in \Sch_{/S}$ is a cofiltered limit of finitely presented $Y_i = \Spec(A_i)\in \Sch^\fp_{/S}$, then $F(Y) = \colim F(Y_i)$.

Then, for all $X\in \Sch_{/S}$, $R\Gamma_v(X,F) = R\Gamma_{\fppf}(X,F)$.$\hfill \Box$
\end{corollary}

\section{Perfect schemes}
\label{sec:perfect}

In this section, we collect some results on perfect schemes and the perfection functor. Notably, we prove that base change for quasi-coherent complexes holds true in the perfect setting without any flatness assumptions (see Lemma \ref{BaseChange}).

\begin{definition} A scheme $X$ over $\F_p$ is \emph{perfect} if the Frobenius map $\Frob_X: X\to X$ is an isomorphism.
\end{definition}

For any scheme $X/\F_p$,  write $X_\perf$ for its perfection, i.e., $X_\perf := \lim X$, with transition maps being Frobenius.

\begin{definition} Let $\Perf$ be the category of perfect qcqs schemes over $\F_p$ endowed with the \emph{$v$-topology}, generated by $v$-covers.
\end{definition}

\begin{remark} To avoid set-theoretic issues in the following (in particular, in defining $v$-cohomology), choose an uncountable strong limit cardinal $\kappa$, i.e. for all cardinals $\lambda<\kappa$, also $2^\lambda<\kappa$. Then replace $\Perf$ by the category of perfect qcqs schemes covered by $\Spec(A)$ with $|A|<\kappa$. Then $\Perf$ is essentially small, and all arguments will go through in this truncated version of $\Perf$.\footnote{In Lemma \ref{wLoc} below, we do a ``big'' construction, which however works in this truncated version of $\Perf$ as $\kappa$ is a strong limit cardinal. One can also do a more careful construction in Lemma \ref{wLoc} to keep the rings smaller, and allow more general $\kappa$.}
\end{remark}

In order to pass between the usual and perfect world, observe the following.

\begin{lemma}\label{SchVsPerfProp} Let $f: X\to Y$ be a morphism of (not necessarily qcqs) schemes over $\F_p$. The following properties hold for $f$ if and only if they hold for $f_\perf$.
\begin{enumerate}
\item[{\rm (i)}] quasicompact,
\item[{\rm (ii)}] quasiseparated,
\item[{\rm (iii)}] affine,
\item[{\rm (iv)}] separated,
\item[{\rm (v)}] integral,
\item[{\rm (vi)}] universally closed,
\item[{\rm (vii)}] a universal homeomorphism.
\end{enumerate}
Moreover, if one of the following properties holds for $f$, then it also holds for $f_\perf$.
\begin{enumerate}
\item[{\rm (viii)}] a closed immersion,
\item[{\rm (ix)}] an open immersion,
\item[{\rm (x)}] an immersion,
\item[{\rm (xi)}] \'etale,
\item[{\rm (xii)}] (faithfully) flat.
\end{enumerate}
\end{lemma}

\begin{proof} Note that $|X| = |X_\perf|$. In particular, all topological conditions are invariant, dealing with (i), (ii), (vi), (vii) (using that base-changes are also compatible), and the faithfully flat case in (xii) follows from the flat case. Moreover, (viii), (ix) and (thus) (x) are clear. For (xii), say $A \to B$ is a flat map of $\F_p$-algebras. Write $A_n := A$, and form the inductive system $\{A_n\}$ with transition maps given by Frobenius, so $\colim A_n = A_\perf$, and similarly on $B$. Then we can identify $B_\perf \simeq \colim A_\perf \otimes_{A_n} B_n$ as a filtered colimit of flat $A_\perf$-modules, which shows that $B_\perf$ is flat over $A_\perf$.

For (iii), we can assume that $Y$ is affine, and we need to prove that $X\in \Sch_{/\F_p}$ is affine if and only if $X_\perf$ is affine. Clearly, if $X$ is affine, then so is $X_\perf$. The converse follows from \cite[App. C, Proposition C.6]{ThomasonKtheory}.

If $f$ is separated, then $f_\perf$ is separated by (viii). Conversely, if $f_\perf$ is separated, then the diagonal morphism $\Delta_f$ is universally closed by (vi). But a morphism of schemes is separated if and only if $\Delta_f(X)\subset X\times_Y X$ is a closed subset.

Now, for (v), we can assume $f: X=\Spec(A)\to Y=\Spec(B)$ is affine. It is clear that if $B\to A$ is integral, then so is $B_\perf\to A_\perf$. Conversely, if $x\in A$ satisfies a monic polynomial equation over $B_\perf$ inside $A_\perf$, then some $p^n$-th power of $x$ satisfies a monic polynomial equation over $B$ inside $A$.

Finally, for (xi), it is enough to prove that if $f: X\to Y$ is \'etale, then the natural map $g: X_\perf\to X\times_Y Y_\perf$ is an isomorphism. But for this, it is enough to observe that the relative Frobenius map $X\to X\times_{Y,\Frob_Y} Y$ is an isomorphism, as it is a universal homeomorphism between \'etale $Y$-schemes.
\end{proof}

Next, we relate line bundles on $X$ and $X_\perf$. First, we record that perfection is not too lossy:

\begin{lemma}
	For any qcqs $\F_p$-scheme $X$, pullback along $X_\perf \to X$ induces $\Pic(X)[\frac{1}{p}] \simeq \Pic(X_\perf)$.
\end{lemma}
\begin{proof}
Since $X_\perf := \lim X$ is a limit of copies of $X$ along Frobenius, we have $\colim \Pic(X) \simeq \Pic(X_\perf)$, where the colimit is indexed by pullback along Frobenius on $X$; the latter raises a line bundle to its $p$-th power, giving the claim.
\end{proof}

Recall that if $X$ is a qcqs scheme, and $\calL$ a line bundle on $X$, then $\calL$ is called ample if for any $x\in X$, there exists a section $s\in \Gamma(X,\calL^{\otimes n})$ for some $n$, such that $X_s = \{y\in X\mid s(y)\neq 0\}$ is an affine neighborhood of $x$. (Cf. e.g. \cite[Tag 01PR]{StacksProject}.)

\begin{lemma}\label{AmplePerf} If $X\in \Sch_{/\F_p}$ and $\calL$ is a line bundle on $X$, then $\calL$ is ample if and only if the pullback of $\calL$ to $X_\perf$ is ample.
\end{lemma}

\begin{proof} Using the identification
\[
\Gamma(X_\perf,\calL^{\otimes n}) = \varinjlim_{f\mapsto f^p} \Gamma(X,\calL^{\otimes n p^m})\ ,
\]
this follows from $|X|=|X_\perf|$ and the affine part of Lemma \ref{SchVsPerfProp}.
\end{proof}

Topological invariance of the \'etale site (cf. e.g. \cite[Tag 04DY]{StacksProject}) implies that $X$ and $X_\perf$ have the same \'etale site.

\begin{theorem} Let $X$ be any $\F_p$-scheme with perfection $X_\perf$. Then $Y\in X_\et\mapsto Y_\perf\in (X_\perf)_\et$ is an equivalence of sites.$\hfill \Box$
\end{theorem}

In characteristic $p$, the perfection is a canonical representative in its universal homeomorphism class.

\begin{lemma}\label{UnivHomeomIsomPerf} If $f:X\to Y$ is a universal homeomorphism of perfect schemes, then $f$ is an isomorphism. In particular, if $f: X\to Y$ is a morphism of $\F_p$-schemes, then $f$ is a universal homeomorphism if and only if $f_\perf$ is an isomorphism.
\end{lemma}

\begin{proof} This is an easy consequence of \cite[Theorem 1]{Yanagihara}. Note that necessarily $f$ is integral (cf. \cite[Tag 04DC]{StacksProject}), so we may assume $Y=\Spec(A)$, $X=\Spec(B)$, where $B$ is integral over $A$. As $A$ is reduced, necessarily $A\to B$ is injective. Then the condition that $X\to Y$ is a universal homeomorphism is the condition that $B$ is weakly subintegral over $A$ in the language of \cite{Yanagihara}. On the other hand, \cite[Theorem 1]{Yanagihara} implies that $A$ is weakly normal in $B$: To check this, it is enough to see that if $b\in B$ is an element such that either both $b^2,b^3\in A$ or $b^p,pb\in A$, then $b\in A$. But in characteristic $p$, either condition implies $b^p\in A$, which by perfectness implies that $b\in A$. Thus, the weakly subintegral closure of $A$ in $B$ is equal to both $A$ and $B$, i.e. $A=B$.
\end{proof}

\begin{remark} In the context of this paper, a different proof of Lemma \ref{UnivHomeomIsomPerf} can be given as follows. Say $A\to B$ is a map of perfect rings inducing a universal homeomorphism on spectra. By $v$-descent (Theorem \ref{thm:VectWhStack}), it is enough to prove that $A\to B$ is an isomorphism after a $v$-localization. Thus, by Lemma \ref{wLoc}, we can reduce to the case that $A$ is w-local, with all local rings valuation rings. As one can check whether a map is an isomorphism on local rings, this reduces us to a valuation ring. But then there is a (unique) section over the generic point, which by integrality extends to the valuation ring. But the section $B\to A$ is surjective, and still a universal homeomorphism on spectra, and therefore (as $B$ is reduced) injective, i.e. bijective. Therefore, $A=B$, as desired.
\end{remark}

We will use the following terminology in the world of perfect schemes.

\begin{definition} Let $g: B\to A$ a map of perfect rings. The map $g$ is called perfectly finitely presented if $A = (A_0)_\perf$ for some finitely presented $B$-algebra $A_0$.
\end{definition}

Any such $B$-algebra $A_0$ is called a model for the $B$-algebra $A$. Note that any two models differ by finite purely inseparable morphisms.

\begin{proposition}\label{PropFP} Let $f: X\to Y$ be a morphism in $\Perf$.\footnote{Recall that this includes the assumption that $X$ and $Y$ are qcqs.} The following conditions are equivalent.
\begin{enumerate}
\item[{\rm (i)}] There is a covering of $X$ by open affine $\Spec(A_i)\subset X$ mapping into open affine $\Spec(B_i)\subset Y$ such that $B_i\to A_i$ is perfectly finitely presented.
\item[{\rm (ii)}] For any open affine $\Spec(A)\subset X$ mapping into an open affine $\Spec(B)\to Y$, the map $B\to A$ is perfectly finitely presented.
\item[{\rm (iii)}] For any cofiltered system $\{Z_i\}\in \Perf_{/Y}$ with affine transition maps and limit $Z=\lim Z_i\in \Perf_{/Y}$, the natural map
\[
\colim \Hom_{/Y}(Z_i,X)\to \Hom_{/Y}(Z,X)
\]
is a bijection.
\end{enumerate}

If these equivalent conditions are satisfied, then $f$ is called perfectly finitely presented.
\end{proposition}

\begin{proof} Clearly, (ii) implies (i). It is a standard exercise to deduce (iii) from (i), as in \cite[\S 8.8]{EGAIV}. Finally, to see that (iii) implies (ii), one reduces to the case $X=\Spec(A)\to Y=\Spec(B)$ is a map of affines. One can write $B\to A$ as a filtered colimit of perfectly finitely presented $B$-algebras $A_i$. Applying condition (iii) with $Z_i=\Spec(A_i)$ gives a map $A\to A_i$ of $B$-algebras such that $A\to A_i\to A$ is the identity. This implies that $A_i\to A$ is surjective, so replacing $B$ by $A_i$ we may assume that $B\to A$ is surjective, i.e. $X\subset Y$ is a closed subscheme. In that case, $X\subset Y$ is a cofiltered intersection of closed subschemes $X_i=\Spec(A_i)\subset Y$ with $B\to A_i$ perfectly finitely presented. Applying condition (iii) with $Z_i = X_i$ (and limit $Z=X$) shows that $X_i\subset X$ for some $i$. But $X\subset X_i$, so $X=X_i$ and $A=A_i$ is perfectly finitely presented.
\end{proof}

Moreover, one has the usual approximation properties for perfectly finitely presented morphisms. For any $X\in \Perf$, let $\Perf^\fp_{/X}\subset \Perf_{/X}$ be the full subcategory of perfectly finitely presented $Y\to X$.

\begin{proposition}\label{ApproxFP} Let $X_i$ be a cofiltered inverse system in $\Perf$ with affine transition maps.\footnote{We recall that all objects of $\Perf$ are assumed to be qcqs.} Then the limit $X=\lim X_i$ exists in $\Perf$, and agrees with the limit taken in schemes.

The natural pullback functor
\[
\twoinjlim_i \Perf^\fp_{/X_i}\to \Perf^\fp_{/X}
\]
is an equivalence.
\end{proposition}

\begin{proof} The standard argument as in \cite[\S 8.8]{EGAIV} works.
\end{proof}

All the perfectly finitely presented morphisms encountered in this paper will by their definitions be perfections of finitely presented morphisms. However, this is a general statement, cf. also \cite[Proposition A.13]{ZhuMixedCharGeomSatake} for a more geometric proof of a similar statement.

\begin{proposition}\label{FPModelsExist} Let $f: X\to Y$ be a perfectly finitely presented morphism in $\Perf$. Then there exists a finitely presented morphism $f_0: X_0\to Y$ of schemes such that $f=(f_0)_\perf$.
\end{proposition}

Any such $f_0: X_0\to Y$ will be called a model for $f$. Again, any two models differ by finite purely inseparable morphisms.

\begin{proof} First, we deal with the absolute case $Y=\Spec(\F_p)$. Using \cite[Theorem C.9]{ThomasonKtheory}, we can write $X$ as a cofiltered limit of $X_{i,0}\in \Sch^\fp_{/\F_p}$ with affine transition maps. Let $i_0$ be any chosen index, and replace the indexing category by the set of indices $i\geq i_0$. Let $X_i$ be the perfection of $X_{i,0}$. Then $X=\lim X_i$, and we write $f_i: X\to X_i$ for the projection maps. Proposition \ref{PropFP} (iii) implies that the map $X=\lim X_i\to X$ of perfect $X_{i_0}$-schemes factors through a map $g: X_i\to X$ for some $i\geq i_0$, i.e. the composite
\[
X\buildrel {f_i}\over\longrightarrow X_i\buildrel g\over\longrightarrow X\ ,
\]
is the identity, where all maps are maps over $X_{i_0}$. We claim that $f_i$ is a closed immersion. As everything is affine over $X_{i_0}$, we can assume that $X_i=\Spec(A_i)$ and $X=\Spec(A)$ are affine, by pulling back to an open affine of $X_{i_0}$. But then there are maps $A\to A_i\to A$ whose composite is the identity. Thus, $A_i\to A$ is surjective, meaning that $f_i$ is a closed immersion.

In other words, we can assume that $X\hookrightarrow X^\prime$ is a closed subscheme of some $X^\prime$ which is the perfection of a finitely presented $X^\prime_0\in \Sch^\fp_{/\F_p}$. Then $|X|\subset |X^\prime|=|X^\prime_0|$ is a closed subset, defining a reduced subscheme $X_0\subset X^\prime_0$, with $X_0\in\Sch^\fp_{/\F_p}$. Moreover, there is a map $X\to (X_0)_\perf$, which is a universal homeomorphism, and therefore an isomorphism by Lemma \ref{UnivHomeomIsomPerf}.

In general, using \cite[Theorem C.9]{ThomasonKtheory} again, we can write $Y$ as a limit of $Y_{i,0}\in \Sch^\fp_{/\F_p}$ with affine transition maps. Letting $Y_i$ be the perfection of $Y_{i,0}$, we see that we can write $Y$ as a limit of $Y_i\in \Perf^\fp_{/\F_p}$. Using Proposition \ref{ApproxFP}, we can assume that $Y=Y_i$ is the perfection of some $Y_0 = Y_{i,0}\in \Sch^\fp_{/\F_p}$. Therefore, $X$ is also perfectly finitely presented over $\F_p$. By the case already handled, $X$ is the perfection of some $X_0^\prime\in \Sch^\fp_{/\F_p}$. Now as $Y_0$ is finitely presented (as a scheme), the composite map $X=(X_0^\prime)_\perf\to Y\to Y_0$ factors through a map $X_0^\prime\to Y_0$, up to a power of Frobenius, which we can forget. Taking $X_0=X_0^\prime\times_{Y_0} Y$ gives the desired model.
\end{proof}

\begin{definition} Let $f: X\to Y$ be a perfectly finitely presented map in $\Perf$. Then $f$ is called proper if it is separated and universally closed.
\end{definition}

\begin{corollary} Let $f: X\to Y$ be a perfectly finitely presented map in $\Perf$, with a model $f_0: X_0\to Y$. Then $f$ is proper if and only if $f_0$ is proper.
\end{corollary}

\begin{proof} This follows from Lemma \ref{SchVsPerfProp}.
\end{proof}

The following somewhat miraculous vanishing result for $\Tor$-groups is responsible for many of the remarkable properties of the $v$-topology on perfect schemes.

\begin{lemma}\label{NoTorPerf} Let $C\leftarrow A\to B$ be a diagram of perfect rings. Then
\[
\Tor_i^A(B,C)=0
\]
for all $i>0$.
\end{lemma}

In the language of derived algebraic geometry, this says that the simplicial $\F_p$-algebra $B\buildrel\L\over\otimes_A C$ is discrete. But $B\buildrel\L\over\otimes_A C$ is still perfect, and one can show that {\em any} perfect simplicial $\F_p$-algebra is discrete, see Proposition \ref{prop:FrobSCR}. We give a more direct proof below.

\begin{proof} We may write $A\to B$ as the composition of the perfection of a free algebra and a quotient map. Clearly, $\Tor$-terms vanish for the free algebra, so we may assume that $A\to B$ is surjective. Let $I=\ker(A\to B)$. By a filtered colimit argument, we may assume that $I=(f_1^{\frac{1}{p^\infty}},\ldots,f_n^{\frac{1}{p^\infty}})$ is perfectly finitely generated. By induction, we may assume that $n=1$, so $I=f^{\frac{1}{p^\infty}} A$. We claim that in this case
\[
I= \varinjlim_{f^{\frac{1}{p^n} - \frac{1}{p^{n+1}}}} A
\]
is the direct limit of $A$ with transition maps given by multiplication by $f^{\frac{1}{p^n} - \frac{1}{p^{n+1}}}$. Here, the map
\[
\gamma: \varinjlim_{f^{\frac{1}{p^n} - \frac{1}{p^{n+1}}}} A\to I
\]
is given by $A\to I$, $a\mapsto f^{\frac{1}{p^n}} a$, in the $n$-spot. Clearly, $\gamma$ is surjective. To check that $\gamma$ is injective, assume that $a\in A$ with $f^{\frac{1}{p^n}} a=0$. By perfectness of $A$, we get $f^{\frac{1}{p^{n+1}}} a^{\frac{1}{p}} = 0$, so in particular $f^{\frac{1}{p^{n+1}}} a=0$. But then $a$ is killed under the transition map, which is multiplication by $f^{\frac{1}{p^n} - \frac{1}{p^{n+1}}}$.

Applying the same argument for $IC = f^{\frac{1}{p^\infty}} C\subset C$ shows that
\[
IC = \varinjlim_{f^{\frac{1}{p^n} - \frac{1}{p^{n+1}}}} C = I\buildrel\L\over\otimes_A C\ .
\]
Thus, $B\buildrel\L\over\otimes_A C$, which is the cone of $I\buildrel\L\over\otimes_A C\to C$, is equal to the cokernel of $IC\to C$, i.e. $B\otimes_A C$, showing that all higher $\Tor$-terms vanish.
\end{proof}

\begin{remark}
	The proof of Lemma \ref{NoTorPerf} shows, in particular, that any quotient $R/I$ of a perfect ring $R$ by the radical $I$ of a finitely generated ideal $I_0$ has finite $\Tor$-dimension as an $R$-module. In Proposition \ref{prop:BoundedTorDim}, this statement will be extended to non-reduced quotients under a mild finite presentation constraint.
\end{remark}

The previous lemma implies the promised base-change result:

\begin{lemma}\label{BaseChange} Let
\[\xymatrix{
X^\prime\ar[r]^{g^\prime}\ar[d]^{f^\prime} & X\ar[d]_f\\
Y^\prime\ar[r]_g & Y
}\]
be a pullback diagram in $\Perf$. Then, for any $K\in D_{qc}(X)$, the base-change morphism
\[
Lg^\ast Rf_\ast K\to Rf^\prime_\ast Lg^{\prime\ast} K
\]
is a quasi-isomorphism.
\end{lemma}

Here as usual $D_{qc}(X)$ denotes the full subcategory of the derived category of $\calO_X$-modules consisting of those complexes of $\calO_X$-modules whose cohomology sheaves are quasicoherent.

\begin{proof} This reduces immediately to the case where all schemes are affine; let $X=\Spec(A)$, $Y=\Spec(B)$ etc., and let $K\in D_{qc}(X) = D_{qc}(A)$. In that case, one has to prove that
\[
K\buildrel\L\over\otimes_A A^\prime\cong K\buildrel\L\over\otimes_B B^\prime\ ,
\]
for which it is enough to see that
\[
B\buildrel\L\over\otimes_A A^\prime = B^\prime\ .
\]
But this is the statement of Lemma \ref{NoTorPerf}.
\end{proof}

\section{$h$-descent for vector bundles on perfect schemes} \label{VDescent}

Our goal is to prove $v$-descent for vector bundles on perfect schemes, as well as certain related bundles defined using the Witt vectors. So, for any perfect scheme $X$ and $n \geq 1$, we write $W_n(X)$ for the scheme obtained by applying the Witt vector functor $W_n(-)$ locally on $X$; let $W(X) := \colim W_n(X)$ be the corresponding $p$-adic formal scheme. Let $\gVect(X)$ denote the groupoid of vector bundles on $X$, and let $\gPic(X)$ be the groupoid of line bundles.

\begin{theorem}\label{thm:VectWhStack}
\begin{enumerate}
\item[{\rm (i)}] If $X=\Spec(A)\in \Perf$ is an affine scheme and $\calE\in \gVect(W_n(X))$ (resp. $\calE\in \gVect(W(X))$) corresponding to some finite projective $W_n(A)$-module (resp. $W(A)$-module) $M$, then
\[
M=R\Gamma_v(X,\calE)\ .
\]
\item[{\rm (ii)}] For each $n$, the functor $X \mapsto \gVect(W_n(X))$ on $\Perf$ is a $v$-stack and, consequently, the functor $X \mapsto \gVect(W(X))$ is a $v$-stack.
\end{enumerate}
\end{theorem}

\begin{remark} This theorem implies in particular that the $v$-topology on $\Perf$ is subcanonical. Given the vast generality of $v$-covers, one might be tempted to believe that the $v$-topology is the canonical topology on $\Perf$; however, Lemma \ref{BreakValRing} provides additional covers in the canonical topology that are not $v$-covers. It may be an interesting problem to describe the canonical topology on $\Perf$.
\end{remark}

\begin{remark}
	\label{rmk:NonDescentWarning}
	We warn the reader that Theorem \ref{thm:VectWhStack} does {\em not} imply descent for standard properties of morphisms in the $h$-topology. For example, the inclusion $j: U := (\A^2 - \{0\})_{\perf} \hookrightarrow X := \A^2_{\perf}$ is not affine, but becomes affine after base change to the blowup $\pi:Y := (\mathrm{Bl}_{0}(X))_\perf \to X$ of $X$ at $0$. The main issue, say in comparison with the fpqc topology, is that pullback along $\pi^*$ is not exact. For the same reason, $h$-descent data for flat quasi-coherent sheaves need not be effective: the pullback of the complex $R j_* \calO_U$ along the blowup $\pi$ is a flat quasi-coherent sheaf on $Y$ equipped with descent data along $\pi$ that does not arise as the pullback of a flat quasi-coherent sheaf on $X$.
\end{remark}

Part (i) of this theorem (for the $h$-topology) is a result of Gabber, cf. \cite[\S 3]{BhattSchwedeTakagi}. Note that the analog of this result without perfections is manifestly false: there is no descent for vector bundles along the inclusion of the reduced subscheme of a non-reduced subscheme. By passing to $\Hom$-bundles, we obtain the following.

\begin{corollary}\label{cor:HomVectWhStack}
Fix some perfect qcqs scheme $X$, and $\calE_1,\calE_2 \in \gVect(W(X))$. Then the functor $(f:Y \to X) \mapsto \Hom(f^* \calE_1, f^* \calE_2)$ is a $v$-sheaf on $\Perf_{/X}$.
\end{corollary}

Here we write $f^* \calE_i \in \gVect(W(Y))$ for the pullback of $\calE_i$ under the map $W(Y) \to W(X)$ induced by $f$ via functoriality. Our strategy for proving Theorem \ref{thm:VectWhStack} is to apply Corollary \ref{SheafUnivSubtr} to the presheaves on $\Sch_{/\F_p}$ obtained from the presheaves on $\Perf$ via $X\mapsto X_\perf$. In this respect, we observe the following.

\begin{proposition}\label{SchVsPerf} Let $F$ be a presheaf of spaces on $\Perf$. Let $F^\prime$ be the presheaf of spaces on $\Sch_{/\F_p}$ defined by $F^\prime(X) = F(X_\perf)$. Then $F$ is a $v$-sheaf if and only if $F^\prime$ is a $v$-sheaf.
\end{proposition}

\begin{proof} If $F^\prime$ is a sheaf, then for any $v$-cover $f:X\to Y$ between $X,Y\in \Perf\subset \Sch_{/\F_p}$, one has Cech descent for $X\to Y$ in $\Sch_{/\F_p}$, saying that
\[
F^\prime(Y) \to \lim F^\prime(X^{\bullet/Y})
\]
is a weak equivalence. But the inclusion $\Perf\subset \Sch_{/\F_p}$ preserves fibre products, so each $X^{n/Y}$ is perfect, and $F^\prime(Y) = F(Y)$ and $F^\prime(X^{n/Y}) = F(X^{n/Y})$, so in particular
\[
F(Y)\to \lim F(X^{\bullet/Y})
\]
is a weak equivalence, proving that $F$ is a sheaf.

Conversely, assume $F$ is a sheaf and let $f: X\to Y$ be a $v$-cover in $\Sch_{/\F_p}$. Then $f_\perf: X_\perf\to Y_\perf$ is a $v$-cover, and Cech descent for $X_\perf\to Y_\perf$ says that
\[
F(Y_\perf)\to \lim F(X_\perf^{\bullet/Y_\perf})
\]
is a weak equivalence. But $X_\perf^{n/Y_\perf} = (X^{n/Y})_\perf$, so this translates into a weak equivalence
\[
F^\prime(Y)\to \lim F^\prime(X^{\bullet/Y})\ ,
\]
showing that $F^\prime$ is a sheaf.
\end{proof}

The following special case is the heart of the proof.

\begin{lemma}\label{lem:VectcdhStack}
Let $f:X \to Y$ be a proper map of noetherian $\F_p$-schemes that is an isomorphism outside some closed subset $Z\subset Y$ with preimage $E\subset X$. 
\begin{enumerate}
\item[{\rm (i)}] If $\calE\in \gVect(Y_\perf)$, then the triangle
\[
R\Gamma_{\fppf}(Y,\calE)\to R\Gamma_{\fppf}(X,\calE)\oplus R\Gamma_{\fppf}(Z,\calE)\to R\Gamma_{\fppf}(E,\calE)
\]
is distinguished.
\item[{\rm (ii)}] Pullback gives an equivalence
\[ \gVect(Y_\perf) \simeq \gVect(X_\perf) \times_{\gVect(E_\perf)} \gVect(Z_\perf)\]
of groupoids.
\end{enumerate}
\end{lemma}

In other words, part (ii) of the lemma asserts that specifying a vector bundle on $Y_\perf$ is equivalent to specifying its pullbacks on $X_\perf$ and $Z_\perf$, together with an identification of the two over $E_\perf$: no higher order isomorphisms, infinitesimal extensions, or coherence data need be specified.

\begin{proof} By faithful flatness of completions and Zariski descent, we may assume $Y = \Spec(A)$, and $Z = \Spec(A/I)$, where $A$ is a noetherian ring, and $I \subset A$ is an ideal such that $A$ is $I$-adically complete.  By abuse of notation, we will also write $I \subset \calO_X$ for the pullback of $I \subset A$ as an ideal sheaf. For part (i), fix some $\calE \in \gVect(Y_\perf)$ corresponding to some finite projective $A_\perf$-module $M$. Writing $M$ as a direct summand of a free module reduces us to the case that $\calE=\calO_{X_\perf}$ is trivial. In that case, the result is \cite[Lemma 3.9]{BhattSchwedeTakagi}.

By passing to $\Hom$-bundles, fully faithfulness of the functor
\[\gVect(Y_\perf) \to \gVect(X_\perf) \times_{\gVect(E_\perf)} \gVect(Z_\perf)\]
follows from ($H^0$ of) part (i).

For essential surjectivity, we must check: given 
\[ \Big(\calE \in \gVect(X_\perf), \calF \in \gVect(Z_\perf), \phi:\calE|_{E_\perf} \simeq f^* \calF\Big) \in \gVect(X_\perf) \times_{\gVect(E_\perf)} \gVect(Z_\perf),\]
there is a unique $\calG \in \gVect(Y_\perf)$ inducing this data. We need some preliminary reductions first.

First, we reduce to the case where $X$ is the blowup of $Y$ along $Z$. By Raynaud-Gruson \cite[Corollary 5.7.12]{RaynaudGruson} and a $2$-out-of-$3$ property for fibre squares, we may assume that $X$ is a blowup of $Y$ along some closed subscheme $Z' \subset Y$ that does not meet $Y-Z$. After replacing $Z$ by a larger infinitesimal neighbourhood if necessary (as this does not change the perfection), we may  assume $Z' \subset Z$ as schemes. Let $E'$ denote the preimage of $Z'$ in $X$. By the full faithfulness we have already shown, it is enough to show that the induced triple $(\calE,\calF|_{Z'_\perf},\phi|_{E'_\perf})$ comes from $\calG \in \gVect(Y_\perf)$. Thus, after replacing $Z$ with $Z'$ and $E$ with $E'$, we may assume that $X$ is the blowup of $Y$ along $Z$. In particular, $E \subset X$ is an effective Cartier divisor whose ideal sheaf is exactly $I \subset \calO_X$, and this ideal sheaf is relatively ample for the map $X \to Y$ by the construction of blowups. In particular, as the base is affine, $I/I^2$ is an ample line bundle on $E$. By Serre vanishing, there exists some $n$ such that $H^i(X, I^k/I^{k+1}) = 0$ for $i > 0$ and $k \geq n$.

We can now begin the proof.  By approximation, the triple $(\calE,\calF,\phi)$ arises from a similar triple $(\calE_0,\calF_0,\phi_0) \in \gVect(X) \times_{\gVect(E)} \gVect(Z)$, at least after Frobenius twisting. Now $\gVect(Y) \to \gVect(Z)$ is bijective on isomorphism classes by affineness of $Y$ and deformation theory. Thus, there is a unique $\calG_0 \in \gVect(Y)$ lifting $\calF_0$. In particular, using $\phi_0$, we have a natural identification $\psi_0:f^*(\calG_0)|_E \simeq \calE_0|_E$. Write $\calG \in \gVect(Y_\perf)$ for its pullback to the perfection. Then $\psi_0$ induces an isomorphism $\psi:f^*(\calG)|_{E_\perf} \simeq \calE|_{E_\perf}$. It is enough to check that the latter lifts to an identification $f^* \calG \simeq \calE$. We will check that $\psi_0$ itself admits such an extension to $X$, at least after replacing all objects by their pullbacks along a fixed (finite) power of Frobenius.

After replacing $\psi_0$ by a large enough Frobenius pullback (depending on $n$), we may assume: there exists an isomorphism $\psi_n:f^*(\calG_0)|_{nE} \simeq \calE_0|_{nE}$ lifting $\psi_0$; here we write $mE \subset X$ for the scheme defined by $I^{m+1}$. The formal existence theorem shows that $\gVect(X) \simeq \lim \gVect(mE)$, so it is enough to show that $\psi_n$ deforms compatible to $mE$ for $m \geq n$. The obstruction to extending across $nE \subset (n+1)E$ lies in $H^1(X, I^n/I^{n+1} \otimes \mathcal{H}\mathrm{om}(f^* \calG_0, \calE_0))$. As $I^n/I^{n+1}$ is an $\calO_E$-module, this simplifies to showing $H^1(X,I^n/I^{n+1}) = 0$; here we use the projection formula, as well as $\psi_0$ to identify $\mathcal{H}\mathrm{om}(f^* \calG_0, \calE_0)|_E \simeq f^* \big(\calEnd(\calG_0)|_Z\big)$. The vanishing now follows from the assumption on $n$, so we have such an extension. Inductively, one extends $\psi_n$ compatibly to all $mE$ for $m \geq n$, proving the theorem.
\end{proof}

\begin{proof}[Proof of Theorem \ref{thm:VectWhStack}]
Part (i) reduces by induction and the $5$-lemma to the case $n=1$. In that case, after the translation given by Proposition \ref{SchVsPerf}, we verify the conditions of Corollary \ref{SheafUnivSubtr}, cf. Corollary \ref{HCohom}. Condition (i) of Proposition \ref{SchVsPerf} follows from Lemma \ref{SchVsPerfProp} and faithfully flat descent. Condition (ii) follows from Lemma \ref{lem:VectcdhStack} (i), while condition (iii) is clear.

Write $\calF_n$ for the prestack $X \mapsto \gVect(W_n(X_\perf))$, and $\calF_\infty = \lim \calF_n$. We will show that $\calF_n$ is a $v$-stack for each $n$ by induction on $n$; this formally implies $\calF_\infty$ is a $v$-stack by passage to limits. For $n = 1$, the result follows from Corollary \ref{SheafUnivSubtr} using Lemma \ref{SchVsPerfProp} and Lemma \ref{lem:VectcdhStack} (ii). Assume inductively that $\calF_{n-1}$ is a $v$-stack. It is enough to prove that if $f: X\to Y$ is a $v$-cover of $Y=\Spec(A)\in \Perf$, and $\calE\in \gVect(W_n(X))$ is a vector bundle equipped with descent data to $Y$, then $\calE$ descends to $Y$. Now $\calE$ defines a $v$-sheaf on $\Perf_{/X}$ which using the descent datum descends to a $v$-sheaf $\calE_Y$ of $W_n(\calO)$-modules on $\Perf_{/Y}$. It is filtered as
\[
0\to \calE_Y/p\buildrel p^{n-1}\over \longrightarrow \calE_Y\to \calE_Y/p^{n-1}\to 0\ .
\]
By induction, both $\calE_Y/p$ and $\calE_Y/p^{n-1}$ are the sheaves associated with a finite projective $A$-module $M_1$, resp. a finite projective $W_{n-1}(A)$-module $M_{n-1}$, with $M_1 = M_{n-1}/p$. By part (i) of the theorem, we have
\[
R\Gamma_v(Y,\calE_Y/p) = M_1\ ,\ R\Gamma_v(Y,\calE_Y/p^{n-1}) = M_{n-1}\ .
\]
But then $M=R\Gamma_v(Y,\calE_Y)$ is an extension of $M_{n-1}$ by $M_1$ which is flat over $\Z/p^n\Z$; any such extension is easily seen to be a finite projective $W_n(A)$-module $M$. One then verifies that $\calE_Y$ is the sheaf associated with $M\in \gVect(W_n(Y))$.
\end{proof}

\section{Construction of line bundles: $K$-theoretic approach}
\label{sec:DetConsKtheory}

Our goal is to attach ``determinant'' line bundles to certain complexes of sheaves that are not linear over the structure sheaf in the setting of perfect schemes. The main technical tool will be the $h$-descent (or, rather, $v$-descent) result proven earlier. For the construction, we first recall the following notion (cf. \S \ref{Appendix}).

\begin{construction}
\label{cons:GradedPic}
	For any scheme $X$, let $\gPic^\Z(X)$ be the groupoid of {\em graded} line bundles on $X$, i.e., an object is given by a pair $(L,f)$ where $L$ is a line bundle on $X$, and $f:X \to \Z$ is a locally constant function; the set $\Isom( (L,f), (M,g))$ is empty if $f \neq g$, and given by $\Isom(L,M)$ otherwise. This groupoid is endowed with a symmetric monoidal structure $\otimes$ where $(L,f) \otimes (M,g) := (L \otimes M, f + g)$, and the commutativity constraint 
	\[ (L \otimes M, f+g) =: (L,f) \otimes (M,g) \simeq (M,g) \otimes (L,f) := (M \otimes L, g+f)\]
	determined by the rule 
	\[ \ell \otimes m \mapsto  (-1)^{f \cdot g} m \otimes \ell.\]
	The endows $\gPic^{\Z}(X)$ with the structure of a (not strictly commutative) Picard groupoid such that 
	\[ \pi_0(\gPic^\Z(X)) = \Pic(X) \times H^0(X_\et,\Z) \quad \mathrm{and} \quad \pi_1(\gPic^\Z(X)) = \calO(X)^\times.\] 
	We view the association $X \mapsto \gPic^{\Z}(X)$ as a sheaf of connective spectra for the \'etale topology. We also write $\gPic(X)$ for the usual (strictly commutative) Picard groupoid of line bundles, so $L \mapsto (L,0)$ establishes a fully faithful embedding $\gPic(X) \subset \gPic^\Z(X)$ compatible with the symmetric monoidal structures. Likewise, the association $(L,f) \mapsto f$ gives a symmetric monoidal functor $\gPic^\Z(X) \to H^0(X_\et,\Z)$. This data fits together into a fibre sequence
	\[ \gPic(X) \to \gPic^{\Z}(X) \to H^0(X_\et,\Z),\]
	of spectra. As $X$ varies, this gives a fibre sequence of sheaves of connective spectra in the \'etale topology. The projection $(L,f) \mapsto L$ gives a (non-symmetric!) monoidal functor $\eta:\gPic^\Z(X) \to \gPic(X)$, and thus splits the sequence as spaces (in fact, as $E_1$-spaces). This allows us to extract honest line bundles from graded ones.
	\end{construction}

The $v$-descent result in the previous section implies:

\begin{proposition}\label{prop:PicZhdescent}
The functor $X \mapsto \gPic^{\Z}(X)$ is a stack in the $v$-topology of $\Perf$.
\end{proposition}

\begin{proof}
We have a fibre sequence 
\[ \gPic(X) \to \gPic^{\Z}(X) \to H^0(X_\et,\Z)\]
of groupoids. It is elementary to check that $X \mapsto H^0(X_\et,\Z)$ is a $v$-sheaf. Using the preceding fibre sequence, the claim now follows from Theorem \ref{thm:VectWhStack} (ii).
\end{proof}

To connect $K$-theory with line bundles, recall the determinant map. Here, for any qcqs scheme $X$, $K(X)$ is defined as by Thomason-Trobaugh, \cite[Definition 3.1]{ThomasonKtheory}, using the $\infty$-category of perfect\footnote{We apologize for the two very different uses of the word ``perfect".} complexes $\Perf(X)$.\footnote{In \cite{ThomasonKtheory}, the language of Waldhausen categories was used. See \cite[Notation 12.11]{Barwick} for a definition of the $K$-theory spectrum in terms of the stable $\infty$-category $\Perf(X)$.} In particular, there is natural map of spaces $\Perf(X)^\simeq\to K(X)$, where $\Perf(X)^\simeq\subset \Perf(X)$ denotes the subcategory with all objects, and only isomorphisms as morphisms. Then $\Perf(X)^\simeq$ is an $\infty$-category (i.e., weak Kan complex) with only invertible morphisms, which is the same thing as a space (i.e., Kan complex). Under the map $\Perf(X)^\simeq\to K(X)$, any perfect complex $C\in \Perf(X)$ defines a point $[C]\in K(X)$. For any distinguished triangle
\[
C^\prime\to C\to C^{\prime\prime}\ ,
\]
one has an identity $[C] = [C^\prime] + [C^{\prime\prime}]$ (up to homotopy) in $K(X)$, under the interpretation of $K(X)$ as a space equipped with an invertible coherently commutative and associative group law, cf. Theorem \ref{EMonVsSpectra}. These are, in some sense, the defining properties of $K(X)$.

\begin{proposition}\label{prop:DetClassical}
There exists a natural functorial map $\det:K(X) \to \gPic^{\Z}(X)$ of connective spectra.
\end{proposition}

Concretely, this means that there is a functor $\Perf(X)^\simeq\to \gPic^\Z(X)$, $C\mapsto \det(C)$ such that for any distinguished triangle as above, $\det(C)\simeq \det(C^\prime)\otimes \det(C^{\prime\prime})$. In this language, this was first discussed by Knudsen-Mumford, \cite{KnudsenMumford}.

\begin{proof} By Zariski descent for $\gPic^\Z$, we need only construct a functorial map for affine schemes $X=\Spec(A)$. But for affine schemes $X=\Spec(A)$, there is a natural equivalence $K(X)\simeq K(A)$, where $K(A)$ is defined as in \S \ref{Appendix}, which comes equipped with $\det: K(A)\to \gPic^\Z(A)$.
\end{proof}

\begin{remark}
The determinant map $K(X) \to \gPic^{\Z}(X)$ does {\em not} factor through $\gPic(X) \subset \gPic^{\Z}(X)$, and is the reason we use $\gPic^{\Z}(X)$.
\end{remark}

Assume now that $X$ is a perfect scheme. Let $\Perf(W(X)\on X)$ be the $\infty$-category of perfect complexes on $W(X)$ which are acyclic after inverting $p$; if $X = \Spec(A)$, then this simply the $\infty$-category of perfect $W(A)$-complexes which are acyclic after base change to $W(A)[\frac{1}{p}]$. Using $v$-descent of line bundles, we will construct a ``determinant'' for objects in $\Perf(W(X)\on X)$, extending the map from Proposition \ref{prop:DetClassical}. To this end, we use $K$-theory spectra, so let $K(W(X)\on X)$ be the $K$-theory spectrum associated with $\Perf(W(X)\on X)$ as in \cite[Definition 3.1]{ThomasonKtheory}. The functor $\Perf(X)\to \Perf(W(X)\on X)$ induces a map of connective spectra $K(X) \to K(W(X)\on X)$. It is known that such maps are equivalences if $X$ and $W(X)$ are regular; this is essentially Quillen's d\'evissage theorem \cite[Theorem 4]{QuillenKtheory}:

\begin{theorem}\label{QuillenDevissage} Let $Y$ be a regular scheme, and $Z\subset Y$ a closed subscheme such that $Z$ is regular. Then the natural map
\[
K(Z)\to K(Y\on Z)
\]
of spaces is a weak equivalence.
\end{theorem}

Concretely, this means that any $C\in \Perf(Y\on Z)$ can be filtered in such a way that all associated gradeds are in $\Perf(Z)$; moreover, the choice of this filtration is essentially irrelevant.

\begin{proof} Let $U=Y\setminus Z$. Then $Y$, $U$ and $Z$ are regular, so $K(Y)=G(Y)$, $K(U)=G(U)$ and $K(Z)=G(Z)$ agree with $G=K^\prime$-theory, cf. \cite[Theorem 3.21]{ThomasonKtheory}. Now the result follows by comparing the localization sequences of \cite[Theorem 5.1]{ThomasonKtheory} and \cite[Proposition 3.2]{QuillenKtheory}.
\end{proof}

A similar result holds true in the perfect case.

\begin{corollary}\label{cor:KRegular}
Assume $X = \Spec(R)$ for the perfection $R$ of a regular $\F_p$-algebra $R_0$. Then $K(X) \simeq K(W(X)\on X)$.
\end{corollary}

\begin{proof} Choose a $p$-adically complete flat $\Z_p$-algebra $A_0$ deforming $R_0$, and a map $\phi:A_0 \to A_0$ lifting Frobenius on $R_0$; this is possible as all obstructions live in (positive degree) coherent cohomology groups on $\Spec(R_0)$, and thus vanish. Let $A_\infty = \colim A_0$, where the colimit is computed along $\phi$, so $A_\infty$ is a flat $\Z_p$-algebra deforming $R$, and $\widehat{A_\infty} \simeq W(R)$: the right hand side is the unique $p$-adically complete $\Z_p$-flat lifting of the perfect ring $R$, and the left hand side provides one such lifting. Let $Y = \Spec(A_\infty)$. 

We claim that $\Perf(Y\on X) \to \Perf(W(X)\on X)$ is an equivalence. Unwinding definitions, we must check that base change along the natural map $A_\infty \to \widehat{A_\infty}$ identifies the $\infty$-categories of perfect complexes on either ring which are acyclic after inverting $p$. This is a standard argument found in derived analogues of the Beauville-Laszlo theorem (see \cite[Lemma 5.12]{BhattAlgTD} for example), and we sketch a proof here for convenience. It suffices to check that $K \simeq K \otimes_{A_\infty} \widehat{A_\infty}$ where $K$ is an $A_\infty$-perfect complex (resp. an $\widehat{A_\infty}$-perfect complex) with $K[\frac{1}{p}] = 0$. But observe that any such $K$ admits an $A_\infty/p^n$-structure (resp. an $\widehat{A_\infty}/p^n$-structure): the colimit of the system
\[ K \xrightarrow{p} K \xrightarrow{p} K \xrightarrow{p} K \to ...\]
is $0$, so multiplication by $p^n$ on $K$ must be $0$ for $n \gg 0$ by the compactness of $K$. The claim now follows from the observation that $A_\infty/p^n \simeq \widehat{A_\infty}/p^n$, and that the same is true in the derived sense since both $A_\infty$ and $\widehat{A_\infty}$ are $p$-torsionfree. 

The equivalence from the previous paragraph gives $K(W(X)\on X) \simeq K(Y\on X) \simeq \colim K(Y_0\on X_0)$, where $X_0 = \Spec(R_0)$, and $Y_0 = \Spec(A_0)$, and the last isomorphism comes from \cite[Proposition 3.20]{ThomasonKtheory}. We also have a compatible description $K(X) \simeq \colim K(X_0)$ for the $K$-theory of $X$. Now the result follows from the equivalence $K(X_0)\simeq K(Y_0\on X_0)$ of Theorem \ref{QuillenDevissage}.
\end{proof}

Combining the previous corollary with $v$-descent gives the promised extension of the determinant map:

\begin{theorem}\label{thm:DetGeneral}
There is a natural functorial (in $X\in \Perf$) map $\widetilde{\det}:K(W(X)\on X) \to \gPic^{\Z}(X)$ of connective spectra extending the determinant map $K(X) \to \gPic^{\Z}(X)$; it is unique up to a contractible space of choices.
\end{theorem}

Intuitively, this map is constructed as follows. If $X=\Spec(A)$, where $A$ is the perfection of a regular $\F_p$-algebra, then any $C\in \Perf(W(X)\on X)$ can be filtered in a way that all associated gradeds $C_i$ lie in $\Perf(X)$. Then, one defines
\[
\widetilde{\det}(C) = \bigotimes_i \det(C_i)\ .
\]
Here, Corollary \ref{cor:KRegular} ensures that this expression is well-defined up to unique isomorphism. The general case follows by $v$-descent, using $v$-descent of line bundles and de Jong's alterations. However, to do the descent, one needs to remember higher homotopies, which is the main reason that we have to work with the $\infty$-category of (connective) spectra, and cannot work with its homotopy category.\footnote{One could, however, truncate all spectra in degrees $>1$, i.e. apply $\tau_{\leq 1}$, and work in the $2$-category of groupoids.}

\begin{proof} Consider the maps of presheaves of groupoids on $\Perf$,
\[
\tau_{\leq 1} K(W(X)\on X)\buildrel f\over\longleftarrow \tau_{\leq 1} K(X)\buildrel g\over\longrightarrow \gPic^\Z(X)\ .
\]
Passing to the $v$-sheafification (in fact, $h$-sheafification is enough), both $f$ and $g$ become equivalences. Indeed, it is enough to check this on perfections of finitely presented schemes. For $f$, this follows from de Jong's alterations, \cite{deJongAlterations}, and Corollary \ref{cor:KRegular}. For $g$, this follows from Proposition \ref{prop:Pic1TruncationofK}. Thus, $\gPic^\Z$ agrees with the $v$-sheafification of $\tau_{\leq 1} K(W(X)\on X)$. But there is a natural map from $K(W(X)\on X)$ to its $1$-truncation, and then to the sheafification.
\end{proof}

\section{Construction of line bundles: Geometric approach}
\label{sec:DetConsGeom}

In this section, we record some geometric observations, which can also be used to construct line bundles\footnote{In fact, even if the line bundle is constructed using the $K$-theoretic approach, some of the geometric lemmas of this section are used.}. We begin by giving a criterion for pullback of line bundles along a proper map to be lossless:

\begin{proposition}\label{PullbackFullyFaithful} Let $f: X\to Y$ be a proper surjective perfectly finitely presented map in $\Perf$. Assume that all geometric fibres of $f$ are connected. Then the pullback functor $\gVect(Y)\to \gVect(X)$ is fully faithful.
\end{proposition}

A similar result holds true for $\gVect(W_n(X))$ and $\gVect(W(X))$.

\begin{proof} For full faithfulness, by passing to $\Hom$-bundles, it is enough to prove that for $\calE\in \gVect(X)$, the adjunction map
\[
\calE\to f_\ast f^\ast \calE
\]
is an isomorphism. This is a local statement, so we can assume that $\calE = \calO_Y$ is trivial, and $Y$ is affine. Choose a model $f_0: X_0\to Y$ of $X$. We need to see that the map of $\calO_Y$-algebras
\[
\calO_Y\to f_\ast \calO_{X_0}
\]
is an isomorphism after perfection. But $A\to H^0(Y,f_\ast \calO_{X_0})$ is a universal homeomorphism (on spectra), thus an isomorphism after perfection by Lemma \ref{UnivHomeomIsomPerf}.
\end{proof}

Next, we want to give criteria for when a vector bundle descends along a proper map. The following lemma breaks arbitrary rings into valuation rings $v$-locally, and will help simplify the base of the morphism:

\begin{lemma}\label{wLoc} 
Let $X$ be a qcqs scheme. Then there is a $v$-cover $\Spec(A) \to X$ such that:
\begin{enumerate}
\item Each connected component of $\Spec(A)$ is the spectrum of a valuation ring.
\item The subset of closed points in $\Spec(A)$ is closed.
\end{enumerate}
\end{lemma}

In particular, $A$ is w-local in the sense of \cite[\S 2]{BhattScholze}.

\begin{proof} 
We may assume that $X=\Spec(S)$ is affine. Pick a set $V$ of representatives for all equivalence classes of valuations, and a map $\Spec(R_i)\to \Spec(S)$ from a valuation ring $R_i$ for each $i\in V$, realizing this valuation. Let $A=\prod_{i\in V} R_i$. Clearly, $\Spec(A)\to\Spec(S)$ is a $v$-cover. We first check (2). For this, note that the formation of Jacobson radicals commutes with products of rings. Hence, the Jacobson radical of $A$ is given by $\fram := \prod_{i \in V} \fram_i$, where $\fram_i \subset R_i$ is the maximal ideal. If we set $k_i := R_i/\fram_i$ to be the corresponding residue field, then $A/\fram \simeq \prod_i k_i =: B$. The closed immersion $\Spec(B) \subset \Spec(A)$ induces a homeomorphism on $\pi_0$: for this, it is enough to note that any idempotent of $B$ lifts uniquely to an idempotent of $A$ (as can be checked in each factor).  Moreover, the ring $B$ is absolutely flat by \cite[Tag 092G]{StacksProject}, and hence $\Spec(B)$ is Hausdorff by \cite[Tag 092F]{StacksProject}. Now any closed point of $\Spec(A)$ comes from $\Spec(B)$ (as the kernel of $A \to B$ is the Jacobson radical), and every point of $\Spec(B)$ gives a closed point (since $\Spec(B)$ is Hausdorff). Thus, the subset of closed points of $\Spec(A)$ coincides with the closed subset $\Spec(B) \subset \Spec(A)$, giving (2).

To proceed further, let $T$ be the Stone-Cech compactification of the discrete set $V$. Recall that elements of $T$ are ultrafilters on $V$, i.e. collections $\calP$ of subsets $U\subset V$ satisfying: (a) stability under finite intersections, (b) if $U\in \calP$ and $U\subset U^\prime$, then $U^\prime\in \calP$, and (c) for each $W\subset V$, exactly one of $W$ and $V\setminus W$ lies in $\calP$. For each subset $W\subset V$, there is a clopen decomposition of $T = T_W\sqcup T_{V\setminus W}$ according to whether the ultrafilter $\calP\in T$ contains $W$ or $V\setminus W$. It is a classical fact that $\Spec(B) \simeq T$ (see \cite[\S 3]{KruckmanStoneCech} for example); explicitly, an ultrafilter $\calP$ corresponds to the prime ideal $\{ (b_i) \in \prod_{i \in V} k_i \mid \{i \in V \mid b_i = 0\} \in \calP \} \subset B$.

Now, as shown above in the proof of (2), the inclusion $\Spec(B) \to \Spec(A)$ identifies $\Spec(B) \simeq \pi_0(\Spec(B)) \simeq \pi_0(\Spec(A))$ (see \cite[\S 2]{BhattScholze} for a discussion of the topology on $\pi_0$). Thus, the canonical map $\Spec(A) \to \pi_0(\Spec(A))$ can be identified with a map $\beta:\Spec(A) \to T$, which can be described explicitly as follows. The preimage of $T_W$ in $\Spec(A)$ is $\Spec(\prod_{i\in W} R_i)$. The fibres $\beta^{-1}(t)=\Spec(A_t)$ are connected components of $\Spec(A)$, and are given by 
\[
A_t = \varinjlim_{W\in \calP} \prod_{i\in W} R_i,
\]
i.e., each $A_t$ is identified with an ultraproduct of the valuation rings $R_i$ (by definition of the ultraproduct), where $t\in T$ corresponds to the ultrafilter $\calP$. Note that the colimit is filtered. For (1), it now suffices to prove that an ultraproduct of valuation rings is a valuation ring. First, $A_t$ is a domain: if $f,g\in A_t$ have product $fg=0$, then $f,g\in \prod_{i\in W} R_i$ for some $W$, and $fg=0\in \prod_{i\in W} R_i$, possibly after shrinking $W$. As each $R_i$ is a domain, $W$ is covered by $\{i\in W|f_i=0\}$ and $\{i\in W|g_i=0\}$. By definition of an ultrafilter, at least one of these sets lies in $\calP$, so $f=0$ or $g=0$ in $A_t$. Now, if $f/g\in \Frac(A_t)$, $f,g\in A_t$, $f,g\neq 0$, then one may again assume $f,g\in \prod_{i\in W} R_i$, and all coordinates of $f,g$ are nonzero. As each $R_i$ is a valuation ring, one of $f_i/g_i$ and $g_i/f_i$ lies in $R_i$ for each $i\in W$. One possibility happens for a set contained in the ultrafilter, showing that one of $f/g$ and $g/f$ lies in $A_t$. Thus, $A_t$ is indeed a valuation ring, proving (1).
\end{proof}

In fact, one can even break up valuation rings, allowing a reduction to valuation rings of rank\footnote{The rank of a valuation ring $V$ is simply the Krull dimension of $V$, and can also be defined purely combinatorially from the value group of $V$. We refer to \cite[Chapter 6]{BourbakiCA} for more on valuation rings. } $1$.

\begin{lemma}\label{BreakValRing} Let $V$ be a perfect valuation ring with valuation $|\cdot|: V\to \Gamma\cup \{0\}$. Let $\alpha: \Gamma\to \Gamma^\prime$ be a map of ordered abelian groups with kernel $\Gamma_0$. Let $V^\prime = V[S^{-1}]$ where $S$ is the set of all elements $f\in V$ with $\alpha(|f|) = 1$; then $V^\prime$ is a valuation ring with a valuation $|\cdot|^\prime: V^\prime\to \Gamma^\prime$. Moreover, let $V\to V_0$ be the quotient of $V$ by the ideal $I$ of all $f\in V$ with $\alpha(|f|)<1$; then $V_0$ is a valuation ring with a valuation $|\cdot|_0: V_0\to \Gamma_0$. Let $V_0^\prime$ be the fraction field of $V_0$.

The sequence
\[
0\to V\to V^\prime\oplus V_0\to V_0^\prime\to 0
\]
is exact, and for any perfect $V$-scheme $X$ and $K\in D_{qc}(X)$, the triangle
\[
R\Gamma(X,K)\to R\Gamma(X\times_{\Spec V} \Spec V^\prime,K)\oplus R\Gamma(X\times_{\Spec V} \Spec V_0,K)\to R\Gamma(X\times_{\Spec V} \Spec V_0^\prime,K)
\]
is distinguished.
\end{lemma}

In other words, one may consider $\Spec V_0\sqcup \Spec V^\prime\to \Spec V$ as a cover for the purposes of quasi-coherent sheaf theory, although it is \emph{not} a $v$-cover. Geometrically, $\Spec V$ is a totally ordered chain of points, $\Spec V_0^\prime\subset \Spec V$ is a point, and $\Spec V_0$ (resp. $\Spec V^\prime$) forms the set of specializations (resp. generalizations) of $\Spec V_0^\prime$ in $\Spec V$.

\begin{proof} The distinguished triangle for general $X$ follows from the exact sequence and Lemma \ref{BaseChange}.

Let us rewrite everything in terms of the fraction field $L$ of $V$, which comes with the valuation $|\cdot|: L\to \Gamma\cup \{0\}$. Then
\[
V=|\cdot|^{-1}(\Gamma_{\leq 1}\cup \{0\})
\]
and
\[
V^\prime = |\cdot|^{-1}(\Gamma_0 \cdot \Gamma_{\leq 1}\cup \{0\})\ .
\]
Moreover, $V_0$ is the quotient of $V$ by
\[
I = |\cdot|^{-1}(\Gamma_{<\Gamma_0}\cup \{0\})
\]
and $V_0^\prime$ is the quotient of $V^\prime$ by $I$; note that $I$ is an ideal of both $V$ and $V^\prime$. Thus, $I=\ker(V\to V_0) = \ker(V^\prime\to V_0^\prime)$, which implies the exactness of
\[
0\to V\to V^\prime\oplus V_0\to V_0^\prime\to 0\ ,
\]
proving the claim.
\end{proof}

The next lemma uses the preceding reductions to give a crucial special case of a fibral criterion for descent of vector bundles under proper maps:

\begin{lemma}\label{VectTrivialValRing} Let $V$ be a perfect valuation ring, let $f: X\to \Spec V$ be a proper perfectly finitely presented map with $R\Gamma(X,\calO_X) = V$; in particular, all geometric fibres of $f$ are connected. Let $\calE$ be a vector bundle on $X$ such that for all geometric points $\bar{y}$ of $\Spec V$, $\calE$ is trivial on $X_{\bar{y}}$. Then $\calE$ is trivial.
\end{lemma}

\begin{remark} The lemma holds more generally, replacing the condition $R\Gamma(X,\calO_X) = V$ by the condition that all geometric fibres are connected (and nonempty), cf.~Remark~\ref{RemVectTrivialnew} below.
\end{remark}

\begin{proof}
We begin with some generalities that help reduce us to the rank $1$ case. Every commutative ring $R$ can be regarded as a filtered colimit $\colim_i R_i$ of its finitely generated $\Z$-subalgebras $R_i \subset R$. When $R$ is a valuation ring, there is an induced valuation on $R_i$, and the corresponding valuation ring $R_i'$ has finite rank (which can be bounded in terms of the transcendence degree of its fraction field over the prime subfield). Moreover, in this case, the inclusion $R_i \subset R$ extends canonically to an extension $R_i' \subset R$ of valuation rings. Thus, we also have $R = \colim_i R_i'$, so  we can write any valuation ring as a filtered colimit of finite rank valuation subrings. Note that the map $\Spec(R) \to \Spec(R_i')$ is surjective: it is flat (as $R$ is a torsionfree $R_i'$-module) and its image contains the closed point, and thus all points as the image of a flat map is stable under generalizations. Applying this to $R=V$, by the usual limit arguments, we may thus assume that the valuation of $V$ is of finite rank. Applying Lemma \ref{BreakValRing}, we may inductively assume that $V$ is of rank $1$. Moreover, we can assume that $V$ is complete, and that the fraction field $K$ of $V$ is algebraically closed, as by Proposition \ref{PullbackFullyFaithful}, it is enough to prove the result over a $v$-cover of $\Spec V$. 

As the fraction field of $V$ is algebraically closed, there is a section $s: \Spec V\to X$, by taking a point of the generic fibre and then taking the closure. We claim that the map
\[
s^\ast: R\Gamma(X,\calE)\to R\Gamma(\Spec V,s^\ast \calE)\cong V^{\rk \calE}
\]
is a quasi-isomorphism. As $\calE$ is trivial on the generic fibre and $R\Gamma(X\times_{\Spec V} \Spec K,\calO)\cong K$, it follows that $s^\ast\otimes_V K$ is a quasi-isomorphism. On the other hand, $\calE$ is trivial over $X\times_{\Spec V} \Spec k$, there $k$ is the residue field of $V$. By finite presentation, $\calE$ is trivial over (the derived scheme) $X\times^{\mathbb{L}}_{\Spec V} \Spec V/g$ for some $g\in \fram=\ker(V\to k)$. As $R\Gamma(X\times^{\mathbb{L}}_{\Spec V} \Spec V/g,\calO)\cong V/g$, one sees that also $s^\ast\buildrel\L\over\otimes_V V/g$ is a quasi-isomorphism. Now the following lemma applied to the cone of $s^\ast$ implies that $s^\ast$ is a quasi-isomorphism.

\begin{lemma} Let $R$ be a ring and $g\in R$ a non-zero divisor. Let $C\in D(R)$ be a complex such that $C\otimes_R R[g^{-1}] = C\buildrel\L\over\otimes_R R/g=0$. Then $C=0$.
\end{lemma}

\begin{remark} For $R=V$ a perfect valuation ring, there exist complexes $0\neq C\in D(V)$ such that $C\otimes_V K = C\buildrel\L\over\otimes_V k=0$, e.g. $C=K/\fram$ (sitting in degree $0$). For this reason, we needed to lift the second condition from $k$ to $V/g$ for some $g\in \fram$ in the proof of Lemma \ref{VectTrivialValRing}. In contrast, if $R$ is a noetherian ring, then any $K \in D(R)$ satisfies: $K \simeq 0$ if and only if $K \otimes^L_R k$ for any residue field $k$ of $R$.
\end{remark}

\begin{proof} As $C\otimes_R R[g^{-1}] = 0$, any cohomology group $H^i(C)$ is $g$-torsion. On the other hand, the long exact cohomology sequence one gets from
\[
0\to R\buildrel g\over\longrightarrow R\to R/g\to 0
\]
by tensoring with $C$ shows that multiplication by $g$ is an isomorphism on $H^i(C)$. Together, these imply $H^i(C)=0$, for all $i$.
\end{proof}

In particular, we see that $R\Gamma(X,\calE)=V^{\rk \calE}$, so $f_\ast \calE$ is a vector bundle on $\Spec V$. Now consider the adjunction map
\[
f^\ast f_\ast \calE\to \calE\ ,
\]
which is a map of vector bundles on $X$. To see whether this is an isomorphism, we can check on fibres, where it follows from the assumption that $\calE$ is trivial on fibres, and the fact that taking $f_\ast \calE = Rf_\ast \calE$ commutes with any base-change by Lemma \ref{BaseChange}. Thus, $\calE\cong f^\ast f_\ast \calE$ is the pullback of the trivial vector bundle $f_\ast \calE$, and therefore trivial.
\end{proof}

Finally, we can give the fibral criterion to descend vector bundles along proper covers.

\begin{theorem}\label{ThmVectTrivial} Let $f: X\to Y$ be a proper perfectly finitely presented map in $\Perf$ such that $Rf_\ast \calO_X = \calO_Y$; in particular, all geometric fibres of $f$ are connected. Let $\calE\in \gVect(X)$. Then $\calE$ descends to $Y$ if and only if for all geometric points $\bar{y}$ of $Y$, $\calE$ is trivial on the fibre $X_{\bar{y}}$.
\end{theorem}

\begin{proof} Clearly, if $\calE\in \gVect(X)$ comes via pullback from $Y$, then it it is trivial on geometric fibres.

For the converse, by $v$-descent for vector bundles and Proposition \ref{PullbackFullyFaithful}, it is enough to prove the result after pullback along some $v$-cover $Y^\prime\to Y$. By Lemma \ref{wLoc}, we may assume that $Y$ is affine, and each connected component is the spectrum of a valuation ring. If $\calE$ is trivial over one connected component, this trivialization spreads to a small open and closed neighborhood. Thus, it is enough to prove that $\calE$ is trivial over each connected component of $Y$, which reduces us to the case that $Y = \Spec(V)$ is the spectrum of a valuation ring. In that case, the result follows from Lemma \ref{VectTrivialValRing}.
\end{proof}

In fact, the condition $Rf_\ast \calO_X = \calO_Y$ can also be checked on geometric fibres:

\begin{lemma}\label{CriterionO} Let $f: X\to Y$ be a proper surjective perfectly finitely presented map in $\Perf$ such that for all geometric points $\bar{y}$ of $Y$ with field of definition $k(\bar{y})$ and fibre $X_{\bar{y}}$, one has $R\Gamma(X_{\bar{y}},\calO)=k(\bar{y})$. Then $Rf_\ast \calO_X = \calO_Y$.
\end{lemma}

We remark that if $f: X\to Y$ arises as the perfection of a proper map $f_0: X_0\to Y_0$ of schemes of finite type over a perfect field $k$, and $X_0$ admits a Frobenius splitting (cf. \cite{MehtaRamanathan}), the conclusion of the lemma implies that $R^i f_{0\ast} \calO_{X_0} = 0$ for $i>0$. Namely, if $X_0$ admits a Frobenius splitting, then $\calO_{X_0}$ is a direct summand of $\calO_X$, and the same is true for its higher direct images. It may be interesting to compare our results on perfect schemes with related results on Frobenius split schemes.

\begin{proof} It is enough to check the assertion locally in the $v$-topology. By Lemma \ref{wLoc}, we may assume that $Y=\Spec(A)$ is the spectrum of a w-local ring all of whose local rings are valuation rings. Then it is enough to check the assertion on stalks, which reduces us to the case that $Y=\Spec(V)$ is the spectrum of a valuation ring. By approximation, we can assume that the valuation on $V$ is of finite rank. Applying Lemma \ref{BreakValRing} inductively, we can assume that $V$ is a rank-$1$-valuation ring, which we can also assume to be complete and algebraically closed. Finally, pick a model $f_0: X_0\to Y=\Spec(V)$ of $f$ as in Proposition \ref{FPModelsExist}.

Consider the complex $K\in D^b(V)$ given as a cone of $V\to R\Gamma(X_0,\calO_{X_0})$, which comes equipped with a Frobenius-linear map $\varphi: K\to K$. Let $M=H^i(K)$ be a cohomology group of $K$. This is a finitely presented $V$-module which comes equipped with a Frobenius-linear map $\varphi: M\to M$. Moreover, by Lemma~\ref{BaseChange} and our hypothesis, $N=\varinjlim_\varphi M$ satisfies $N\otimes_V k(s)=N\otimes_V k(\eta)=0$, where $s,\eta\in \Spec(V)$ are the special and generic point.\footnote{If $i>0$, then $N=H^i(X,\calO_X)$, and there is a short exact sequence $0\to H^i(X,\calO_X)\otimes_V k(s)\to H^i(X_s,\calO_{X_s})\to \Tor_1^V(H^{i+1}(X,\calO_X),k(s))\to 0$, where the middle term is zero; thus, $N\otimes_V k(s)=0$. A similar argument works for $i=0$, where $N=H^0(X,\calO_X)/V$.} The statement for the generic fibre implies that there is some $m$ such that $\varphi^m(M)\subset M_\tors$ is contained in the torsion submodule of $M$. Then the statement for the special fibre implies that there is some $m$ such that $\varphi^m(M)\subset \fram M_\tors$, where $\fram\subset V$ is the maximal ideal. By finite presentation, this implies that there is some $g\in \fram$ with $\varphi^m(M)\subset gM_\tors$. But $M_\tors$ is finitely presented, so there is some $n$ such that $g^n M_\tors = 0$. Then $\varphi^{mn}(M)\subset g^n M_\tors = 0$. This implies that $N=\varinjlim_\varphi M=0$. Thus, $\varinjlim_\varphi K$ is quasi-isomorphic to $0$, which implies that $R\Gamma(X,\calO_X) = V$, as desired.
\end{proof}

In particular, checking isomorphisms can be done pointwise.

\begin{corollary}\label{PointwiseIsom} Let $f: X\to Y$ be a proper perfectly finitely presented map of perfect $\F_p$-schemes such that for all geometric points $\bar{y}$ of $Y$, the fibre $X_{\bar{y}}$ is isomorphic to $\Spec(k(\bar{y}))$. Then $f$ is an isomorphism.
\end{corollary}

\begin{proof} Picking any finite type model $f_0$ of $f$, the assumption implies that $f_0$ is quasifinite, thus finite, and in particular affine; therefore, $f$ itself is affine, so $f$ is determined by $f_\ast \calO_X$. But Lemma \ref{CriterionO} implies $f_\ast \calO_X = \calO_Y$, so that $f$ is an isomorphism.

Alternately, one may argue as follows, as suggested by the referee. As above, we first show that $f$ is the perfection of a finite morphism $f_0$. The fibral assumption shows that $f_0(k)$ is bijective for any algebraically closed field $k$, and hence $f_0$ is universally injective. It is also clear that $f_0$ is universally surjective. As $f_0$ is finite, it follows that $f_0$ is a universal homeomorphism. One then concludes using Lemma~\ref{UnivHomeomIsomPerf}.
\end{proof}

In our application, the geometric fibres will be of the following form.

\begin{lemma}\label{TrivialLineBundle} Let $k$ be a perfect field, and let $Q$ be a finite length $W(k)$-module. Let $X$ be a perfect $k$-scheme which comes equipped with a filtration $\Fil^i \calQ_X\subset \calQ_X=Q\otimes_{W(k)} W(\calO_X)$ whose associated gradeds $\gr^i \calQ_X = \Fil^i \calQ_X / \Fil^{i+1} \calQ_X$ are finite projective $\calO_X$-modules. Then the line bundle
\[
\calL = \bigotimes_i \det_{\calO_X} \gr^i \calQ_X
\]
on $X$ is trivial.
\end{lemma}

Using Theorem \ref{thm:DetGeneral}, this follows from the identity
\[
\calL = \widetilde{\det}(\calQ_X) = \widetilde{\det}(Q)\otimes_k \calO_X\ .
\]
Here, we give an independent, more geometric, proof.

\begin{proof} We may reduce to the universal case where $X$ parametrizes filtrations of $\calQ_X$ with $\gr^i \calQ_X$ finite projective of given rank. Refining further (which gives rise to a map with geometrically connected fibres), we may assume that $X$ parametrizes such filtrations with each $\gr^i \calQ_X$ a line bundle.

Let $n$ be the length of $Q$. Let $X_0=\Spec(k)$, $\calQ_0=Q$ regarded as a sheaf on $X_0$, and inductively let $f_i: X_i\to X_{i-1}$ parametrize locally free quotients $\calG_i$ of rank $1$ of $\calQ_{i-1}/p$, and set $\calQ_i = \ker(f_i^\ast \calQ_{i-1}\to \calG_i)$ over $X_i$. Then $f_i$ is a proper surjective perfectly finitely presented map, and $X=X_n$. Moreover, each geometric fibre of $f_i$ is the perfection of a projective space. Applying Lemma \ref{CriterionO}, this shows that $Rf_{i\ast} \calO_{X_i} = \calO_{X_{i-1}}$.

We claim by descending induction on $i=n,n-1,\ldots,0$ that $\calL$ descends (necessarily uniquely) to a line bundle $\calL_i$ on $X_i$. For $i=n$, this is a tautology, so assume $\calL$ descends to $\calL_{i+1}$ over $X_{i+1}$. By Theorem \ref{ThmVectTrivial}, to check whether $\calL_{i+1}$ descends to $X_i$, it is enough to check on all geometric fibres, so let $\bar{x}\in X_i$ be a geometric point, giving a finite length $W(k(\bar{x}))$-module $Q_i$, equipped with a fixed filtration into $1$-dimensional $k(\bar{x})$-modules. Replacing $Q$ by $Q_i$, we can assume that $i=0$.

Now $\calL_1$ is a line bundle over $X_1$, which is the perfection of a projective space $\P^k$. Let $Y\subset X$ parametrize those filtrations which are refinements of the $p$-adic filtration of $Q$. Then $\calL$ restricted to $Y$ agrees with
\[
\left(\bigotimes_i \det_k (p^i Q/p^{i+1}Q)\right)\otimes_k \calO_Y\ ,
\]
and is therefore trivial. On the other hand, the map $Y\to X_1$ is the perfection of a proper surjective map with geometrically connected fibres: It is a successive (perfection of a) flag variety bundle. But $\calL_1$ becomes trivial over $Y$, and is thus trivial by Proposition \ref{PullbackFullyFaithful}.
\end{proof}

\begin{remark}\label{RemVectTrivialnew} In the period in which this paper was being refereed, we discovered that the fibral conditions in Theorem~\ref{ThmVectTrivial} can be weakened substantially: one only needs geometric connectedness for the fibers. Since this result is likely more applicable, we record it here; it is not used elsewhere in the paper. The proof involves reduction to very big nonarchimedean fields. One may wonder whether the theorem admits a more ``classical'' proof.

\begin{theorem} Let $f: X\to Y$ be a proper surjective perfectly finitely presented map in $\Perf$ such that all geometric fibres of $f$ are connected. Let $\calE\in \gVect(X)$. Then $\calE$ descends to $Y$ if and only if for all geometric points $\bar{y}$ of $Y$, $\calE$ is trivial on the fibre $X_{\bar{y}}$.
\end{theorem}

\begin{proof} As in the proof of Theorem~\ref{ThmVectTrivial}, one reduces to the case that $Y=\Spec(V)$ is the spectrum of a valuation ring, and then, as in the first paragraph of the proof of Lemma~\ref{VectTrivialValRing}, we reduce to $V$ having rank $1$. Let $K$ be the fraction field of $V$, and $\frak m\subset V$ the maximal ideal, and let $\eta$ and $s$ be the generic and special points of $\Spec(V)$ respectively. We begin by explaining a reduction to the case where $f$ is flat.  Choose a model $f_0:X_0 \to \Spec(V)$ of $f_0$, so $f_0$ is a finitely presented proper map with perfection $f$. Let $X_0^\prime \subset X_0$ be the closure of the generic fiber $X_{0,\eta} \subset X_0$. The induced map $g_0:X_0^\prime \to \Spec(V)$ is proper and flat with geometrically connected fibers: $g_{0,\ast} \calO_{X_0^\prime}$ is a torsionfree finite $V$-algebra of rank $1$, and must thus coincide with $V$. The closed immersion $X_0^\prime \to X_0$ is an isomorphism away from the special fibre, so $X$ has a $v$-cover by $X^\prime$ and $X_s$ with overlap $X^\prime_s$; here we drop the subscript `0' to denote passage to the perfection. By Theorem~\ref{thm:VectWhStack}, vector bundles on $X$ are the same as vector bundles on $X^\prime$  and $X_s$ together with an isomorphism over $X^\prime_s$. As the global functions on $X^\prime_s$ are constant by geometric connectedness, it suffices to show triviality of $\calE$ over $X^\prime$ and $X_s$ separately. Since we are assuming triviality over $X_s$, we may thus assume $X = X^\prime$ is flat over $V$.

We are now in the following setup: $V$ is a valuation ring of rank $1$, $f:X \to \Spec(V)$ is a proper perfectly finitely presented flat map with geometrically connected fibers, and $\calE \in \Vect(X)$ is trivial on both fibers $X_\eta$ and $X_s$. We will need to arrange a few extra properties. First, we want that $\calO_X$ is integrally closed in $\calO_{X_\eta}$. For this, we make the extra assumption that the fraction field $K$ of $V$ is algebraically closed, which amounts to a further $v$-cover on the base. Let $f_0: X_0\to \Spec(V)$ be a proper flat and reduced model of $f$. In this situation, by \cite[\S 6.4.1, Corollary 5]{BoschGuentzerRemmert}, the normalization $X_0^\prime$ of $X_0$ in its generic fibre is of finite presentation over $\Spec(V)$. Arguing as in the first paragraph, we can then reduce to the case where $X=X^\prime$ has the property that $\calO_X$ is integrally closed in $\calO_{X_\eta}$.

Let $M=H^0(X,\calE)\subset H^0(X_\eta,\calO_{X_\eta})\cong K^r$. This is bounded (as one can find an injection $\calE\hookrightarrow \calO_X^r$ into a trivial vector bundle, as $\calE_\eta$ is trivial). Moreover, we claim that $M=\Hom_V(\frak m,M)$, where $\frak m\subset V$ is the maximal ideal. This follows from the statement $\calE = \Hom_V(\frak m,\calE)$ on the level of sheaves, which can be reduced to $\calO_X=\Hom_V(\frak m,\calO_X)$. To check this, let $f_0: X_0\to \Spec(V)$ be a proper flat reduced model of $f$ such that $\calO_{X_0}$ is integrally closed in $\calO_{X_{0,\eta}}$. Then $\Hom_V(\frak m,\calO_X) = \varinjlim_\varphi \Hom_V(\frak m,\calO_{X_0})$ as $\calO_{X_0} = \calO_{X_{0,\eta}}\cap \calO_X$ by our assumption that $\calO_{X_0}$ is integrally closed in $\calO_{X_{0,\eta}}$. But $\calO_{X_0}$ is locally a free $V$-module (of infinite rank), by~\cite[Corollaire 3.3.13]{RaynaudGruson}, so $\Hom_V(\frak m,\calO_{X_0}) = \calO_{X_0}$ by reduction to $\Hom_V(\frak m,V)=V$. In conclusion,
\[
\Hom_V(\frak m,\calO_X) = \varinjlim_\varphi \Hom_V(\frak m,\calO_{X_0}) = \varinjlim_\varphi \calO_{X_0} = \calO_X\ ,
\]
as desired.

In fact, as we are allowed to make further $v$-covers, we can assume in addition that $K$ is spherically complete (i.e., any decreasing sequence of closed balls in $V$ has nonempty intersection; equivalently, $\Ext^1_V(\frak m,V) = 0$) and with value group $\mathbb R_{\geq 0}$. We claim that in this situation $M=H^0(X,\calE)$ is a finite free $V$-module.

\begin{lemma} Let $K$ be a complete nonarchimedean field which is spherically complete, and with value group $\mathbb R_{\geq 0}$. Let $V$ be the valuation ring of $K$, $\frak m\subset V$ the maximal ideal, and let $M\subset K^r$ be a sub-$V$-module such that $M$ is bounded and $M=\Hom_V(\frak m,M)$. Then $M$ is a finite free $V$-module.
\end{lemma}

\begin{proof} If $r=1$, then any sub-$V$-module is of the form $\{x\in K\mid |x|\leq a\}$ or $\{x\in K\mid |x|<a\}$ for some $a\in \mathbb R_{\geq 0}\cup \{\infty\}$. The case $a=\infty$ is not allowed as $M$ is bounded, and the case $\{x\in K\mid |x|<a\}$ is excluded by the condition $M=\Hom_V(\frak m,M)$; the right-hand side evaluates to $\{x\in K\mid |x|\leq a\}$ in this case. As any $a\in \mathbb R_{\geq 0}$ is the absolute value of some $x\in K$ by assumption, it follows that $M$ is free of rank $1$ if $a\neq 0$; otherwise $M=0$.

In general, we induct on $r$, so let $M^\prime = M\cap K^{r-1}$ and $M^{\prime\prime} = M/M^\prime$. Then $M^\prime\subset K^{r-1}$ and $M^{\prime\prime}\subset K$ are bounded $V$-submodules. Moreover, $M^\prime=\Hom_V(\frak m,M^\prime)$ as this holds for $M$; thus, by induction, $M^\prime$ is finite free. This implies, as $K$ is spherically complete, that $\Ext^1_V(\frak m,M^\prime) = 0$, so one finds that also $M^{\prime\prime} = \Hom_V(\frak m,M^{\prime\prime})$. Therefore, $M^{\prime\prime}$ is finite free, and so is the extension $M$ of $M^{\prime\prime}$ by $M^\prime$.
\end{proof}

In particular, this applies to show that $M=H^0(X,\calE)$ is a finite free $V$-module. We have a short exact sequence
\[
0\to M\otimes_V k(s)\to H^0(X_s,\calE_s)\to \Tor_1^V(H^1(X,\calE),k(s))\to 0\ .
\]
But $M\otimes_V k(s)$ and $H^0(X_s,\calE_s)$ are $k(s)$-vector spaces of the same dimension as $\calE_s$ is trivial, so $M\otimes_V k(s)\to H^0(X_s,\calE_s)$ is an isomorphism. This implies that the map of vector bundles $f^\ast M\to \calE$ is an isomorphism, as it is an isomorphism in both fibres. Thus, $\calE$ is trivial, as desired.
\end{proof}

\end{remark}

\section{Families of torsion $W(k)$-modules}
\label{sec:torsionfamilies}

In this section, we collect some results on the behaviour of ``families'' of torsion $W(k)$-modules indexed by perfect schemes; some of the discussion overlaps with \cite{ZhuMixedCharGeomSatake}. The following definition will help with the bookkeeping:

\begin{definition} Consider the set of sequences $\lambda := (\lambda_1,\lambda_2,...,\lambda_n,\ldots)$ of non-negative integers such that $\lambda_j \geq \lambda_{j+1}$ for all $j$, and $\lambda_j=0$ for sufficiently large $j$. For another such sequence $\mu$, we say $\mu \leq \lambda$ if $\lambda - \mu$ (computed term wise) is a non-negative linear combination of $\epsilon_j := (0,\ldots,0,1,-1,0,\ldots)$, where the $1$ entry is in the $j$-th spot. 

If $\lambda$ is any such sequence, we write $\lambda-1$ for the sequence with entries $(\lambda-1)_j$ given by $\lambda_j-1$ in case $\lambda_j\geq 1$, and $0$ otherwise.
\end{definition}

The sequence $\lambda$ represents the isomorphism class of a finite torsion module over $W(k)$.

\begin{definition}
Fix a perfect field $k$. A finitely generated $p$-power torsion $W(k)$-module $Q$ is said to have {\em type} $\lambda$ if $Q \simeq \oplus_j W(k)/p^{\lambda_j}$; in that case, write $\lambda=\lambda(Q)$. Let $R$ be a perfect ring, and let $Q$ be a finitely generated $p$-power torsion $W(R)$-module. Then $Q$ has type $\leq \lambda$ if $\lambda(Q \otimes W(k(x)))\leq \lambda$ for all $x\in \Spec(R)$. If $\lambda(Q \otimes W(k(x)))=\lambda$ for each $x \in \Spec(R)$, then $Q$ has type \emph{exactly} $\lambda$.
\end{definition}

The case where $Q$ has type exactly $\lambda$ is very rigid; in particular, it implies that $Q$ is finitely presented.

\begin{lemma}\label{lem:ConstantTypeProjective}
Let $R$ be a perfect ring, and let $Q$ be a finitely generated $W(R)$-module of $p$-power torsion, of type exactly $\lambda$. Then each $p^i Q / p^{i+1} Q$ is a finite projective $R$-module.
\end{lemma}

\begin{proof} Consider first $i = 0$. For any point $x \in \Spec(R)$, we know $Q/p\otimes k(x) \simeq (Q\otimes W(k(x)))/p$. Thus, $Q/p$ is a finitely generated $R$-module such that $Q/p\otimes k(x)$ has the same rank $n$ for all $x\in \Spec(R)$. As $R$ is reduced, this implies that $Q/p$ is finite projective: for any point $x\in \Spec(R)$, after replacing $R$ by some localization around $x$, there are $f_1,\ldots,f_n\in Q/p$ freely generating $Q/p\otimes k(x)$. The cokernel $C$ of the resulting map $R^n\to Q/p$ is finitely generated with $C\otimes k(x)=0$. By Nakayama, $C$ is trivial in a neighborhood. But in this neighborhood, the surjective map $R^n\to Q/p$ has to be an isomorphism at all points (as it is a surjective map of vector spaces of the same dimension), and thus is injective, as $R$ is reduced. In particular, the kernel $pQ$ of $Q\to Q/p$ is still finitely generated, and of type exactly $\lambda-1$, so the claim for $i>0$ follows by induction.
\end{proof}

\begin{remark}\label{rmk:RankGraded}
In Lemma \ref{lem:ConstantTypeProjective}, one can compute the rank of the projective module $p^iQ/p^{i+1}Q$ explicitly: it is the largest $j$ such that $i < \lambda_j$; we denote this number by $n_\lambda(i)$. More pictorially: if one visualizes $\lambda$ as a ``bar graph'' where each bar is made from blocks and the $j$-th bar has $\lambda_j$ blocks, then $n_\lambda(i)$ is the size of the $i$-th row (starting at $0$).
\end{remark}

Using these notations, let us recall the usual characterization of the majorization inequality $\lambda\geq \mu$.

\begin{lemma}\label{CharMajorization} Let $\lambda=(\lambda_1,\lambda_2,\ldots)$ and $\mu=(\mu_1,\mu_2,\ldots)$ be eventually $0$ decreasing sequences of nonnegative integers. Let $Q_\lambda,Q_\mu$ be torsion $W(k)$-modules with $\lambda(Q_\lambda)=\lambda$ and $\lambda(Q_\mu)=\mu$. Then the following conditions are equivalent.
\begin{enumerate}
\item[{\rm (i)}] One has $\lambda\geq \mu$.
\item[{\rm (ii)}] For all $n\geq 1$, the inequality
\[
\sum_{j=1}^n \lambda_j\geq \sum_{j=1}^n \mu_j
\]
holds, and is an equality for sufficiently large $n$.
\item[{\rm (iii)}] For all $m\geq 0$, the inequality
\[
\sum_{i\geq m} n_\lambda(i)\geq \sum_{i\geq m} n_\mu(i)
\]
holds true (noting that both sums are finite), and is an equality for $m=0$.
\item[{\rm (iv)}] The lengths $\lg(Q_\lambda) = \lg(Q_\mu)$ are equal, and for all $m\geq 0$, $\lg(p^m Q_\lambda)\geq \lg(p^m Q_\mu)$.
\end{enumerate}
\end{lemma}

\begin{proof} The equivalence of (i) and (ii) is standard: Clearly (i) implies (ii) as adding a sequence $\epsilon_n=(0,\ldots,0,1,-1,0,\ldots)$ preserves the inequalities. Conversely, for the smallest $n$ where the inequality is (ii) is strict, one can subtract $\epsilon_n+\epsilon_{n+1}+\ldots+\epsilon_{n^\prime}$ from $\lambda$ for suitable $n^\prime$ while preserving all inequalities, and then argue inductively. Moreover, (iii) and (iv) are the same, as
\[
\lg(p^mQ_\lambda) = \sum_{i\geq m} n_\lambda(i)\ .
\]

As
\[
\sum_j \lambda_j = \sum_i n_\lambda(i)\ ,
\]
one sees that the equality part in (ii) and (iii) is equivalent. It remains to show that the inequalities are equivalent. We give the proof that (ii) implies (iii); the converse is identical.\footnote{In fact, $\lambda\mapsto n_\lambda$ is (up to shift $i\mapsto i+1$) the transpose of a partition, so the situation is symmetric.} Thus, take any $m\geq 0$, and let $n$ be maximal such that $\mu_n>m$. Then
\[
\sum_{i\geq m} n_\mu(i) = \sum_{j=1}^n (\mu_j - m)\ .
\]
Let $n^\prime$ be maximal such that $\lambda_{n^\prime}>m$, so that we have a similar equality
\[
\sum_{i\geq m} n_\lambda(i) = \sum_{j=1}^{n^\prime}(\lambda_j - m)\ .
\]
Assume first that $n\leq n^\prime$. Applying (ii) for $n$ shows
\[
\sum_{i\geq m} n_\mu(i) = \sum_{j=1}^n (\mu_j - m)\leq \sum_{j=1}^n (\lambda_j - m)\ ,
\]
which is at most
\[
\sum_{j=1}^{n^\prime}(\lambda_j - m) = \sum_{i\geq m} n_\lambda(i)\ ,
\]
as all further summands for $n<j\leq n^\prime$ are positive. If $n>n^\prime$, then again applying (ii) for $n$ shows that
\[
\sum_{i\geq m} n_\mu(i) = \sum_{j=1}^n (\mu_j - m)\leq \sum_{j=1}^n (\lambda_j - m)\ ,
\]
which is at most
\[
\sum_{j=1}^{n^\prime}(\lambda_j - m) = \sum_{i\geq m} n_\lambda(i)\ ,
\]
as all extra summands $\lambda_j - m$ for $n^\prime<j\leq n$ are nonpositive.
\end{proof}

The basic source of finite torsion $W(k)$-modules in the sequel comes from isogenies:

\begin{definition} For any perfect ring $R$, a map $f: N\to M$ of finite projective $W(R)$-modules is called an {\em isogeny} if it admits an inverse up to multiplication by a power of $p$.
\end{definition}

In particular, any isogeny is injective.

\begin{lemma}\label{PDDim1} A finitely generated $p$-torsion $W(R)$-module $Q$ is of projective dimension $1$ if and only if it can be written as the cokernel of an isogeny.
\end{lemma}

\begin{proof} If $Q$ is the cokernel of the isogeny $f: N\to M$, then
\[
0\to N\to M\to Q\to 0
\]
is a projective resolution of $Q$ as $W(R)$-module, showing that the projective dimension is (at most) $1$; clearly, $Q$ is not projective, so the projective dimension is $1$.

Conversely, if $M$ is of projective dimension $1$, then pick a surjection $M=W(R)^n\to Q$. Its kernel $N$ is projective. By checking at the characteristic $0$ points of $\Spec(W(R))$ (which are dense), one sees that $N$ is of rank $n$, and thus finite projective.
\end{proof}

We will also need the following lemma.

\begin{lemma}\label{ReductionProjectiveDim}
Let $R$ be a perfect ring, and let $Q$ be a finitely generated $R$-module. Then $Q$ is a projective $R$-module if and only if it is of projective dimension $1$ as $W(R)$-module.
\end{lemma}

\begin{proof} Note that $R=W(R)/p$ is of projective dimension $1$ as $W(R)$-module. Thus, the same is true for any free module, and then also for any direct summand of a free module, i.e. any projective module.

Conversely, assume $Q$ is a finitely generated $R$-module that is of projective dimension $1$ as $W(R)$-module. Take any $x\in \Spec(R)$, and let $n=\dim_{k(x)} (Q\otimes k(x))$. Pick a map $M=W(R)^n\to Q$ that is surjective after $\otimes k(x)$. By Nakayama, we may assume that it is surjective, after some localization on $\Spec(R)$. Let $f: N\hookrightarrow M$ be the kernel of $M\to Q$. Then $N$ is a finite projective $W(R)$-module (as $Q$ is of projective dimension $1$), and there is an injection
\[
g: pW(R)^n\hookrightarrow N\ .
\]
On the other hand, after base-change $R\to k(x)$, $g$ is surjective. Applying Nakayama again (noting that $g$ being surjective is equivalent to $g\mod p$ being surjective), we can assume that $g$ is surjective, after a further localization around $x$. Then $g: pW(R)^n\cong N$, and $Q=W(R)^n/pW(R)^n = R^n$ is finite projective.
\end{proof}

Now we want to prove that the locus where the cokernel $Q$ of an isogeny has type $\leq \lambda$ is closed.

\begin{lemma}\label{Semicontinuity} Let $R$ be a perfect ring, and let $\beta: N\to M$ be an isogeny of finite projective $W(R)$-modules with cokernel $Q$. Fix any sequence $\lambda$ of non-negative integers as above. The set
\[
\Spec(R)_{\leq \lambda}\subset \{x\in \Spec(R)\mid \lambda(Q\otimes W(k(x)))\leq \lambda\}
\]
is a closed subset of $\Spec(R)$.
\end{lemma}

The proof uses the Demazure scheme, so we define this first (see also \cite[\S 1.3]{ZhuMixedCharGeomSatake}).

\begin{definition}\label{DefDemazure} Let $X$ be a perfect $\F_p$-scheme, let $\calQ$ be a finitely generated $p$-power torsion quasicoherent $W(\calO_X)$-module of projective dimension $1$, and let $\lambda$ be a sequence of non-negative integers as above. The Demazure scheme is the perfect scheme
\[
\Dem_\lambda(\calQ)\to X
\]
parametrizing decreasing filtrations $\Fil^i \calQ\subset \calQ$ such that $\gr^i \calQ = \Fil^i \calQ/\Fil^{i+1} \calQ$ is a finite projective $\calO_X$-module of rank $n_\lambda(i)$.\footnote{See Remark \ref{rmk:RankGraded} for the definition of $n_\lambda(i)$.}
\end{definition}

We reserve the name Demazure resolution for the ``absolute" construction defined in Definition \ref{DefDemazureRes} below; a Demazure scheme may not be (the perfection of) a smooth scheme.

\begin{proposition}\label{ExDemazure} The Demazure scheme exists, i.e. the moduli problem is representable by a perfect scheme. The map $\Dem_\lambda(\calQ)\to X$ is a proper perfectly finitely presented morphism.
\end{proposition}

In the proof, we use the following convention: if $X$ is a perfect scheme, $\calF$ a finitely presented quasicoherent $\calO_X$-module, and $n \in \Z_{\geq 0}$, then $\Quot(\calF,n)$ denotes the perfection of the $\Quot$-scheme, over $X$, parametrizing locally free quotients $\calF \twoheadrightarrow \calG$ with $\rk(\calG) = n$. If $\calF$ is locally free, this is the perfection of a Grassmannian. In general, $\Quot(\calF,n)\to X$ is a proper perfectly finitely presented morphism.

\begin{proof} Assume first that $\lambda=0$. In that case $\Dem_\lambda(\calQ)\subset X$ is the subset of $x\in X$ such that $Q\otimes W(k(x))=0$, and we claim that this is open and closed (which implies the proposition in this case). The claim is local, so we may assume $X = \Spec(R)$, and $\calQ$ is associated to a finitely presented $W(R)$-module $Q$ that can be written as a cokernel of an isogeny $f: N\to M$ of finite free $W(R)$-modules. Then the locus where $Q\otimes W(k(x))=0$ is (by Nakayama) the locus where $f$ is an isomorphism, i.e. where $\det f\in W(R)$ is invertible. But $\det f\in W(R)[\frac{1}{p}]^\times$, which implies that $x\mapsto v_p(\det (f\otimes W(k(x))))$ is a locally constant function on $\Spec R$; in particular, the locus where $\det f$ is invertible is open and closed.

In general, consider $\Quot(\calQ/p,n_\lambda(0))\to X$. This parametrizes the different choices for $\Fil^1 \calQ$ such that $\calQ/\Fil^1 \calQ$ is a finite projective $\calO_X$-module of rank $n_\lambda(0)$: Over $\Quot(\calQ/p,n_\lambda(0))$, one has the universal cokernel $\calQ/p\to \calG$, and one can define $\Fil^1 \calQ = \ker(\calQ\to \calG)$. This is still of projective dimension $1$, and
\[
\Dem_\lambda(\calQ) = \Dem_{\lambda - 1}(\Fil^1 \calQ)\ ,
\]
where $\lambda - 1=(\lambda_1-1,\ldots,\lambda_n-1,0,\ldots)$ is the sequence of non-negative integers obtained from $\lambda$ by subtracting $1$ from all positive entries. By induction, the proposition follows.
\end{proof}

\begin{remark}
The special case $\lambda = 0$ in the above proof can be presented slightly differently via $K$-theory as follows. In \S \ref{sec:DetConsKtheory}, we defined a canonical map $r:K_0(W(X) \on X) \to H^0(X,\Z)$ by composing the determinant construction with the canonical map $\gPic^\Z(X) \to H^0(X,\Z)$. Now any $\calQ$ as in the proposition defines element of $K_0(W(X) \to X)$. Unwinding definitions shows that $r(\calQ)$ is given by the locally constant function $X \to \Z$ sending $x \in X$ to the length of $\calQ \otimes^L_{W(\calO_X)} W(\kappa(x)) \simeq \calQ \otimes_{W(\calO_X)} W(\kappa(x))$ as a finite torsion $W(\kappa(x))$-module. In particular, the vanishing locus of this function is clopen in $X$. 
\end{remark}

\begin{lemma}\label{ImageDemazure} Let $X$, $\calQ$ and $\lambda$ be as in Definition \ref{DefDemazure}. The image of the morphism $\Dem_\lambda(\calQ)\to X$ is the locus
\[
X_{\leq \lambda} = \{x\in X\mid \lambda(\calQ\otimes W(k(x)))\leq \lambda\}\ .
\]
If $\calQ$ is of type exactly $\lambda$, then $\Dem_\lambda(\calQ)\to X$ is an isomorphism. Moreover, if the fibre over a geometric point $\bar{x}$ of $X$ is nonempty, then
\[
R\Gamma(\Dem_\lambda(\calQ)\times_X \bar{x},\calO) = k(\bar{x})\ ,
\]
where $k(\bar{x})$ is the residue field of $\bar{x}$.
\end{lemma}

\begin{proof} To determine the image, we may assume that $X=\Spec(k)$ is a point; then $\calQ$ determines a torsion $W(k)$-module $Q$. We need to show that the module
\[
Q \cong \bigoplus_{i=1}^n W(k)/p^{\lambda_i(Q)}
\]
admits a decreasing filtration $\Fil^i Q\subset Q$ with gradeds $\gr^i Q\cong k^{n_\lambda(i)}$ if and only if $\lambda(Q)\leq \lambda$. Assume first that $Q$ admits such a filtration. Then $\Fil^1 Q$ admits a filtration of type given by $\lambda-1$, as $n_{\lambda-1}(i) = n_\lambda(i+1)$ for all $i\geq 0$; it follows by induction that $\lambda(\Fil^1 Q)\leq \lambda-1$. But then $Q$ is an extension
\[
0\to \Fil^1 Q\to Q\to k^{n_\lambda(0)}\to 0
\]
from which it follows that
\[
\lambda(Q)\leq \lambda(\Fil^1 Q) + (\underbrace{1,\ldots,1}_{n_\lambda(0)},0,\ldots)\leq (\lambda-1) + (\underbrace{1,\ldots,1}_{n_\lambda(0)},0,\ldots) = \lambda\ ,
\]
as desired. This analysis also implies that if one has equality $\lambda(Q)=\lambda$, then the filtration is given by $\Fil^i Q = p^iQ$; thus, Corollary  \ref{PointwiseIsom} implies that $\Dem_\lambda(\calQ)\to X$ is an isomorphism if $\calQ$ is of type exactly $\lambda$.

Conversely, assume that $\lambda(Q)\leq \lambda$, again in the case $X=\Spec(k)$. In that case, we have to show that
\[
R\Gamma(\Dem_\lambda(\calQ),\calO) = k\ ,
\]
which implies in particular that $\Dem_\lambda(\calQ)$ is nonempty. Recall that $\Dem_\lambda(\calQ)$ can be written as
\[
\Dem_\lambda(\calQ) = \Dem_{\lambda - 1}(\Fil^1 \calQ)\ ,
\]
where $\Fil^1 \calQ$ is the kernel of the universal quotient $\calQ\to \calG$ over $\Quot(\calQ/p,n_\lambda(0))$. Arguing inductively (always using Lemma \ref{CriterionO} to pass from fibrewise information to global information), we see that the locus
\[
\Quot(\calQ/p,n_\lambda(0))_{\leq \lambda-1}\subset \Quot(\calQ/p,n_\lambda(0))
\]
where $\ker(\calQ\to \calG)$ is of type $\leq \lambda - 1$ is the image of
\[
f: \Dem_{\lambda-1}(\Fil^1 \calQ)\to \Quot(\calQ/p,n_\lambda(0))\ ,
\]
and in particular closed, and $Rf_\ast \calO = \calO_{\mathrm{Im}\ f}$, so it remains to prove that
\[
R\Gamma(\Quot(\calQ/p,n_\lambda(0))_{\leq \lambda-1},\calO) = k\ .
\]

Let us describe the locus $\Quot(\calQ/p,n_\lambda(0))_{\leq \lambda-1}$. Note that $Q/p$ comes with a decreasing filtration
\[
F^m (Q/p)=\ker(Q/p\buildrel {p^m}\over\longrightarrow p^mQ/p^{m+1}Q)\ .
\]
We claim that a quotient $\calQ/p\to \calG$ with kernel $\calF\subset \calQ/p$ lies in $\Quot(\calQ/p,n_\lambda(0))_{\leq \lambda-1}$ if and only if for all $m\geq 0$, $\dim (\calF\cap F^m(\calQ/p))\geq a_m$ for certain integers $a_m$ (determined by $\lambda$ and $Q$).\footnote{Here, $\calQ$ stands for the base extension of $Q$ to varying fields.} Indeed, by Lemma \ref{CharMajorization}, $\lambda(\Fil^1 \calQ)\leq \lambda - 1$ if and only if for all $m\geq 0$, the length of $p^m \Fil^1 \calQ$ is bounded by the length of $p^m Q_{\lambda-1}$, where $Q_{\lambda-1}$ is a torsion $W(k)$-module with $\lambda(Q_{\lambda-1})=\lambda-1$. But there is an exact sequence
\[
0\to p^m \Fil^1 \calQ\to p^m \calQ\to \coker(\calF\buildrel {p^m}\over\longrightarrow p^m\calQ/p^{m+1}\calQ)\to 0\ ,
\]
which allows one to compute the length of $p^m \Fil^1 \calQ$ in terms of $Q$, $\dim \calF = n_{\lambda_Q}(0) - n_\lambda(0)$ and $\dim (\calF\cap F^m(\calQ/p))$. More precisely, the inequalities are
\[
\dim (\calF\cap F^m(\calQ/p))\geq n_{\lambda_Q}(0) - n_\lambda(0) + \sum_{i\geq m+1} (n_{\lambda_Q}(i)-n_\lambda(i))=:a_m\ .
\]

Thus, the following lemma finishes the proof. For this we need to observe two inequalities. First,
\[\begin{aligned}
\dim F^m Q/p = n_{\lambda_Q}(0) - n_{\lambda_Q}(m)&\geq n_{\lambda_Q}(0) - n_\lambda(m) + \sum_{i\geq m+1} (n_{\lambda_Q}(i)-n_\lambda(i))\\
&\geq n_{\lambda_Q}(0) - n_\lambda(0) + \sum_{i\geq m+1} (n_{\lambda_Q}(i)-n_\lambda(i))=a_m\ ,
\end{aligned}\]
using the majorization
\[
\sum_{i\geq m} n_\lambda(i)\geq \sum_{i\geq m} n_{\lambda_Q}(i)\ .
\]
Second,
\[
\dim \calF = n_{\lambda_Q}(0) - n_\lambda(0)\geq n_{\lambda_Q}(0) - n_\lambda(0) + \sum_{i\geq m+1} (n_{\lambda_Q}(i)-n_\lambda(i))=a_m\ ,
\]
using the same majorization.
\end{proof}

\begin{lemma} Let $k$ be a perfect field, and $V$ a finite-dimensional $k$-vector space equipped with a decreasing filtration $F^m V\subset V$, $m=1,\ldots,N$. Let $a_1,\ldots,a_N$ be integers such that $a_m\leq \dim F^m V$, and let $n\leq \dim V$ be an integer such that $n\geq a_m$ for all $m$. Consider the space $X$ of subspaces $\calF\subset V$ of dimension $n$ such that $\dim (\calF\cap F^m V)\geq a_m$ for all $m=1,\ldots,N$. Then $X$ is a projective $k$-scheme, and
\[
R\Gamma(X_\perf,\calO)=k\ .
\]
\end{lemma}

\begin{proof} Clearly, $X$ is a projective $k$-scheme, as it is closed inside a Grassmannian.

For the cohomological statement, we argue by induction on $N$. We may clearly assume $a_m > 0$ for all $m$. Let $\tilde{X}\to X$ be the covering which parametrizes subspaces $W\subset \calF\cap F^N V$ of exact dimension $a_N$. Then all fibres of $\tilde{X}\to X$ are perfections of Grassmannians. By Lemma \ref{CriterionO}, the derived pushforward of $\calO_{\tilde{X}_\perf}$ is $\calO_{X_\perf}$. Thus, it is enough to prove that
\[
R\Gamma(\tilde{X}_\perf,\calO)=k\ .
\]
On the other hand, $\tilde{X}$ maps to the Grassmannian of $a_N$-dimensional subspaces of $F^N V$, and each fibre is a similar scheme for $V/W$ with the induced filtration and $a_m^\prime = a_m - a_N$, $n^\prime = n-a_N$. By induction and Lemma \ref{CriterionO}, the result follows.
\end{proof}

\begin{proof}[Proof of Lemma \ref{Semicontinuity}] By Lemma \ref{ImageDemazure} and Proposition \ref{ExDemazure}, the desired locus is the image of a proper perfectly finitely presented morphism, and therefore closed.
\end{proof}

\section{The Witt vector affine Grassmannian}
\label{sec:WittAffGrGLn}

In this section, we use the results proved earlier in the paper to establish the promised representability result for the Witt vector affine Grassmannian.

\subsection{Statements}

We now introduce the affine Grassmannian. For this, fix once and for all an integer $n\geq 0$.

\begin{definition} For any sequence $\lambda=(\lambda_1,\ldots,\lambda_n,0,\ldots)$ of non-negative integers as above, let $\Gr_{\leq \lambda}$ be the functor on $\Perf$ sending $X\in \Perf$ to the set of finite projective $W(\calO_X)$-submodules $\calE\subset W(\calO_X)^n$ such that the defining inclusion $\beta: \calE\hookrightarrow W(\calO_X)^n$ is an isogeny, and $\calQ=\coker(\beta)$ is of type $\leq \lambda$.

	Let $\Gr_\lambda\subset \Gr_{\leq \lambda}$ be the subfunctor where the type is exactly $\lambda$. For $\mu\leq \lambda$, there is a closed immersion $\Gr_{\leq \mu} \subset \Gr_{\leq \lambda}$ of functors (by Lemma \ref{Semicontinuity}), and $\Gr_\lambda = \Gr_{\leq \lambda} \setminus \cup_{\mu < \lambda} \Gr_{\leq \mu}$ is open in $\Gr_{\leq \lambda}$.
\end{definition}

Note that by Theorem \ref{thm:VectWhStack} and Corollary \ref{cor:HomVectWhStack}, the functor $\Gr_{\leq \lambda}$ is a $v$-sheaf. The following theorem was proved by Zhu, \cite{ZhuMixedCharGeomSatake}.

\begin{theorem}[Zhu]\label{thm:ZhuWittGrProper}
The functor $\Gr_{\leq \lambda}$ is represented by the perfection of a proper algebraic space over $\F_p$.
\end{theorem}

Our main result is:

\begin{theorem}\label{thm:WittGrProjective}
The functor $\Gr_{\leq \lambda}$ is representable by a proper perfectly finitely presented $\F_p$-scheme, and there is a natural ample line bundle $\calL \in \Pic(\Gr_{\leq \lambda})$.
\end{theorem}

In particular, $\Gr_{\leq\lambda}$ is the perfection of a projective $\F_p$-scheme.

In the $K$-theoretic approach of Theorem \ref{thm:DetGeneral}, the line bundle $\calL$ is given by $\calL=\widetilde{\det}(\calQ)$. Note that $\calQ$ is quasi-isomorphic to the perfect complex $\calE\to W(\calO)^n$, which is supported set-theoretically on the special fibre $\{p=0\}$. In the geometric approach, the existence of $\calL$ is Theorem \ref{ExLineBundle} below.

Our proof of Theorem \ref{thm:WittGrProjective} is independent of Theorem \ref{thm:ZhuWittGrProper}. Our strategy is to first construct the line bundle $\calL$, and then to employ a fundamental theorem of Keel \cite{KeelSemiAmple} on positivity of line bundles in characteristic $p$; this will allow us to work our way up to $\Gr_{\leq \lambda}$ from lower-dimensional strata by induction. Keel's theorem only applies to projective schemes, so we cannot apply it directly to $\Gr_{\leq \lambda}$; instead, we will apply it to a suitable $h$-cover. 

\subsection{The Demazure resolution}

To study $\Gr_{\leq \lambda}$, it will be useful to have access to (what will turn out to be) a convenient resolution.

\begin{definition}\label{DefDemazureRes} The Demazure resolution is the map
\[
\psi: \widetilde{\Gr_\lambda} = \Dem_\lambda(\calQ)\to \Gr_{\leq \lambda}\ ,
\]
of $v$-sheaves on $\Perf$, where $W(\calO)^n\to \calQ$ denotes the universal cokernel on $\Gr_{\leq \lambda}$.
\end{definition}

\begin{remark}\label{RemDemazure}
By Proposition \ref{ExDemazure}, the map $\psi: \widetilde{\Gr_\lambda}\to \Gr_{\leq \lambda}$ is relatively representable by a proper perfectly finitely presented map. Moreover, $\psi$ is surjective and an isomorphism over $\Gr_\lambda$ by Lemma \ref{ImageDemazure}, and
\[
R\psi_\ast \calO_{\widetilde{\Gr}_\lambda} = \calO_{\Gr_{\leq \lambda}}
\]
by Lemma \ref{ImageDemazure} and Lemma \ref{CriterionO}.\footnote{As we do not yet know that $\Gr_{\leq \lambda}$ is representable, all of these assertions mean that they hold true after an arbitrary base-change to a representable $S$ over $\Gr_{\leq \lambda}$. Note that viewing $\Gr_{\leq \lambda}$ as an object of a suitable ringed topos (the one associated to the site $\Perf$ ringed using $\calO$ equipped with the $v$-topology) leads to a potentially different notion of a pushforward. Nevertheless, both these definitions coincide by Lemma \ref{BaseChange}. }
\end{remark}

The following proposition justifies the name Demazure \emph{resolution}.

\begin{proposition}\label{ExDemazureResolution}	The functor $\widetilde{\Gr_\lambda}$ is representable by the perfection of a smooth projective $\F_p$-scheme. It represents the functor associating to $X\in \Perf$ the set of $\lambda_1$-tuples $\calE_{\lambda_1}\subset \ldots\subset \calE_1\subset \calE_0 = W(\calO_X)^n$ of finite projective $W(\calO_X)$-submodules of $W(\calO_X)^n$ such that each $\calQ_i := \calE_i/\calE_{i+1}$ is a finite projective $\calO_X$-module of rank $n_\lambda(i)$, where $n_\lambda(i)$ is the number defined in Remark \ref{rmk:RankGraded}.
\end{proposition}

\begin{proof} From the definitions, it follows that $\widetilde{\Gr_\lambda}$ is the subfunctor of those $\lambda_1$-tuples $(\calE_1,\ldots,\calE_{\lambda_1})$ as in the statement, for which the cokernel $\calQ$ of $\calE_{\lambda_1}\hookrightarrow W(\calO_X)^n$ is of type $\leq \lambda$. However, Lemma \ref{ImageDemazure} guarantees that $\calQ$ is always of type $\leq \lambda$.

The moduli description now presents $\widetilde{\Gr_\lambda}$ as a successive perfect Grassmannian bundle. Indeed, set $X_0=\Spec(\F_p)$, $\calE_0 = W(\calO_{X_0})^n$. Define inductively $\pi_i: X_{i+1}\to X_i$ and $\calE_{i+1}$ over $X_{i+1}$, by letting $X_{i+1}=\Quot(\calE_i/p,n_\lambda(i))$ parametrize quotients $\calQ_i$ of $\calE_i/p$ of rank $n_\lambda(i)$, and setting $\calE_{i+1} = \ker(\pi_i^\ast \calE_i\to \calQ_i)$. Then $\calE_{i+1}$ is still a finite projective $W(\calO_{X_{i+1}})$-module by Lemma \ref{ReductionProjectiveDim}, and $X_{\lambda_1} = \widetilde{\Gr_\lambda}$. By induction, each $X_i$ is the perfection of a smooth projective $\F_p$-scheme.
\end{proof}

The intermediate isogenies being recorded in $\widetilde{\Gr_{\lambda}}$ keep track of how the quotient $\calQ$ on $\Gr_{\leq \lambda}$, which is a $p^{\lambda_1}$-torsion module, is filtered by $p$-torsion modules. The following example is perhaps useful: 

\begin{example} Consider the case $n = 3$ and $\lambda = (2,1,0)$. Then $\widetilde{\Gr_\lambda}$ parametrizes chains of isogenies $\calE_2 \to \calE_1 \to \calE_0=W(\calO)^3$, where each $\calQ_i := \calE_i/\calE_{i+1}$ is a vector bundle of rank $2-i$. The only element $\mu < \lambda$ is $\mu := (1,1,1)$. Hence, $\Gr_{\leq \lambda}$ has a $2$-step stratification with open $\Gr_\lambda$, and closed $\Gr_\mu$. The Demazure resolution $\widetilde{\Gr_\lambda}$ is a tower of two perfect $\P^2$-bundles over a point, and hence has dimension $4$; it follows $\Gr_\lambda$ also has dimension $4$. Now $\Gr_\mu$ is a single point, and classifies the inclusion $p\calE_0 \hookrightarrow \calE_0$. The fibre of $\psi$ over $\Gr_\mu$ is the perfection of $\P^2$, corresponding to the Grassmannian of rank-$2$-quotients of $\calE_0/p\calE_0=\calO^3$.
\end{example}

\subsection{The line bundle $\calL$}

The goal of this section is to construct a natural line bundle $\calL$ on $\Gr_{\leq \lambda}$.

\begin{theorem}\label{ExLineBundle} There is a line bundle $\calL_\lambda$ on $\Gr_{\leq \lambda}$ such that $\psi^\ast \calL_\lambda$ is the line bundle $\widetilde{\calL_\lambda}$ on $\widetilde{\Gr_\lambda}$ given by
\[
\widetilde{\calL_\lambda} = \bigotimes_{i=0}^{\lambda_1-1} \det_{\calO_X} \calQ_i\ .
\]
The restriction of $\calL_\lambda$ to $\Gr_{\leq \mu}$ is given by $\calL_\mu$ for $\mu<\lambda$, compatibly in $\mu$.
\end{theorem}

We recall that by Proposition \ref{PullbackFullyFaithful} and Remark \ref{RemDemazure}, the functor $\psi^\ast$ from line bundles on $\Gr_{\leq \lambda}$ to line bundles on $\widetilde{\Gr_\lambda}$ is fully faithful.

\begin{proof} In the $K$-theoretic approach, one can define $\calL_\lambda = \widetilde{\det}(\calQ)$, which is clearly compatible with restriction to $\Gr_{\leq \mu}$. In the geometric approach, the existence of $\calL_\lambda$ follows by combining Theorem \ref{ThmVectTrivial}, Remark \ref{RemDemazure} and Lemma \ref{TrivialLineBundle}.

However, in the geometric approach it is not immediately clear that the line bundles $\calL_\lambda$ thus constructed are compatible for varying $\lambda$. To see this, it is enough to consider the special case $\lambda=(N,0,\ldots)$: any two comparable $\lambda$'s are both less than some $\lambda$ of this form. Then $\widetilde{\Gr_\lambda}$ parametrizes finite projective $W(\calO)$-submodules $\calE_N\subset \ldots \subset \calE_1\subset \calE_0 = W(\calO)^n$ such that each $\calQ_i = \calE_i / \calE_{i+1}$ is a line bundle. For any $\mu<\lambda$, there is a cover $\widetilde{\Gr_\mu}(\lambda)\to \widetilde{\Gr_\mu}$ refining the universal filtration of $\calQ$ on $\widetilde{\Gr_\mu}$ to a complete flag; this has geometrically connected fibres, given by products of perfections of classical flag varieties. Then $\widetilde{\Gr_\mu}(\lambda)\subset \widetilde{\Gr_\lambda}$ is a closed subfunctor, given by the condition that certain quotients $\calE_i/\calE_j$ are killed by $p$. The line bundles thus constructed on $\widetilde{\Gr_\lambda}$ and $\widetilde{\Gr_\mu}(\lambda)$ are clearly compatible, and then Proposition \ref{PullbackFullyFaithful} implies the same for $\Gr_{\leq \mu}\subset \Gr_{\leq \lambda}$.
\end{proof}

Because of the compatibility between the different $\calL_\lambda$, we will simply call them $\calL$ below.

\subsection{Proof of Theorem \ref{thm:WittGrProjective}}

We put the preceding geometry to use in proving $\calL \in \Pic(\Gr_{\leq \lambda})$ is ample.\footnote{Cf. Lemma \ref{AmplePerf} for the meaning of this statement.}

\begin{lemma}
\label{lem:LineBundlesonDemazure}
The following statements are true about line bundles on $\widetilde{\Gr_\lambda}$.
\begin{enumerate}
\item[{\rm (i)}] The line bundle $\otimes_{i=0}^{\lambda_1 - 1} \det(\calQ_i)^{a_i}$ is ample if $a_0 \gg a_1 \gg a_2 \gg ... \gg a_{\lambda_1 - 1} \gg 0$.
\item[{\rm (ii)}] For each $x \in \Gr_\lambda$ and $i=0,\ldots,\lambda_1-1$, there exists a section $s \in H^0(\widetilde{\Gr_\lambda},\det(\calQ_i))$ such that $s(x) \neq 0$.
\end{enumerate}
\end{lemma}

\begin{proof} For (i), we work by induction using the tower
\[ X_{\lambda_1} \stackrel{\pi_{\lambda_1 - 1}}{\to} X_{\lambda_1 - 1} \to ... \stackrel{\pi_1}{\to} X_1 \stackrel{\pi_0}{\to} X_0\]
encountered in the proof of Proposition \ref{ExDemazureResolution}. Specifically, by induction, we will check that for $m \leq \lambda_1$, the bundle $\otimes_{i=0}^{m-1} \det(\calQ_i)^{a_i}$ is ample on $X_m$ provided $a_0 \gg a_1 \gg a_2 \gg \dots \gg a_{m-1} \gg 0$. When $m = 0$, $X_0$ is a point, so the statement is clear. Assume inductively that there exist integers $b_0 \gg b_1 \gg \dots \gg b_{m-1} \gg 0$ such that $\otimes_{i=0}^{m-1} \det(\calQ_i)^{b_i}$ is ample on $X_m$. Now $X_{m+1} \to X_m$ is (the perfection of) a Grassmannian fibration with universal quotient bundle $\calQ_m$, so $\det(\calQ_m)$ is relatively ample for this morphism. It follows that $\det(\calQ_m) \otimes \Big(\otimes_{i=0}^{m-1} \det(\calQ_i)^{b_i}\Big)^N$ is ample on $X_{m+1}$ for $N \gg 0$, which proves the inductive hypothesis, and thus the claim.

For (ii), as all the $\calQ_i$'s are killed by $p$, we first note the following: for each $i$, there is a natural map $\calO^n \simeq p^i W(\calO)^n/p^{i+1}W(\calO^n) \to \calQ_i$ of vector bundles on $\widetilde{\Gr_\lambda}$ given by taking a local section $f \in \calO^n$ to the residue class of $p^i \tilde{f}$ for a suitable lift $\tilde{f} \in W(\calO)^n$. Moreover, these maps are surjective over $\Gr_\lambda$. Indeed, at a point $x \in \Gr_\lambda$, the $p$-adic filtration on $\calQ_x$ coincides with the one coming from the $\calQ_i$, i.e., $\calQ_i = p^i\calQ/p^{i+1}\calQ$ over $\Gr_\lambda$; see the proof of Lemma \ref{ImageDemazure}. As the map  $W(\calO)^n \to \calQ$ is surjective, also $p^i\calO^n/p^{i+1} \calO^n \to p^i\calQ/p^{i+1}\calQ$ is surjective, and hence the map $\calO^n \to \calQ_i$ considered above is surjective at $x$. This implies that $\wedge^{\rk(\calQ_i)} (\calO^n) \to \wedge^{\rk(\calQ_i)} (\calQ_i) = \det(\calQ_i)$ is surjective at $x$. We then find a section of $\det(\calQ_i)$ non-vanishing at $x$ by picking a suitable general $1$-dimensional subspace of $L_0 \subset \wedge^{\rk(\calQ_i)} (\F_p^{\oplus n})$, and using the induced map $L_0 \otimes \calO_{\widetilde{\Gr_{\lambda}}} \to \wedge^{\rk(\calQ_i)} (\calO^n) \to \det(\calQ_i)$ as a section.
\end{proof}

Next, we check that $\calL$ is nef, even strictly nef.

\begin{lemma}\label{lem:StrictlyNef}
The line bundle $\calL \in \Pic(\Gr_{\leq \lambda})$ is strictly nef. In particular, $\widetilde{\calL}\in \Pic(\widetilde{\Gr_\lambda})$ is nef.
\end{lemma}

Here ``strictly nef'' means that $\calL$ has positive degree on any non-constant curve.

\begin{proof} Fix a smooth connected projective curve $C$ over $k$, and a non-constant map $f: C_\perf\to \Gr_{\leq \lambda}$. Then the generic point of $C_\perf$ maps into $\Gr_\mu$ for a unique $\mu$, in which case $f$ factors through $\Gr_{\leq \mu}$ by Lemma \ref{Semicontinuity}. Renaming $\mu$ as $\lambda$, we can assume that $C_\perf$ meets $\Gr_\lambda$. Then the pullback of $\widetilde{\Gr_\lambda}\to \Gr_{\leq \lambda}$ to $C_\perf$ is proper and admits a generic section; by properness, the section extends to a map $\tilde{f}: C_\perf\to \widetilde{\Gr_\lambda}$. In that case, we need to prove that $\widetilde{\calL} = \otimes_i \det_{\calO} (\calQ_i)$ is ample on $C_\perf$.

By Lemma \ref{lem:LineBundlesonDemazure} (ii) (and the assumption that $C_\perf$ meets $\Gr_\lambda$), $\det_{\calO} (\calQ_i)|_{C_\perf}$ has a nonvanishing section, and thus is effective. If their tensor product is not ample, then it follows that all $\det_{\calO} (\calQ_i)$ become trivial over $C_\perf$. On the other hand, Lemma \ref{lem:LineBundlesonDemazure} (i) guarantees that a weighted tensor product of the $\det_{\calO}(\calQ_i)$ is ample (on $\widetilde{\Gr_\lambda}$, and thus on $C_\perf$ since $f$ is non-constant), which is a contradiction.
\end{proof}

Now we can prove that $\widetilde{\calL}$ is big.

\begin{lemma}\label{lem:BigonDemazure}
The line bundle $\widetilde{\calL}$ is big, with exceptional locus contained in the boundary $\widetilde{\Gr_\lambda} \setminus \Gr_\lambda$.
\end{lemma}

We briefly recall the meaning the terms used above; see \cite[\S 1,2]{LazarsfeldPos1} and \cite[\S 0]{KeelSemiAmple} for more. Fix a line bundle $N$ on a proper variety $X$ over some field.
\begin{enumerate}
	\item We say that $N$ is {\em big} if, for $m \gg 0$, we can write $N^{\otimes m} \simeq A \otimes E$, where $A$ is an ample line bundle, and $E$ is an effective line bundle. 
	\item If $N$ is nef, its {\em exceptional locus} is the Zariski closure of the union of all closed subvarieties $Z \subset X$ such that $N|_Z$ is not big.
\end{enumerate}

In the perfect setting, we define these notions by passage to finite type models. 

\begin{proof} Recall that $\widetilde{\calL} = \otimes_{i=0}^{\lambda_1 - 1} \det(\calQ_i)$. Choose integers $N \gg a_0 \gg a_1 \gg \dots \gg a_{\lambda_1 - 1} \gg 0$, and write 
\[ \widetilde{\calL}^{\otimes N} =  \Big(\otimes_{i=0}^{\lambda_1 - 1} \det(\calQ_i)^{a_i}\Big) \otimes \Big(\otimes_{i=0}^{\lambda_1 - 1} \det(\calQ_i)^{N - a_i}\Big).\]
By Lemma \ref{lem:LineBundlesonDemazure} (i), the first term on the right is ample, while Lemma \ref{lem:LineBundlesonDemazure} (ii) implies that the second term is effective. It immediately follows that $\widetilde{\calL}^{\otimes N}$, and hence $\widetilde{\calL}$, is big. In fact, this factorisation combined with Lemma \ref{lem:LineBundlesonDemazure} (ii) shows that for any $x \in \Gr_\lambda$, we can write $\widetilde{\calL}^{\otimes N} = A(D)$, where $A$ is ample, and $D$ is an effective divisor missing $x$. Thus, the exceptional locus $E(\widetilde{\calL}^{\otimes N}) = E(\widetilde{\calL})$ misses $x$ by \cite[Lemma 1.7]{KeelSemiAmple} and Lemma \ref{lem:StrictlyNef}. Varying $x$ then shows that $E(\widetilde{\calL}) \subset \widetilde{\Gr_\lambda} \setminus \Gr_{\lambda}$.
\end{proof}

We now finish the promised proof:

\begin{proof}[Proof of Theorem \ref{thm:WittGrProjective}]
We will prove that $\Gr_{\leq \lambda}$ is representable, and $\calL$ is ample on $\Gr_{\leq \lambda}$, by induction on $\lambda$. When $\lambda$ is minimal, $\calL = \widetilde{\calL}$ is ample on the Grassmannian $\Gr_{\leq \lambda} = \widetilde{\Gr_\lambda}$. Assume inductively that $\calL|_{\Gr_{\leq \mu}}$ is ample for all $\mu < \lambda$. 

First, we prove that $\widetilde{\calL}$ is semiample on $\widetilde{\Gr_\lambda}$. Note that $\widetilde{\calL}$ is nef by Lemma \ref{lem:StrictlyNef}. Using Keel's \cite[Theorem 1.9]{KeelSemiAmple}, it is enough to check that $\widetilde{\calL}|_{E(\widetilde{\calL})}$ is semiample. By Lemma \ref{lem:BigonDemazure}, the locus $E(\widetilde{\calL})$ is contained in $\psi^{-1}(\Gr_{\leq \lambda} \setminus \Gr_\lambda) = \psi^{-1}(\cup_{\mu < \lambda} \Gr_{\leq \mu})$. By induction, we know that $\calL|_{\Gr_{\leq \mu}}$ is ample for $\mu < \lambda$. Using \cite[Lemma 1.8]{KeelSemiAmple}, this shows $\calL|_{\cup_{\mu < \lambda} \Gr_{\leq \mu}}$ is ample, so $\widetilde{\calL}|_{\psi^{-1}(\cup_{\mu < \lambda} \Gr_{\leq \mu})}$ is semiample, and thus $\widetilde{\calL}|_{E(\widetilde{\calL})}$ is semiample.

Let $\phi:\widetilde{\Gr_\lambda} \to X$ be the Stein factorization associated to $\widetilde{\calL}$, so $\phi$ is a proper surjective perfectly finitely presented map of perfectly finite presented $k$-schemes with geometrically connected fibres (and thus a $v$-cover), and $\widetilde{\calL}^{\otimes N} = \phi^* \calM$ for some ample $\calM \in \Pic(X)$ and $N\geq 1$.\footnote{For the construction, choose finite type models, apply the Stein factorization, and then go to the perfection.}

We claim that
\begin{equation}
	\label{eqn:relationsagree}
	\widetilde{\Gr_\lambda} \times_{\Gr_{\leq \lambda}} \widetilde{\Gr_\lambda} = \widetilde{\Gr_\lambda} \times_X \widetilde{\Gr_\lambda}
\end{equation}
as closed subschemes of $\widetilde{\Gr_\lambda}\times \widetilde{\Gr_\lambda}$. Assuming this claim for the moment, we can finish the proof as follows. Both $X$ and $\Gr_{\leq \lambda}$ are $v$-sheaves, and thus are given by the coequalizer
\[
\widetilde{\Gr_\lambda} \times_X \widetilde{\Gr_\lambda}\rightrightarrows \widetilde{\Gr_\lambda}\ ,
\]
resp.
\[
\widetilde{\Gr_\lambda} \times_{\Gr_{\leq \lambda}} \widetilde{\Gr_\lambda}\rightrightarrows \widetilde{\Gr_\lambda}\ ,
\]
as $v$-sheaves. As the equivalence relation agrees, we get $\Gr_{\leq \lambda}=X$, on which $\calL^{\otimes N}=\calM$ (by Lemma \ref{PullbackFullyFaithful}) is ample, as wanted.

It remains to verify equation \eqref{eqn:relationsagree} above. As everything is perfect, and in particular reduced, this can be checked on $k$-points, at least after enlarging $k$ to make it algebraically closed.

Thus, let $x,y\in \widetilde{\Gr_\lambda}(k)$ be any pair of points mapping to the same point of $X$. As the fibres of $\widetilde{\Gr_\lambda}\to X$ are geometrically connected, it follows that there is a geometrically connected $F\subset (\widetilde{\Gr_\lambda})_k$ contracted to a point in $X$, such that $x,y\in F$. We can even assume that $F$ is a curve. Note that $\widetilde{\calL}^{\otimes N}$ is trivial on $F$. We want to show that $F$ gets contracted under $\widetilde{\Gr_\lambda}\to \Gr_{\leq \lambda}$, as then $(x,y)\in (\widetilde{\Gr_\lambda}\times_{\Gr_{\leq \lambda}} \widetilde{\Gr_\lambda})(k)$. But as the pullback of $\calL^{\otimes N}$ to $F$ is trivial, $F$ gets contracted by Lemma \ref{lem:StrictlyNef}.

Conversely, let $x,y\in \widetilde{\Gr_\lambda}(k)$ be a pair of points mapping to the same point of $\Gr_{\leq \lambda}$. The same arguments apply, as the fibres of $\widetilde{\Gr_\lambda}\to \Gr_{\leq \lambda}$ are geometrically connected by Lemma \ref{ImageDemazure}, and $\calM$ is ample on $X$.
\end{proof}

\section{Affine Grassmannians for general groups}
\label{sec:WittAffGrG}

In this section, we deduce as corollaries several results concerning more general group schemes. Fix a complete discrete valuation field $K$ of characteristic $0$ with perfect residue field $k$ of characteristic $p$ and ring of integers $\calO_K$. Moreover, let $G$ be a reductive group over $K$, and let $\calG$ be a smooth affine group scheme over $\calO_K$ with generic fibre $G$.

For a $k$-algebra $R$, we define the relative Witt vectors $W_{\calO_K}(R)$ as
\[
W_{\calO_K}(R) = W(R)\otimes_{W(k)} \calO_K\ ,
\]
noting that there is a canonical inclusion $W(k)\hookrightarrow \calO_K$.

\begin{definition} The \emph{($p$-adic) loop group} of $G$ is the functor
\[
LG: R\mapsto G(W_{\calO_K}(R)[\frac{1}{p}])
\]
on perfect $k$-algebras. The \emph{positive ($p$-adic) loop group} of $\calG$ is the functor
\[
L^+\calG: R\mapsto \calG(W_{\calO_K}(R))
\]
on perfect $k$-algebras.
\end{definition}

Clearly, both functors take values in groups. As regards representability, we have the following well-known result.

\begin{proposition} The functor $L^+\calG$ is representable by an affine perfect scheme. The functor $LG$ is a strict ind-(perfect affine scheme), meaning that it can be written as an inductive limit of perfect affine schemes along closed immersions.
\end{proposition}

\begin{remark} One can define $LG$ and $L^+\calG$ on general (non-perfect) $k$-algebras by the same formula. Then $L^+\calG$ is already representable by a (non-perfect) affine scheme, but $LG$ is not representable by a \emph{strict} ind-(affine scheme). The problem is that elements of $W_{\calO_K}(R)[\frac{1}{p}]$ do not admit a simple description as infinite sequences of elements of $R$, if $R$ is not perfect.
\end{remark}

\begin{proof} The first assertion is true for any affine scheme $X$ over $\calO_K$ in place of $\calG$, cf. work of Greenberg, \cite{Greenberg}, and follows by observing that any element of $W_{\calO_K}(R)$ can be written as an infinite sequence of elements of $R$ in such a way that addition and multiplication are given by polynomial functions. The statement about $LG$ can be reduced to the case of $\GL_n$ by fixing a closed embedding $G\hookrightarrow \GL_n$ (which induces a closed embedding $LG\hookrightarrow L\GL_n$), where one gets representable subfunctors by restricting the pole orders of $g, g^{-1}\in \GL_n(W_{\calO_K}(R)[\frac{1}{p}])$.
\end{proof}

\begin{definition} The \emph{affine Grassmannian} of $\calG$ is the fpqc quotient
\[
\Gr_\calG = LG / L^+\calG
\]
on the category of perfect $k$-schemes.
\end{definition}

One immediately verifies the following proposition, cf. also \cite[Theorem 5]{Kreidl}.

\begin{proposition} If $\calG=\GL_n$, then for any perfect $k$-algebra $R$, $\Gr_{\GL_n}(R)$ is the set of finite projective $W_{\calO_K}(R)$-modules $M\subset W_{\calO_K}(R)[\frac{1}{p}]^n$ such that $M[\frac{1}{p}] = W_{\calO_K}(R)[\frac{1}{p}]^n$.$\hfill \Box$
\end{proposition}

Now one has the following result.

\begin{corollary} The affine Grassmannian $\Gr_\calG$ can be written as an increasing union of perfections of quasiprojective schemes over $k$, along closed immersions. If $\calG$ is a parahoric group scheme, then $\Gr_\calG$ can be written as an increasing union of perfections of projective schemes over $k$, along closed immersions.
\end{corollary}

\begin{proof} Replacing $\calG$ by $\mathrm{Res}_{\calO_K/W(k)} \calG$, we may assume that $\calO_K=W(k)$. One can find a representation $\calG\hookrightarrow \GL_n$ such that $\GL_n/\calG$ is quasi-affine, cf. \cite[\S 1.b]{PappasRapoport}. In that case, the induced map $\Gr_\calG\hookrightarrow \Gr_{\GL_n}$ is a locally closed embedding, cf. \cite[Proposition 1.20]{ZhuMixedCharGeomSatake} in the case considered here. This reduces representability of $\Gr_\calG$ to the case of $\calG=\GL_n$. But here, the result follows from Theorem \ref{thm:WittGrProjective}.

If $\calG$ is parahoric, then \cite[\S 1.5.2]{ZhuMixedCharGeomSatake} shows that $\Gr_\calG$ is ind-proper, and thus (by ind-quasi projectivity) ind-projective.
\end{proof}

One can also determine the connected components of $\Gr_\calG$ in case $\calG$ is parahoric, using Kottwitz' map
\[
\kappa: LG(\bar{k})\to \pi_1(G)_{\Gal_K}\ ,
\]
defined for any algebraically closed field $\bar{k}$ containing $k$.

\begin{proposition}[{\cite[Proposition 1.21]{ZhuMixedCharGeomSatake}}]\label{AffFlagConnComp} Assume that $\calG$ is parahoric. There are canonical bijections $\kappa: \pi_0(LG)\cong \pi_0(\Gr_\calG)\buildrel\cong\over\longrightarrow \pi_1(G)_{\Gal_K}$, where $\Gal_K$ is the absolute Galois group of $K$.$\hfill \Box$
\end{proposition}

\section{The central extension of $LG$}
\label{sec:WittAffGrSL}

In this section, let us fix $G=\SL_n$, $n\geq 2$, over a complete discrete valuation ring $\calO_K$ of mixed characteristic with perfect residue field $k$ as above. We have the affine Grassmannian $\Gr_{\SL_n}$.

\begin{proposition} The affine Grassmannian $\Gr_{\SL_n}$ parametrizes finite projective $W_{\calO_K}(R)$-modules $M\subset W_{\calO_K}(R)[\frac{1}{p}]^n$ such that $M[\frac{1}{p}] = W_{\calO_K}(R)[\frac{1}{p}]^n$ and $\det M = W_{\calO_K}(R)$. There is a natural ample line bundle $\calL$ on $\Gr_{\SL_n}$ given by
\[
\calL = \widetilde{\det}_R(p^a W_{\calO_K}(R)^n/M)
\]
for any $a\ll 0$.
\end{proposition}

\begin{proof} The first part is standard, cf. \cite[Theorem 5]{Kreidl}. For the second part, note that
\[
\widetilde{\det}_R: K(W(R)\on R)\to \gPic^\Z(R)
\]
induces by composition $K(W_{\calO_K}(R)\on R)\to K(W(R)\on R)$ (using that a perfect complex of $W_{\calO_K}(R)$-modules stays perfect as a complex of $W(R)$-modules) a similar determinant map
\[
\widetilde{\det}_R: K(W_{\calO_K}(R)\on R)\to \gPic^\Z(R)\ .
\]
Clearly, the line bundle $\calL$ is independent of the choice of $a\ll 0$.
\end{proof}

The line bundle $\calL$ is \emph{not} equivariant under the action of $LG$. However, for any fixed element $g\in LG(R)$, the line bundles $g^\ast \calL$ and $\calL$ on $\Gr_{\SL_n}\otimes_k R$ are locally on $R$ isomorphic: Their difference is given by the line bundle $\det_R( p^a W_{\calO_K}(R)^n / g W_{\calO_K}(R)^n )$ on $R$ for $a\ll 0$.

\begin{definition} Let $\widetilde{LG}$ be the functor on perfect $k$-algebras $R$ given by
\[
\widetilde{LG}(R) = \{(g,\alpha)\mid g\in LG(R), \alpha: g^\ast\calL\cong \calL\ \mathrm{on}\ \Gr_{\SL_n}\otimes_k R\}\ .
\]
\end{definition}

With this definition, it is clear that $\widetilde{LG}$ acts on $\calL$.

\begin{proposition} There is a short exact sequence
\[
1\to \G_m\to \widetilde{LG}\to LG\to 1
\]
of Zariski sheaves (thus, of $v$-sheaves). This makes $\widetilde{LG}$ a central extension of $LG$ by $\G_m$.
\end{proposition}

\begin{proof} The projection to $LG$ is $(g,\alpha)\mapsto g$, and the embedding of $\G_m$ is given by $\alpha\in \G_m\mapsto (1,\alpha)$. It remains to see that the automorphism group scheme of $\calL$ on $\Gr_{\SL_n}$ is given by $\G_m$. Thus, we have to see that
\[
H^0(\Gr_{\SL_n}\otimes_k R,\G_m) = R^\times\ .
\]
It is enough to see that
\[
H^0(\Gr_{\SL_n}\otimes_k R,\calO) = R\ .
\]
But $\Gr_{\SL_n}$ is an increasing union of perfections of reduced projective $k$-schemes, which we can assume to be connected by Proposition \ref{AffFlagConnComp}.
\end{proof}

On $k$-rational points, one can identify $\widetilde{LG}(k)$ with a more familiar object. We want to identify the central extension
\[
1\to k^\ast\to \widetilde{LG}(k)\to \SL_n(K)\to 1\ .
\]
For any field $F$, Steinberg, \cite{SteinbergExtension}, has constructed a central extension (whose kernel was identified by Matsumoto, \cite{Matsumoto}),
\[
1\to K_2(F)\to \widetilde{\SL_n(F)}\to SL_n(F)\to 1\ .
\]
Here, $K_2(F)$ is the second $K$-group of $F$; note that Milnor and Quillen $K$-theory give the same answer in this range. Thus, $K_2(F)$ is the quotient of the abelian group $F^\ast\otimes_\Z F^\ast$ by the Steinberg relations $x\otimes (1-x)=0$ for all $0,1\neq x\in F$. One way to explain the relation of this extension to Quillen $K$-theory is as follows. There is a natural map
\[
B\SL_n(F)\to B\GL_n(F)\to B\GL_\infty(F)\to (B\GL_\infty(F))^+ = \tau_{\geq 1} K(F)\ ,
\]
where $\tau_{\geq 1} K(F)$ is the connected component of $0$ of Quillen's $K$-theory space $K(F)$ of $F$. This map factors canonically over the simply connected cover $\tau_{\geq 2} K(F)$, as the (determinant) map $\SL_n(F)\to \pi_1 K(F) = F^\times$ is trivial. Composing with $\tau_{\geq 2} K(F)\to B^2 K_2(F)$, one gets a map
\[
B\SL_n(F)\to B^2 K_2(F)\ ,
\]
which (by passing to loops) is equivalent to a map of $\E_1$-groups $\SL_n(F)\to BK_2(F)$, i.e. an extension of $\SL_n(F)$ by $K_2(F)$.

\begin{proposition} The extension
\[
1\to k^\ast\to \widetilde{LG}(k)\to \SL_n(K)\to 1
\]
is the pushout of Steinberg's extension
\[
1\to K_2(K)\to \widetilde{\SL_n(K)}\to \SL_n(K)\to 1
\]
along the tame Hilbert symbol map $K_2(K)\to K_1(k) = k^\ast$.
\end{proposition}

\begin{proof} The extension $\widetilde{LG}(k)$ is the group of pairs $(g,\alpha)$ of $g\in \SL_n(K)$ and $\alpha: \widetilde{\det}_k(p^a \calO_K^n / g \calO_K^n)\cong k$, for any $a\ll 0$. This can be encoded in the map of $\E_1$-groups $\SL_n(K)\to Bk^\ast$ sending $g\in \SL_n(K)$ to the $1$-dimensional $k$-vector space $\widetilde{\det}_k(p^a \calO_K^n / g \calO_K^n)$. This can be refined to a map of $\E_1$-groups $\SL_n(K)\to K(\calO_K\on k)$ by sending $g$ to the perfect complex $p^a \calO_K^n / g \calO_K^n$.

The tame Hilbert symbol map $K_2(K)\to K_1(k) = k^\ast$ can be defined as the boundary map in the long exact sequence
\[
\ldots\to K_2(\calO_K)\to K_2(K)\to K_1(k)\to K_1(\calO_K)\to \ldots
\]
coming from the fibration sequence $K(\calO_K\on k)\to K(\calO_K)\to K(K)$ and the identification $K(\calO_K\on k)\cong K(k)$. Recall that Steinberg's extension comes from the map of $\E_1$-groups $\SL_n(K)\to \Omega K(K)$ sending any $g\in \SL_n(K)$ to the induced loop in $K(K)$. It remains to see that composing this map with the map $\Omega K(K)\to K(\calO_K \on k)$ from the fibration sequence induces the map of $\E_1$-groups $\SL_n(K)\to K(\calO_K\on k)$ considered above. This follows from the constructions.
\end{proof}

\begin{proposition} For any $m\geq 1$, the representation
\[
H^0(\Gr_{\SL_n},\calL^{\otimes m})
\]
of $\widetilde{LG}$ is infinite-dimensional. More precisely, for any projective subscheme $X\subset \Gr_{\SL_n}$, the restriction map
\[
H^0(\Gr_{\SL_n},\calL^{\otimes m}) \to H^0(X,\calL^{\otimes m})
\]
is surjective, and the right-hand side is infinite-dimensional once $\dim X>0$.
\end{proposition}

\begin{proof} If $X=(X_0)_\perf$ is the perfection of a positive-dimensional projective scheme $X_0$ over $k$ with an ample line bundle $\calL_0$ on $X_0$ with pullback $\calL$ to $X$, then
\[
H^0(X,\calL) = \varinjlim_{s\mapsto s^p} H^0(X_0,\calL_0^{\otimes p^k})
\]
is infinite-dimensional. If $Y_0\subset X_0$ is a closed subscheme with perfection $Y\subset X$, then the restriction map $H^0(X_0,\calL^{\otimes p^k})\to H^0(Y_0,\calL^{\otimes p^k}|_{Y_0})$ is surjective for $k$ large enough (by Serre vanishing), and thus $H^0(X,\calL)\to H^0(Y,\calL)$ is surjective. Applying these observations, one gets the result; the $n \geq 2$ assumption ensures that $\Gr_{\SL_n}$, and hence a suitable $X \subset \Gr_{\SL_n}$, is not zero-dimensional.
\end{proof}

We end this section with several questions.

\begin{question}
\begin{enumerate}
\item[{\rm (i)}] Can one construct an explicit nonzero section of $\calL$ (or some tensor power) on $\Gr_{\SL_n}$? This would give rise to divisors on $\Gr_{\SL_n}$, which are classically known as Theta-divisors.
\item[{\rm (ii)}] What can be said about the representation $H^0(\Gr_{\SL_n},\calL^{\otimes m})$ of $\widetilde{LG}$? Is it (topologically) irreducible? Classically, these representations are important in Kac-Moody theory and the Verlinde formula, cf. e.g. \cite{BeauvilleLaszlo}. Note that here, the representation is much bigger, as already the space of sections on any finite-dimensional part is infinite-dimensional.
\item[{\rm (iii)}] Is $\Pic(\Gr_{\SL_n}) = \Z[\frac{1}{p}]\cdot \calL$? Note that as $\Gr_{\SL_n}$ is a functor on perfect schemes, its Picard group is a $\Z[\frac{1}{p}]$-module.
\item[{\rm (iv)}] Let $\calG$ be a general split, simple and simply connected group. Is $\Pic(\Gr_\calG)\cong \Z[\frac{1}{p}]\cdot \calL_\calG$, with a specified generator $\calL_\calG$? One can get some multiple of the primitive line bundle via pullback from $\SL_n$, and the line bundle coming from the adjoint representation of $\calG$ should be the $2h^\vee$-th power of $\calL_\calG$, where $h^\vee$ is the dual Coxeter number. E.g., if $\calG=E_8$, the primitive line bundle should be a $60$-th root of the line bundle coming from the adjoint representation of $E_8$. Faltings, \cite{FaltingsPrimLineBundle}, constructed the primitive line bundle by constructing natural divisors (on the corresponding affine flag variety), related to orbits under $G(k[t^{-1}])$ in the equal characteristic case; this approach seems to break down in mixed characteristic.
\item[{\rm (v)}] Steinberg's extension exists for a general split, simple and simply connected group $G$, cf. e.g. \cite{BrylinskiDeligne}. Can one compare it with the central extension of $LG$ corresponding to the primitive line bundle?
\item[{\rm (vi)}] Is there a finite-type structure on $\Gr_{\SL_n}$, of some sort? Presumably, such a structure would give rise to a subrepresentation of $H^0(\Gr_{\SL_n},\calL^{\otimes m})$, so properties of this representation (such as irreducibility) may be relevant to this question.
\end{enumerate}
\end{question}

\section{$h$-descent for the derived category of quasi-coherent complexes}
\label{sec:hDescentDerived}

In this section, we investigate quasi-coherent sheaf theory on perfect schemes more thoroughly; these results complement those in \S \ref{VDescent} by extending them to the derived category (and, in fact, give new proofs of the results in \S \ref{VDescent} that do not rely on projective methods like the ones in Lemma~\ref{lem:VectcdhStack}), but are not used elsewhere in the paper. For convenience, we define:

\begin{definition}
	The $h$-topology on $\Perf$ is the topology generated by declaring perfectly finitely presented $v$-covers to be covers; thus, the perfection of an $h$-cover in $\Sch_{/\F_p}$ gives an $h$-cover in $\Perf$, and any $h$-cover in $\Perf$ is of this form, up to refinements (by Lemma \ref{ApproxUnivSubtr} and Lemma \ref{ApproxFP}).
\end{definition}

First, recall the warning in Remark \ref{rmk:NonDescentWarning}: it is difficult to extend the $h$-descent results for vector bundles in $\Perf$ to a larger class of quasi-coherent sheaves, as even flat quasi-coherent sheaves fail $h$-descent. Nevertheless, it turns out that this problem is specific to the abelian world, and disappears if we pass directly to the derived category. More precisely, in the language of \cite{LurieHTT}, one has the following descent results:

\begin{theorem} 
	\label{thm:hDescentDerived}
	Regard $D_{qc}(-)$, $\Perf(-)$, etc. as presheaves of spaces on $\Perf$. Then:
	\begin{enumerate}
		\item The functor $X \mapsto D_{qc}(X)$ gives an $h$-sheaf of spaces on $\Perf$.
		\item The functor $X \mapsto \Perf(X)$ gives a hypercomplete\footnote{The notion of a hypercomplete sheaf is specific to working in the $\infty$-categorical setting. Roughly speaking, a sheaf in the $\infty$-categorical setup is only required to satisfy descent along Cech covers, while a hypersheaf is required to satisfy descent along hypercovers. If the sheaf takes on $n$-truncated values for some finite integer $n$, then the notions coincide.} $v$-sheaf of spaces on $\Perf$. The same applies to $\Perf(W_n(X))$, $\Perf(W(X))$ and $\Perf(W(X) \on X)$.
		\item Let $\Perf^\fp\subset \Perf$ be the subcategory of all objects, and only perfectly finitely presented morphisms; this still carries the $h$-topology. The association $X\mapsto D^b_{qc}(X)$ is functorial in $X\in \Perf^\fp$, and defines an $h$-sheaf.
		\item Let $k$ be a perfect field, or more generally the perfection of a regular $\F_p$-algebra of finite Krull dimension, and give $\Perf^\fp_{/k}$ the induced $h$-topology. Then the functor $X \mapsto D_{qc}(X)$ is a hypercomplete $h$-sheaf on $\Perf^\fp_{/k}$.
	\end{enumerate}
\end{theorem}

	\begin{remark}
		\label{rmk:BoundedTorDimFFinite}
		In Theorem \ref{thm:hDescentDerived} (4), the same result follows formally if $k$ is the perfection of an algebra $R_0$ which can be written as a quotient of a regular $\F_p$-algebra of finite Krull dimension. Rings with this property include noetherian $F$-finite rings (by \cite[Remark 13.6]{Gabbertstructures}), and complete noetherian local rings (by the Cohen structure theorem).
	\end{remark}

\begin{remark}
	Theorem \ref{thm:hDescentDerived} (2) is a close analogue of $h$-descent for perfect complexes in the setting of derived schemes (see \cite{HalpernLeistnerPreygel} as well as Theorem \ref{thm:hCechDescentDerivedQCoh} below); here the $h$-topology on derived schemes is defined by passing to underlying classical schemes. However, the result above is stronger: in the derived setting, the functor $X \mapsto \Perf(X)$ does {\em not} give a hypercomplete sheaf. Indeed, if $X$ is any derived scheme with classical truncation $\tau(X)$, then the constant simplicial derived scheme with value $\tau(X)$ is a $h$-hypercover of $X$, but $\Perf(X) \neq \Perf(\tau(X))$ unless $X$ is classical. In fact, even if $X$ is classical, but not reduced, then the same argument applies to $X_\red\hookrightarrow X$.
\end{remark}

\begin{remark}
	Theorem \ref{thm:hDescentDerived} (4) gives a fully faithful inclusion $D^b_{qc}(X) \subset D^b(\Perf^\fp_{/X,h},\calO)$ for any $X \in \Perf^\fp_{/k}$. This leads to a new numerical invariant of bounded complexes $K \in D^b_{qc}(X)$ as follows: define the $h$-amplitude of such a $K$ to be the amplitude of $K \in D^b(\Perf^\fp_{/X,h},\calO)$. For example, any flat quasi-coherent $\calO_X$-module has $h$-amplitude $0$. On the other hand, there exist non-discrete complexes in $D^b_{qc}(X)$ with $h$-amplitude $0$: the complex $R j_* \calO_U$ in Remark \ref{rmk:NonDescentWarning} has $h$-amplitude $0$ as it becomes a flat quasi-coherent module after an $h$-cover. It might be interesting to study this invariant further in the context of local algebra.
\end{remark}

The proof of Theorem \ref{thm:hDescentDerived} takes up the rest of \S \ref{sec:hDescentDerived}. We begin in \S \ref{ss:PerfSCRDisc} by proving that any perfect simplicial commutative ring is discrete. In \S \ref{ss:DescendableMaps}, we study a property of maps of $E_\infty$-rings singled out recently by Akhil Mathew; the key result here is Proposition \ref{prop:ProperSurjectiveDescendable}, which shows that $h$-covers give descendable maps on cohomology (there is also a partial converse in Theorem \ref{thm:hCoverClassical}). The results in \S \ref{ss:PerfSCRDisc} and \S \ref{ss:DescendableMaps} are then applied in \S \ref{ss:hDescentComplexes} to prove Theorem \ref{thm:hDescentDerived} (1) - (3).  Theorem \ref{thm:hDescentDerived} (4) is then established in \S \ref{ss:hHyperDescentBounded}; the key result here is Proposition \ref{prop:BoundedTorDim}: perfectly finitely presented rings have finite global dimension. Finally, in \S \ref{ss:WittVectorCoh}, we use the discreteness of perfect simplicial commutative rings proven in \S \ref{ss:PerfSCRDisc} to improve and reprove an $h$-descent theorem for Witt vector cohomology from \cite{BBE}.

\subsection{Discreteness of perfect simplicial commutative rings}
\label{ss:PerfSCRDisc}

We work in the setting of derived algebraic geometry given by simplicial commutative rings over $\F_p$. Note that any such ring carries a canonical Frobenius endomorphism, so it makes sense to talk about perfect simplicial commutative $\F_p$-algebras. Our basic observation is that Frobenius kills higher homotopy:

\begin{proposition}
	\label{prop:FrobSCR}
If $A$ is perfect simplicial commtuative $\F_p$-algebra, then $A$ is discrete.
\end{proposition}
\begin{proof}
	Fix an integer $i > 0$, and set $A_i = \Sym_{\F_p}(\F_p[i])$. As an element of $\pi_i(A)$ is induced from the canonical element in $\pi_i(A_i)$ along a map $A_i \to A$, it suffices to show that $\Frob$ kills $\pi_i(A_i)$. For $i = 1$, we can write this ring as $\F_p \otimes_{\F_p[x]}^L \F_p$, which allows us to identify $\pi_1(A_1) \simeq (x)/(x^2)$ via a standard resolution. It is then clear that $\Frob$ kills $\pi_1(A_1)$. In general, we proceed by induction using the formula $A_{i+1} = \F_p \otimes^L_{A_i} \F_p$ to get an identification $\pi_{i+1}(A_{i+1}) \simeq \pi_{i+1}(\F_p \times_{A_i} \F_p) \simeq \pi_i(A_i)$ that is compatible with Frobenius.
\end{proof}

\begin{remark}
	The analogous cosimplicial statement is false: if $X$ is an ordinary elliptic curve over $\F_p$ and $Y = X_\perf$, then $R\Gamma(Y,\calO_Y)$ is a non-discrete complex that can be represented by a cosimplicial perfect $\F_p$-algebra. Indeed, the ordinarity ensures that $H^1(Y,\calO_Y) \neq 0$, while the desired cosimplicial presentation can be obtained via any Cech complex associated to an affine open cover.
\end{remark}

\begin{remark}
	\label{rmk:multSCR}
	Proposition \ref{prop:FrobSCR} can also be deduced from the following more general and more precise assertion, pointed out by Gabber: if $A$ is a simplicial commutative ring, then the multiplication $m:A \times A \to A$ induces the $0$ map on $\pi_i$ for $i > 0$. Gabber suggested an explicit simplicial proof, but we give a different argument here. To show this statement for all $A$, it is enough to check the universal case $A = \Sym(K)$ for $K = (\Z\oplus \Z)[i]$ with $i > 0$. For any (free) simplicial abelian group $K$, there is a natural commutative diagram (of simplicial sets) 
	\[ \xymatrix{ K \times K \ar[r] \ar[d] & K \otimes_{\Z} K \ar[d] & \\
	\Sym(K) \times \Sym(K) \ar[r] & \Sym(K) \otimes_{\Z} \Sym(K) \ar[r] & \Sym(K) }\]
	with the composite $K \times K \to \Sym(K)$ being induced by the multiplication map. For $K  = (\Z\oplus \Z)[i]$ with $i > 0$, one has $K \otimes_{\Z} K = \Z^4[2i]$, so $\pi_i(K \otimes_{\Z} K) = 0$. The diagram then shows that the multiplication map $A \times A \to A$ induces the $0$ map $\pi_i(K \times K) \simeq \pi_i(A \times A) \to \pi_i(A)$, which proves the claim.
\end{remark}

Passage to the perfection makes sense in the world of simplicial commutative $\F_p$-algebras, and is local for the \'etale topology. Consequently, there is a perfection functor $X \mapsto X_\perf$ on derived $\F_p$-schemes. The previous result then translates to:

\begin{corollary}
Let $X$ be a derived $\F_p$-scheme with underlying scheme $\tau(X) \hookrightarrow X$. Then $\tau(X)_\perf \simeq X_\perf$. In particular, $X_\perf$ is classical.
\end{corollary}
\begin{proof}
	The dual statement is that $A \to \pi_0(A)$ is an isomorphism after perfection for any simplicial commutative $\F_p$-algebra $A$, which was shown in Proposition \ref{prop:FrobSCR}.
\end{proof}

One consequence of the discreteness of perfect simplicial commutative $\F_p$-algebras is the coincidence of homotopy-colimits with naive ones:

\begin{lemma}
	\label{lem:PerfectColimits}
	The collection of perfect rings is closed under colimits in all $\F_p$-algebras. Moreover, any such colimit is automatically a homotopy-colimit in simplicial commutative $\F_p$-algebras. In particular, the perfection functor $A \mapsto A_\perf$ is cocontinuous on simplicial commutative $\F_p$-algebras.
\end{lemma}
\begin{proof}
Filtered colimits are easy to handle in all statements, so we reduce to cofibre coproducts. So say $C \gets A \to B$ is a diagram of $\F_p$-algebras. For the first statement, we must show that $B \otimes_A C$ is perfect if $A$, $B$, and $C$ are so. The Frobenius map on the tensor product is induced by passage to colimits from the Frobenius endomorphism of the diagram $C \gets A \to B$, so the claim is clear. For the second statement,  by the same argument, the simplicial commutative $A$-algebra  $B \otimes^L_A C$ is perfect, and hence discrete by Proposition \ref{prop:FrobSCR}. The last statement follows easily from these considerations.
\end{proof}

\begin{remark}
	This fails if $B$ and $C$ are perfect, but $A$ is not. For example, set $A$ to be an imperfect field, and $B = C = A_\perf$. Then $B \otimes_A C$ has non-trivial nilpotents.
\end{remark}

\subsection{Descendable maps of $E_\infty$-rings}
\label{ss:DescendableMaps}

The goal of this section\footnote{In this section, we depart from our standing conventions, and go ``fully'' derived. Thus, for an $E_\infty$-ring $A$ (which could be a discrete ring), the notation $\Mod(A)$ refers to the stable $\infty$-category of $A$-module spectra (and coincides with the usual derived category $D(A)$ when $A$ is discrete). Likewise, all tensor products are always derived. The one exception is that ``schemes'' refers to ordinary schemes; when we need derived schemes, we say so.} is to prove the following theorem, which gives derived $h$-descent for quasi-coherent complexes on noetherian schemes.

\begin{theorem}
	\label{thm:hCechDescentDerivedQCoh}
	Let $f:X \to S$ be an $h$-cover of noetherian schemes with derived Cech nerve $X^{\bullet/Y}$. Then $D_{qc}(S) \simeq \lim D_{qc}(X^{\bullet/S})$.
\end{theorem}

\begin{remark}
	Theorem \ref{thm:hCechDescentDerivedQCoh} applies to {\em all} quasi-coherent complexes, and thus specializes to a similar descent result for perfect or pseudo-coherent complexes (since the property of being perfect or pseudo-coherent can be detected\footnote{For detection of pseudo-coherence after passage to an $h$-cover, one argues as in Lemma \ref{lem:hCoverBounded} below.} after pullback along an $h$-cover for noetherian schemes); the latter was proven earlier in \cite[\S 4]{HalpernLeistnerPreygel} for all noetherian derived schemes. However,  Theorem \ref{thm:hCechDescentDerivedQCoh} does {\em not} extend to derived schemes. Indeed, let $A = \Sym_{\C}(\C[2])$.  Then $A \to \pi_0(A)$ gives an $h$-cover of derived schemes. If one had derived $h$-descent in this setting, then the base change functor $D(A) \to D(\pi_0(A))$ would be conservative. However, $K = A[u^{-1}]$, where $u \in \pi_2(A)$ is the generator, satisfies: (a) $K \neq 0$, and (b) $K \otimes_A \pi_0(A) = 0$.
\end{remark}

To prove this derived $h$-descent result for complexes, we will use a property of morphisms of rings recently singled out by Akhil Mathew \cite{MathewGaloisGroup} that implies descent for complexes. In fact, it is this property, and not Theorem \ref{thm:hCechDescentDerivedQCoh}, that plays an important role in the sequel.

Fix a {\em stable homotopy theory} $\calC$ in the sense of \cite{MathewGaloisGroup}, i.e., $\calC$ is a presentable symmetric monoidal stable $\infty$-category where the $\otimes$-product commutes with colimits in each variable; the main relevant example for us will be $\calC = D_{qc}(X)$ for a qcqs scheme $X$. Mathew studied the following class of maps in $\calC$:

\begin{definition}
	A map $A \to B$ in $\CAlg(\calC)$ is {\em descendable} if $\{A\} \to \{\Tot_{\leq n} B^\bullet \}$ is a pro-isomorphism of $A$-modules, where $B^\bullet$ is the derived Cech nerve of $A \to B$.
\end{definition}

A list of examples coming up in algebraic geometry is given in Lemma \ref{lem:DescendableSchemesBasics} below (see also Example \ref{ex:DescendableGorenstein}). Before studying this notion further, we record the main consequence of interest to us.

\begin{theorem}[{\cite[Proposition 3.21]{MathewGaloisGroup}}]
	\label{thm:DescendableDescent}
	If $\calC$ is a stable homotopy theory, and $A \to B$ is a descendable map in $\CAlg(\calC)$ with Cech nerve $B^{\bullet}$, then $A \simeq R\lim B^\bullet$, and  $\Mod_{\calC}(A) \simeq \lim \Mod_{\calC}(B^\bullet)$.
\end{theorem}

\begin{remark}
	Assume $\calC = D(\Ab)$ is the derived category of abelian groups. Fix a map $A \to B$ of discrete rings with Cech nerve $B^\bullet$. If $A \to B$ is descendable, then $A \simeq R\lim B^\bullet$ by Theorem \ref{thm:DescendableDescent}. It is tempting to guess the converse is also true. However, this is false: if $A = \Z_p$ and $B = \F_p$, then $A \to B$ is not descendable (as it is not so after inverting $p$), but $A \simeq R\lim B^\bullet$ by \cite[Theorem 4.4]{CarlssonDerComp}.
\end{remark}

In order to use this notion, we record some basic stability properties.

\begin{lemma}
\label{lem:DescendableBasics} Let $A \to B \to C$ be composable maps in $\CAlg(\calC)$.
\begin{enumerate}
\item If $A \to B$ and $B \to C$ are descendable, so is $A \to C$.
\item If $A \to C$ is descendable, so is $A \to B$. 
\end{enumerate}
\end{lemma}
\begin{proof}
	This is \cite[Proposition 3.23]{MathewGaloisGroup}.
\end{proof}

In applications, it will be useful to use the following more quantitative version:

\begin{definition}
A map $A \to B$ in $\CAlg(\calC)$ is {\em descendable of index $\leq m$} if $F^{\otimes_A m} \to A$ is null-homotopic, where $F$ is the fibre of $A \to B$.
\end{definition}

We give one example first.

\begin{example}
	\label{ex:DescendableGorenstein}
	Let $A = \Z[a,b,c,d]/(ab + cd - 1)$, and $B = A[\frac{1}{a}] \times A[\frac{1}{c}]$. Then $A \to B$ is faithfully flat, so $Q := B/A$ is flat. Let $F \simeq Q[-1]$ be the fibre of $A \to B$, so $F^{\otimes m} = Q^{\otimes m}[-m]$ is concentrated in degree $m$. Thus, the map $F^{\otimes m} \to A$ is classified by an element of $\Ext^m_A(Q^{\otimes m}, A)$. As $A$ is regular (it is smooth over $\Z$) of dimension $4$, it has global dimension $4$, so the relevant $\Ext$-group vanishes for $m \geq 5$, and hence $A \to B$ is descendable of index $\leq 5$. More generally, the same argument shows that if $A$ is any Gorenstein (noetherian) ring of dimension $d$, and $A \to B$ is faithfully flat, then $A \to B$ is descendable of index $\leq d+1$: indeed, $A$ has injective dimension $d$ as an $A$-module.
\end{example}

We begin our analysis by observing that this notion is compatible with the previous one:

\begin{lemma}
A map $A \to B$ in $\CAlg(\calC)$ is descendable if and only if it is descendable of index $\leq m$ for some $m \geq 0$. Moreover, any lax-monoidal cocontinuous functor preserves the property of being descendable of index $\leq m$.
\end{lemma}

\begin{proof}
	Let $F$ be the fibre of $A \to B$. Assume first that $A \to B$ is descendable. Then the map $F \to A$ becomes null-homotopic after applying $- \otimes_A B$. By \cite[Proposition 3.26]{MathewGaloisGroup}, it follows that the $m$-fold composition
	\[ F^{\otimes_A m} \to F^{\otimes_A m-1} \to \dots \to F \to A\]
	is null-homotopic for some large $m$, so $A \to B$ is descendable of index $\leq m$. Conversely, assume $A \to B$ is descendable of index $\leq m$, so $F^{\otimes_A m} \to A$ is null-homotopic. If $B^\bullet$ is the Cech nerve of $A \to B$, then the fibre of the map $\{A\} \to \{\Tot_{\leq n} B^\bullet\}$ of pro-objects is identified with $\{F^{\otimes_A n}\}$, where the transition maps multiply the first two factors; see \cite[Corollary 6.7]{CarlssonDerComp}. The assumption that $F^{\otimes_A m} \to A$ is null-homotopic then implies that this fibre is pro-zero: the $m$-fold transition map $F^{\otimes_A (m+i)} \to F^{\otimes_A i}$ is null-homotopic for each $i \in \N$; this proves the claim.
	
	Let $\Phi:\calC \to \calD$ be a lax-monoidal cocontinuous functor between stable homotopy theories, and assume that $A \to B$ in $\CAlg(\calC)$ is descendable of index $\leq m$; so, if $F$ is the fibre, then $F^{\otimes_A m} \to A$ is null-homotopic in $\Mod_A$. As $\Phi$ is exact, $\Phi(F)$ is the fibre of $\Phi(A) \to \Phi(B)$. As $\Phi$ is lax-monoidal and cocontinuous, one checks, using the bar-resolution construction of the tensor product of $A$-modules, that the induced functor $\Mod_A \to \Mod_{\Phi(A)}$ is also lax-monoidal. This gives, for each $i \in \N$, natural maps 
	\[\Phi(F)^{\otimes_{\Phi(A)} i} \to \Phi(F^{\otimes_A i})\] 
	whose composition with the canonical map $\Phi(F^{\otimes_A i}) \to \Phi(A)$ is the canonical map $\Phi(F)^{\otimes_{\Phi(A)} i} \to \Phi(A)$. Specializing to $i = m$ then shows that the canonical map $\Phi(F)^{\otimes_{\Phi(A)} m} \to \Phi(A)$ is null-homotopic,  proving the claim.
\end{proof}

An arbitrary filtered colimit of descendable maps is {\em not} descendable in general:

\begin{example}
	If $R$ is a discrete ring and $I \subset R$ is a locally nilpotent ideal that is not nilpotent (for example: $R = \C[x_1,x_2,x_3,\dots]/(x_1,x_2^2,x_3^3,\dots)$, and $I = (x_1,x_2,x_3,\dots)$), then $R \to R/I$ is not descendable: for any $m$, the multiplication map $I^{\otimes m} \to R$ has image $I^m$, and is thus non-zero by hypothesis. However, we can write $R/I = \colim R/J$ as a filtered colimit of $R$-algebras indexed by finitely generated subideals $J \subset I$. Each such $J$ is nilpotent, so $R \to R/J$ is descendable (by Lemma \ref{lem:DescendableSchemesBasics} below). Thus, descendability is not stable under filtered colimits.
\end{example}

Nevertheless, it turns out that filtered colimits of descendable maps of bounded index are descendable, provided the indexing category is not too large: 

\begin{lemma}
	\label{lem:DescendableLimits}
Let $I$ be a filtered category such that inverse limits over $I^\opp$ have finite cohomological dimension (in spectra). Fix $A \in \CAlg(\calC)$, and an $I$-indexed system $\{A_i\}$ in $\CAlg(\calC)_{A/}$. If each $A \to A_i$ is descendable of index $\leq m$ for some $m$ independent of $i$, then $A \to A_\infty = \colim A_i$ is descendable.
\end{lemma}
\begin{proof}
We give a proof when $I = \N$ (which is the only relevant case in the sequel), leaving the generalization to general $I$ to the reader. Let $F_i$ be the fibre of $A \to A_i$, so $F_i^{\otimes m} \to A$ is null-homotopic for all $i$. Let $F = \colim F_i$ be the fibre of $A \to A_\infty$. Then one knows
\[ \Map(F,A) \simeq \lim \Map(F_i,A),\]
which gives a short exact sequence
\[ 0 \to R^1 \lim \pi_1 \Map(F_i,A) \to \pi_0 \Map(F,A) \to \lim \pi_0(\Map(F_i,A)) \to 0.\]
The assumption tells us that the map $F^{\otimes m} \to A$ lies in the subspace 
\[ R^1 \lim \pi_1 \Map(F_i^{\otimes m}, A) \subset \pi_0 \Map(F^{\otimes m}, A).\]
It remains to observe that the map on $\pi_0$ induced by the obvious map
\[ \Map(F^{\otimes m},A) \times \Map(F^{\otimes m}, A) \to \Map(F^{\otimes 2m},A)\]
kills the $R^1 \lim$ terms (as it pairs them to an $R^2 \lim$, which always vanishes for $\N$-indexed towers). 
\end{proof}

We now collect some examples of descendable maps from algebraic geometry:

\begin{lemma}
\label{lem:DescendableSchemesBasics}
 Fix a map $f:X \to Y$ of qcqs schemes.
\begin{enumerate}
\item If $\{j_i:U_i \to X\}$ is a finite open cover of $X$ by qc opens,  then $\calO_X \to \prod_i R j_{i,\ast} \calO_{U_i}$ is descendable in $D_{qc}(X)$.
\item If $f$ is fppf, then $\calO_Y \to R f_* \calO_X$ is descendable in $D_{qc}(Y)$.
\item If $I \subset \calO_X$ is nilpotent and quasi-coherent, then $\calO_X \to \calO_X/I$ is descendable in $D_{qc}(X)$.
\item If $\calA \to \calB$ is descendable in $D_{qc}(X)$, then $R f_* \calA \to R f_* \calB$ is descendable in $D_{qc}(Y)$.
\end{enumerate}
\end{lemma}
\begin{proof}
	For (1), we may write $D_{qc}(X)$ as a finite limit of categories of the form $D_{qc}(U_{i_1} \cap ... \cap U_{i_n})$. By \cite[Proposition 3.24]{MathewGaloisGroup}, we reduce to the case where $X = U_i$ for some $i$, whence the claim is clear by Lemma \ref{lem:DescendableBasics} (2). For (2), using (1), as lax-monoidal cocontinuous functors preserve descendability, we reduce to the case where $X$ and $Y$ are affine, and then it follows from \cite[Proposition 3.31]{MathewGaloisGroup}; (3) similarly follows from \cite[Proposition 3.33]{MathewGaloisGroup}. (4) again follows from the preservation of descendability under lax-monoidal cocontinuous functors.  
\end{proof}

\begin{remark}
Mathew shows that if $A \to B$ is a faithfully flat map of discrete rings, and $B$ is countably generated as an $A$-algebra, then $A \to B$ is descendable. We do not know if the countable generation assumption is necessary.
\end{remark}

The main ``new'' class of descendable morphisms relevant for us is:

\begin{proposition}
\label{prop:ProperSurjectiveDescendable}
Let $f:X \to Y$ be an $h$-cover of noetherian schemes. Then $\calO_Y \to R f_* \calO_X$ is descendable.
\end{proposition}

The proof below is adapted from an argument due to  Akhil Mathew for the case of finite morphisms; our previous argument relied on almost mathematics and was restricted to perfect schemes. 

\begin{proof}
	We freely use Lemma \ref{lem:DescendableBasics} (2) to replace $f$ by further covers when necessary. Using Lemma \ref{lem:DescendableSchemesBasics}, it is enough to handle the case where $f$ is proper surjective. Assume that $f$ is of inductive level $\leq n$ for some $n \geq 0$ (by Proposition \ref{ProperGood}). We prove the claim by induction on $n$. If $n = 0$, then, after refinements, $f$ factors as a composition of an fppf map and a nilpotent closed immersion, so the claim follows from Lemma \ref{lem:DescendableSchemesBasics} (2), (3), (4) and Lemma \ref{lem:DescendableBasics} (1). In general, after refinements, we have a factorization $X \to X' \to Y' \to Y$ where $X \to X'$ is proper and fppf,  $Y' \to Y$ is a nilpotent closed immersion, and $X' \to Y'$ is an isomorphism outside a closed subset $Z' \subset Y'$ such that $X' \times_{Y'} Z' \to Z'$ is of inductive level $\leq n-1$.  By the argument used in the $n=0$ case, it is enough to prove the statement for $X' \to Y'$. We now rename $X = X'$, $Y = Y'$, and $Z = Z'$ for convenience. By Lemma \ref{lem:DescendableSchemesBasics} (1), (3) and (4), the property of descendability can be checked locally on $Y$, so we may assume $Y = \Spec(A)$ is affine. Write $U = Y - Z$, and let $I \subset A$ be an ideal defining the closed subset $Z \subset Y$. Then the fibre $F$ of $\calO_Y \to R f_* \calO_X$ is a $\calO_Y/I^n$-complex for $n \gg 0$; here we use that $F$ has coherent cohomology groups, is bounded, and is trivial over $\Spec(A) - V(I)$ by flat base change for coherent cohomology. After replacing $I$ with this power, we assume $F$ admits the structure of an $\calO_Y/I$-complex.  By induction, we know that $\calO_Z \to R f_* \calO_{f^{-1}Z}$ is descendable. So there is some $k \gg 0$ such that the composite $F^{\otimes k} \to \calO_Y \to \calO_Z$ is null-homotopic, and hence that $F^{\otimes k} \to \calO_Y$ factors as $F^{\otimes k} \to I \to \calO_Y$. Then the canonical map $F^{\otimes k+1} \to \calO_Y$ factors as 
	\[F^{\otimes k+1} \simeq F^{\otimes k} \otimes F \stackrel{a}{\to} I \otimes F \stackrel{b}{\to} F \stackrel{c}{\to} \calO_Y\]
where $a$ is induced from the previous factorization, $b$ is the multiplication map, and $c$ is the canonical map. The map $b$ is null-homotopic as $F$ admits an $\calO_Y/I$-structure, so the above composition is also null-homotopic.
\end{proof}

\begin{proof}[Proof of Theorem \ref{thm:hCechDescentDerivedQCoh}]
	We first record some generalities. Let $\calF := D_{qc}(-)$, viewed as a presheaf of spaces on derived schemes. In this proof, we will say that a map $f:X \to Y$ of (always qcqs) derived schemes is good if $f$ is of universal $\calF$-descent, i.e., that $\calF(Y) \simeq \lim \calF(X^{\bullet/Y})$, and the same is true after arbitrary base change on $Y$. Then, by Lemma \ref{DevissageLemma}, we have: (a) the class of good maps is stable under composition and base change, and (b) if a composite $X \to Y \to Z$ is good, then so is $Y \to Z$. Moreover, if $f:X \to Y$ is an affine map with $\calO_Y \to Rf_* \calO_X$ is descendable, then $f$ is good by Theorem \ref{thm:DescendableDescent}.

Now let $f:X \to S$ be an $h$-cover of noetherian schemes. We must check that $f$ is good. This can be checked locally on $S$, so we may assume that $S$ is affine. Choose a Zariski cover $j:U \to X$ such that $U$ is affine, and set $g:U \to S$. Then $g$ is affine and an $h$-cover of noetherian schemes, so $\calO_S \to Rg_* \calO_U$ is descendable, and thus $g$ is good by the last sentence of the previous paragraph. It then follows that $f$ is good, proving the claim.
\end{proof}

	Using Proposition \ref{prop:ProperSurjectiveDescendable}, one can classify descendable maps in cases of interest in algebraic geometry:

	\begin{theorem}
		\label{thm:hCoverClassical}
		Let $A \to B$ be a finitely presented map of noetherian discrete rings. Then $\Spec(B) \to \Spec(A)$ is an $h$-cover if and only if $A \to B$ is descendable.
	\end{theorem}
	\begin{proof}
		The ``only if'' direction follows from Proposition \ref{prop:ProperSurjectiveDescendable}. Conversely, assume that $A \to B$ is a descendable map of connective $E_\infty$-rings. We will check that $\Spec(B) \to \Spec(A)$ is a topological quotient; this suffices to prove the claim by the well-known (see \cite[\S 3]{VoevodskyHTop}) characterization of $h$-covers of noetherian schemes as universal topological quotients (and the stability of descendability under base change). To prove this, note that, by \cite[Theorem 1.10]{BhattHL}, the poset of quasi-compact open subsets of $\Spec(A)$ can be identified with the $\infty$-category $\calC_A$ of eventually connective compact localizations of $A$ in $\CAlg(D(A))$, i.e., $\calC_A \subset \CAlg(D(A))$ is the full subcategory of those $C$ which are bounded above as $A$-complexes, compact as commutative $A$-algebras, and satisfy $C \otimes_A C \simeq C$ via the multiplication map. If we write $B^\bullet$ for the Cech nerve of $A \to B$, then descendability gives an identification $D(A) \simeq \lim D(B^\bullet)$. This equivalence is symmetric monoidal, and thus induces $\CAlg(D(A)) \simeq \lim \CAlg(D(B^\bullet))$. It is easy to see\footnote{The only non-trivial bit is to check descent of compactness, i.e., show: given $C \in \CAlg(D(A))$ satisfying $C \otimes_A C \simeq C$, if $C \otimes_A B$ is compact in $\CAlg(D(B))$, then $C$ is compact in $\CAlg(D(A))$. For this, one first notes that the functor $\Map_A(C,-)$ on $\CAlg(D(A))$ takes on discrete values since $C \otimes_A C \simeq C$: in fact, the values are always contractible or empty. The same also applies to $C \otimes_A B^i \in \CAlg(D(B^i))$ for all $i$. The descent of compactness now follows as totalization of cosimplicial $n$-truncated spaces commute with filtered colimits for any finite $n$.} that this identification induces an equivalence $\calC_A \simeq \lim \calC_{B^\bullet}$. Now say $V \subset \Spec(A)$ is a subset whose inverse image $U \subset \Spec(B)$ is a quasi-compact open; we must check that $V$ is a quasi-compact open. By construction, we see that $p_1^{-1}(U) = p_2^{-1}(U)$ as quasi-compact open subsets of $\Spec(B \otimes_A B)$, where $p_1,p_2:\Spec(B \otimes_A B) \to \Spec(B)$ are the two projection maps. But then $U$ defines an object of $\lim \calC_{B^\bullet}$, which thus comes from one in $\calC_A$, i.e., $U$ is the inverse image of some quasi-compact open $W \subset \Spec(A)$. Since $\Spec(B) \to \Spec(A)$ is surjective, it follows that $W = V$, so $V$ is a quasi-compact open subset.
	\end{proof}

\subsection{$h$-descent for complexes}
\label{ss:hDescentComplexes}

We now discuss quasi-coherent complexes in the $h$-topology on perfect schemes; here the latter is defined as the topology generated by perfections of fppf covers and finitely presented proper surjections. The key descent result is:

\begin{theorem}
	\label{thm:hCoversDescendable}
Let $f:X \to Y$ be an $h$-cover in $\Perf$. Then $\calO_Y \to R f_* \calO_X$ is descendable.
\end{theorem}
\begin{proof}
By approximation, devissage, and (derived) base change, it is enough to prove the following: if $f_0:X_0 \to Y_0$ is either an fppf or a proper surjective map of noetherian schemes, then the map $f:X \to Y$ on perfections induces a descendable map $\calO_Y \to R f_* \calO_X$.
	
	For the fppf case: we know that $\calO_{Y_0} \to R f_{0,\ast} \calO_{Y_0}$ is descendable of some index $\leq m$ by Lemma \ref{lem:DescendableSchemesBasics} (2). Write $X_n \to X_0$ and $Y_n \to Y_0$ for the $n$-fold Frobenius maps, and let $g_n:W_n := X_n \times_{Y_n} Y \to Y$ be the base change. Then we know that $X \simeq \lim W_n$. Now each $\calO_Y \to R g_{n,\ast} \calO_{W_n}$ is descendable of index  $\leq m$ by base change. Since $m$ is independent of $n$, the claim follows from Lemma \ref{lem:DescendableLimits}.
	
	For the proper case, one argues exactly as in the fppf case using Proposition \ref{prop:ProperSurjectiveDescendable} instead of Lemma \ref{lem:DescendableSchemesBasics} (2); the only difference is that the fibre product defining the scheme $W_n$ appearing above must be replaced with the derived fibre product (to get base change for the index of descendability); by Proposition \ref{prop:FrobSCR}, the limit $\lim W_n$ is still simply $X$, so the previous argument goes through. 
\end{proof}

We can now prove most of Theorem \ref{thm:hDescentDerived}:

\begin{proof}[Proof of Theorem  \ref{thm:hDescentDerived} (1), (2)]
	Part (1) is proven exactly as Theorem \ref{thm:hCechDescentDerivedQCoh} using Theorem \ref{thm:hCoversDescendable} instead of Proposition \ref{prop:ProperSurjectiveDescendable}.

	For (2), first note the functor $X \mapsto \Perf(X)$ is an $h$-sheaf of spaces by (1). Indeed, if $f:X \to S$ is an $h$-cover in $\Perf$ with Cech nerve $X^{\bullet/S} \to S$, then (1) gives $D_{qc}(S) \simeq \lim D_{qc}(X^{\bullet/S})$. Here the limit on the right is also a limit of symmetric monoidal $\infty$-categories, so passage to the dualizable objects shows $\Perf(S) \simeq \lim \Perf(X^{\bullet/S})$, i.e., $\Perf(-)$ is an $h$-sheaf of spaces. Now, for integers $a \leq b$ and a scheme $X$, let $\Perf^{[a,b]}(X) \subset \Perf(X)$ be the full $\infty$-category spanned by perfect complexes with $\Tor$-amplitude contained in $[a,b]$, i.e., those perfect complexes $K$ that can, locally on $X$, be represented by a complex of finite projective modules located in degrees between $a$ and $b$. Then one can check:
	\begin{enumerate}
		\item $\Perf^{[a,b]}(-)$ gives a presheaf of $(b-a)$-truncated spaces on $\Perf$: for any $X \in \Perf$ and any two $K,L \in \Perf^{[a,b]}(X)$, 
	\[ \Map(K,L) := \tau^{\leq 0} R\Hom(K,L) \simeq \tau^{\leq 0} R\Gamma(X,K^\vee \otimes L),\]
	and the latter space is $(b-a)$-truncated since $K^\vee \otimes L \in D^{\geq -b+a}_{qc}(X)$.
		\item $\Perf(-) = \colim \Perf^{[-n,n]}(-)$ as presheaves on $\Perf$.
		\item Membership in $\Perf^{[a,b]}(X) \subset \Perf(X)$ can be detected $v$-locally. In fact, one has a stronger statement: if $K \in \Perf(A)$ for a ring $A$, then $K \in \Perf^{[a,b]}(A)$ provided $K \otimes_A k \in D^{[a,b]}(k)$ for every residue field $k$ of $A$. To see this, one can reduce to $A$ being noetherian (by (4) below) and local, and then use minimal free resolutions.
		\item Both $\Perf(-)$ and $\Perf^{[a,b]}(-)$ commute with filtered colimits of rings.
	\end{enumerate}
	Since $\Perf(-)$ is an $h$-sheaf on $\Perf$, it follows from (a) and (c) that $\Perf^{[a,b]}(-)$ gives an $h$-sheaf of truncated spaces on $\Perf$ as well. By (d) and Lemma \ref{ApproxUnivSubtr}, it follows that $\Perf^{[a,b]}(-)$ is a $v$-sheaf of truncated spaces; here we use that totalizations of truncated cosimplicial spaces commutes with filtered colimits. Also, by truncatedness, this $v$-sheaf is automatically hypercomplete. In other words, if $X^\bullet \to S$ is a $v$-hypercover in $\Perf$, then $\Perf^{[a,b]}(S) \simeq \lim \Perf^{[a,b]}(X^\bullet)$. By working locally on $S$ and calculating $\Map(\calO_S,\calO_S)$ using this equivalence, it follows that $\Perf(S) \to \lim \Perf(X^\bullet)$ is fully faithful as well. Finally, if $K^\bullet \in \lim \Perf(X^\bullet)$, then there exist integers $a \leq b$ such that $K^0 \in \Perf^{[a,b]}(X^0) \subset \Perf(X^0)$. But then $K^i \in \Perf^{[a,b]}(X^i)$ for each $i$ since $K^i$ is a pullback of $K^0$ along some simplicial structure map $X^i \to X^0$. Thus, $K \in \lim \Perf^{[a,b]}(X^\bullet) \simeq \Perf^{[a,b]}(S) \subset \Perf(S)$; this proves (2) for $\Perf(X)$. It is then easy to deduce the same result for $\Perf(W_n(X))$, $\Perf(W(X))$ and $\Perf(W(X) \on X)$.
\end{proof}

We end by noting a corollary (or, really, an equivalent form) of this $h$-descent result, extending Lemma \ref{lem:VectcdhStack} to the derived category:

\begin{corollary}
	Let $f:X \to Y$ be a proper surjective finitely presented map in $\Sch_{/\F_p}$ which is an isomorphism over a quasi-compact open $U \subset Y$. If $Z \subset Y$ is the complement of $U$ and $E := f^{-1}(Z) \subset X$, then pullback induces an equivalence of $\infty$-categories
	\[ D_{qc}(Y_\perf) \simeq D_{qc}(X_\perf) \times_{D_{qc}(E_\perf)} D_{qc}(Z_\perf),\]
	and thus a similar statement for the corresponding $\infty$-category of perfect complexes.
\end{corollary}
\begin{proof}
	This follows from Theorem \ref{thm:hDescentDerived} (1) and Theorem \ref{CritHSheaf} applied to $F(X) = D_{qc}(X_\perf)$; the statement about perfect complexes follows immediately by passage to dualizable objects.
\end{proof}

\begin{proof}[Proof of Theorem  \ref{thm:hDescentDerived} (3)] Next, we establish part (3) of Theorem \ref{thm:hDescentDerived}, i.e.~$X\mapsto D^b_{qc}(X)$ is functorial in $X\in \Perf^\fp$, and gives an $h$-sheaf of spaces. Note that the functoriality is a bit surprising from the perspective of ``classical'' algebraic geometry: bounded complexes often become unbounded after pullback along non-flat maps. This phenomenon does not occur in the perfect setting thanks to the following result.

\begin{proposition}\label{prop:FiniteTorDim} Let $R\to S$ be a perfectly finitely presented map of perfect $\F_p$-algebras. Then $S$ is of finite Tor-dimension over $R$.
\end{proposition}

\begin{proof} We will show that if $S$ is the perfection of $R[X_1,\ldots,X_n] / (f_1,\ldots,f_m)$, then the Tor-dimension of $S$ over $R$ is bounded by $m$. To see this, we may replace $R$ by the perfection of $R[X_1,\ldots,X_n]$ to assume $n=0$. Moreover, by induction, we may assume that $m=1$. Then $S= R  / f^{1/p^\infty} R$, so it is enough to see that $f^{1/p^\infty} R$ is flat. But the proof of Lemma \ref{NoTorPerf} shows that
\[
f^{1/p^\infty} R = \colim_{f^{1/p^n - 1/p^{n+1}}} R\ ,
\]
which is a filtered colimit of flat $R$-modules, and thus flat.
\end{proof}

This directly implies that $X\mapsto D^b_{qc}(X)$ is functorial in $X\in \Perf^\fp$. To see that it is an $h$-sheaf, it is enough to observe that we already have $h$-descent for $D_{qc}(X)$, and that the following lemma holds true.

\begin{lemma}
		\label{lem:hCoverBounded}
		Let $f:X \to Y$ be an $h$-cover in $\Perf$. Given $K \in D_{qc}(Y)$, if $f^* K \in D_{qc}(X)$ is bounded, so is $K$.
	\end{lemma}
	\begin{proof}
This is clear if $f$ is the perfection of an fppf map. Thus, by devissage, it is enough to prove this for $f$ being the perfection of a proper surjective map of inductive level $\leq n$. We work by induction on $n$. If $n = 0$, then, after refinements, $f$ is the perfection of an fppf map, so the claim is clear. For general $n$, we thus reduce to the case where $f$ is perfectly proper surjective, an isomorphism outside a closed subset $Z \subset Y$ with preimage $E \subset X$, and the restriction $f_Z:E \to Z$ of $f$ is of inductive level $\leq n-1$. Then we have a pullback diagram
		\[ \xymatrix{ \calO_Y \ar[d] \ar[r] & R f_* \calO_X \ar[d] \\
		\calO_Z \ar[r] & R f_* \calO_E.}\]
Tensoring this diagram with $K$, and using the projection formula, gives a pullback diagram
			\[ \xymatrix{ K \ar[d] \ar[r] & R f_* (f^* K) \ar[d] \\
				K|_Z \ar[r] & R f_* (f^*K|_E) \simeq R f_{Z,*} f_Z^*(K|_Z).}\]
				Now the terms on the right are bounded since $f^* K$ is bounded (as derived pushforward along qcqs maps has finite cohomological dimension). By induction applied to $f_Z$, the complex $K|_Z$ is also bounded. The preceding square then shows that $K$ is bounded.
	\end{proof}
\end{proof}

\subsection{$h$-hyperdescent for complexes}
\label{ss:hHyperDescentBounded}

In this section, we establish hyperdescent for quasi-coherent complexes, i.e., prove Theorem \ref{thm:hDescentDerived} (4). So fix a ring $k$ that is the perfection of a regular $\F_p$-algebra of finite Krull dimension.  We start by observing a general result on finite global dimension in this setup.

	\begin{proposition}
		\label{prop:BoundedTorDim}
	Let $R$ be a perfectly finitely presented $k$-algebra. Then $R$ has finite global dimension, i.e., there exists some positive integer $N$ such that each $R$-module has projective dimension $\leq N$.
	\end{proposition}
	\begin{proof}
		Choose a perfect polynomial ring $P$ over $k$ equipped with a surjection $P \to R$. Then $R \otimes_P^L R \simeq R$ by Lemma \ref{lem:PerfectColimits}. Thus, for any $R$-module $M$, we have 
	\[ M \simeq M \otimes^L_R (R \otimes^L_P R) \simeq M \otimes^L_P R.\]
	Thus, it suffices to prove the result for $P$, which is the perfection of a regular $\mathbb F_p$-algebra $P_0$ of finite Krull dimension. Let $d = \dim(P)$. Write $P_0 \to P_n$ for the $n$-fold Frobenius map, so there are natural $P_0$-algebra maps $P_n \to P_{n+1}$ given by Frobenius, and $P \simeq \colim P_n$. Now, for any $P$-module $M$, we can regard $M$ as a $P_n$-module via restriction along $P_n \to P$, and write $F_n := M \otimes_{P_n} P$ for its base change back to $P$. Then we have natural maps $F_n \to M$ compatibly in $n$, and this leads to
	\[ M \simeq \colim F_n.\]
	Now each $F_n$ has projective dimension $\leq d$ since $P_n$ is a regular ring of dimension $d$, and $P_n \to P$ is flat. Writing an $\N$-indexed colimit as the cone of a map between direct sums then shows that $M$ has projective dimension $\leq d+1$, so setting $N = d+1$ solves the problem.
	\end{proof}

\begin{remark}
	Proposition \ref{prop:BoundedTorDim} implies that for any $X \in \Perf^\fp_{/k}$, each $K \in D^b_{qc}(X)$ has projective amplitude contained in $[a,b]$ for suitable integers $a \leq b$, i.e.,  $K$ can be represented, on affines $\Spec(A) \subset X$,  by a complex $E^\bullet$ of projective $A$-modules with $E^i = 0$ for $i \not\in [a,b]$ (see \cite[Tag 0A5M]{StacksProject}).
\end{remark}

\begin{remark}
	\label{rmk:BoundedTorDim}
	Proposition \ref{prop:BoundedTorDim} implies that for any perfectly finitely presented $k$-algebra $R$, the $\Tor$-dimension of any $R$-module is bounded above by some fixed integer $N$. We will show\footnote{An earlier version of this preprint asserted that one may choose $N = d$. This is false: if $R = k[x,y]/(xy)_{\perf}$, then $d = 1$, but $\Tor_2^R(R/(x), R/(y)) \neq 0$. We thank Gabber for pointing out the mistake, the previous example, and the correct bound. Gabber has also pointed out that the global dimension $N$ in Proposition~\ref{prop:BoundedTorDim} is $\leq 2d+1$, where equality is obtained in some examples; we do not prove that here.} that, in fact, one can choose $N = 2d$, where $d := \dim(R)$. We work by induction on $d$. We can assume that $R$ is the perfection of a noetherian complete local ring. If $d = 0$, there is nothing to prove as $R$ is a product of fields. If $d > 0$, choose a Noether normalization of $R$, i.e., a finite injective map of noetherian rings $P_0 \to R_0$ with $P_0$ regular, and $R = R_{0,\perf}$; set $P = P_{0,\perf}$. Then it is easy to see that the $\Tor$-dimension of any $P$-module is bounded above by $d = \dim(P_0) = \dim(R)$. Moreover, a standard argument\footnote{Indeed, after possibly adjusting our initial choice of $P_0$ and $R_0$, the map $P_0 \to R_0$ is generically \'etale, so there exist some non-zero $f \in P_0$ such that $P_0 \to R_0$ is finite \'etale after inverting $f$. One then checks that such an $f$ does the job.} implies that there exists some non-zero $f \in P$ such that $P \to R$ is almost finite \'etale with respect to the ideal $I = (f^{\frac{1}{p^\infty}}) \subset P$; note that the ideal $I$ is flat and satisfies $I^2 = I$, so almost mathematics (\cite{GabberRamero}) with respect to $I$ makes sense. In particular, the multiplication map $R \otimes^L_P R \to R$ is an almost direct summand as a $R \otimes^L_P R$-complex by almost \'etaleness. (In fact, $R \otimes_P^L R \simeq R \otimes_P R$ by Lemma \ref{lem:PerfectColimits}.) Given an $R$-module $M$, we may view $M$ as an $R \otimes_P^L R$-module via the multiplication map, and $M \otimes_P^L R$ as an $R \otimes_P^L R$-complex in the natural way; applying $\big(M \otimes_P^L R\big) \otimes_{R \otimes_P^L R}^L (-)$ to the preceding retraction shows that $M$ is an almost direct summand of $M \otimes_P^L R$ as an $R \otimes_P^L R$-complex. Using the second factor inclusion $R \hookrightarrow R \otimes_P^L R$, this implies that any $R$-module $M$ is an almost direct summand of a complex of the form $N \otimes_P^L R$, where $N$ is some $P$-module, and the $R$-module structure comes from the second factor. In particular, since the $\Tor$-dimension of $N$ is bounded above by $d$, the same is true for $M$ in the almost world. To pass back to the non-almost world, set $J = IR \subset R$, so both $I$ and $J$ are flat ideals in the corresponding rings; here we use the proof of Lemma \ref{NoTorPerf} to get the flatness of $J$. The preceding almost vanishing (and the flatness of $J$) implies that $I \otimes_P^L M \simeq J \otimes_R^L M \simeq J \otimes_R M$ has $\Tor$-dimension $\leq d$; here we use that $I \otimes_P^L (-)$ kills almost zero objects in $D(P)$ since $I \otimes_P^L P/I \simeq I \otimes_P P/I \simeq I/I^2 = 0$. We now have canonical short exact sequences
\[ 0 \to \Tor_1^R(R/J,M) \to J \otimes_R M \to JM \to 0\]
and 
\[ 0 \to JM \to M \to M/JM \to 0.\]
By induction, both $M/JM$ and $\Tor_1^R(R/J,M)$ have $\Tor$-dimension $\leq 2d-2$ over $R/J$. Now $R/J$ has $\Tor$-dimension $\leq 1$ over $R$ as $J$ is $R$-flat; the same then holds for any flat $R/J$-module by Lazard's theorem. By taking flat resolutions, it follows that both $M/JM$ and $\Tor_1^R(R/J,M)$ have $\Tor$-dimension $\leq 2d-1$ over $R$. The first exact sequence then shows that $JM$ has $\Tor$-dimension $\leq 2d$ over $R$; the second one then implies that $M$ has $\Tor$-dimension $\leq 2d$ as well. 
\end{remark}

\begin{remark}
The homological properties of perfections of noetherian $\F_p$-algebras established above have been investigated previously in the commutative algebra literature to some extent. For example, \cite[Theorem 3.1]{AberbachHochster} establishes that reduced quotients of perfectly finitely presented rings have $\Tor$-dimension $\leq d$, which can be deduced from Lemma~\ref{NoTorPerf}. Likewise, \cite[Theorem 1.2]{MohsenHomological} proves that such quotients which are domains have a finite resolution by countably generated projectives, which follows immediately from the proof of Proposition \ref{prop:BoundedTorDim}.
\end{remark}

	Using Proposition \ref{prop:BoundedTorDim}, one can reprove Kunz's theorem characterizing regularity.

	\begin{corollary}[Kunz]
		\label{cor:Kunz}
		Let $R$ be a noetherian $\F_p$-algebra. Then $R$ is regular if and only if the absolute Frobenius $R \to R$ is flat.
	\end{corollary}
	\begin{proof}
		The forward direction is well-known, and we have nothing new to offer. Instead, we explain the converse: if Frobenius is flat, then $R$ is regular. This can be checked after completion at points, so we may assume $R$ is a complete noetherian local $\F_p$-algebra; here we use that the hypothesis on flatness of Frobenius passes to completions of $R$. By Remark \ref{rmk:BoundedTorDimFFinite}, every $R_\perf$-module has finite $\Tor$-dimension.  The faithful flatness of $R \to R_\perf$ then implies that every $R$-module has finite $\Tor$-dimension, so $R$ is regular.
	\end{proof}

To prove that quasi-coherent complexes give hypercomplete $h$-sheaves, our basic tool will be the following:

\begin{lemma}
	\label{lem:CommuteTotTensor}
	Let $R$ be the perfection of a finitely presented $k$-algebra. Let $M \in D(R)$, and let $N^\bullet$ be a cosimplicial $R$-module such that $R\lim N^\bullet$ is bounded. Then the natural map gives
\begin{equation}
\label{eq:CommuteTotTensor1}
M \otimes_R^L R\lim N^\bullet \simeq R\lim(M \otimes_R^L N^\bullet).
\end{equation}
\end{lemma}

The proof below applies to any ring of finite global dimension.

\begin{proof}
	Let $M = R\lim \tau_{\leq i} M$ be the Postnikov tower for $M$. For any $K \in D^b(R)$, we first observe that the canonical map
	\[ M \otimes_R^L K \simeq \big(R\lim \tau_{\leq i} M\big) \otimes_R^L K \to R\lim(\tau_{\leq i} M \otimes_R^L K) \]
	is an equivalence. Indeed, by Proposition \ref{prop:BoundedTorDim}, this reduces to the case where $K$ is given by a single projective $R$-module, which is clear (by calculating homology of either side). Applying this observation to $K = R\lim N^\bullet$ (which is permissible, since $R\lim N^\bullet$ is bounded by assumption), we obtain
\begin{equation}
	\label{eq:CommuteTotTensor2}
M \otimes_R^L R\lim N^\bullet \simeq R\lim \big(\tau_{\leq i} M \otimes_R^L R\lim N^\bullet).
\end{equation}
	Similarly, applying this observation to $K = N^\bullet$ shows that 
	\begin{equation}\label{eq:CommuteTotTensor3}
M \otimes_R^L N^\bullet \simeq R\lim\big(\tau_{\leq i} M \otimes_R^L N^\bullet)
\end{equation}
as cosimplicial $A$-complexes. Plugging equations \eqref{eq:CommuteTotTensor2} and \eqref{eq:CommuteTotTensor3} into the two sides of equation \eqref{eq:CommuteTotTensor1},  and commuting limits, we reduce to the case where $M = \tau_{\leq i} M$ is bounded below.  Assume first that $M$ is, in fact, bounded. Then, using Proposition \ref{prop:BoundedTorDim}, we reduce to the case where $M$ is a projective $R$-module, which is straightforward (via Dold-Kan). To pass to the general case, it is enough to show that for fixed $k$, the maps
\[ a_i:\tau^{\leq i} M \otimes_R^L R\lim N^\bullet \to M \otimes_R^L R\lim N^\bullet \]
and
\[ b_i:R\lim( \tau^{\leq i} M \otimes_R^L N^\bullet) \to R\lim(M \otimes_R^L N^\bullet)\]
(induced by the canonical map $\tau^{\leq i} M \to M$) induce isomorphisms on $H^k$ for $i$ sufficiently large. For $a_i$, this follows by observing that the functor $- \otimes_R^L R\lim N^\bullet$ has bounded homological dimension (i.e., carries $D^{> i}(R)$ to $D^{>i - c}(R)$ for some fixed $c$) by Proposition \ref{prop:BoundedTorDim} since $R\lim N^\bullet$ is bounded below (in fact, it is bounded). For $b_i$, it suffices to observe that the functor $R\lim(- \otimes_R^L N^\bullet)$ has bounded homological dimension by Proposition \ref{prop:BoundedTorDim} since each $N^i$ is an $R$-module, and hence bounded below as a complex.
\end{proof}

Using the previous lemma, we immediately see that quasi-coherent complexes are hypercomplete:  

\begin{lemma}
	\label{lem:QCohComplexHypercomplete}
	For any $Y \in \Perf^\fp_{/k}$ and $K \in D_{qc}(Y)$, the functor $(f:X \to Y) \mapsto R\Gamma(X, f^* K)$ gives a hypercomplete $h$-sheaf of spectra on $\Perf^\fp_{/Y}$.
\end{lemma}
\begin{proof}
Let $f:X^\bullet \to Y$ be an $h$-hypercover. We must check that pullback induces an equivalence
	\[ R\Gamma(Y,K) \simeq R\lim R\Gamma(X^\bullet, f^* K).\]
	By Zariski hyperdescent, we may assume $Y = \Spec(A)$ and $X^\bullet$ are all affine, and that $K$ corresponds to some $M \in D(A)$. By the projection formula, we are reduced to showing that
	\[ M \simeq R\lim \big(R\Gamma(X^\bullet, \calO_{X^\bullet}) \otimes^L_A M)\]
	via the canonical map. This follows from Lemma \ref{lem:CommuteTotTensor} applied to $N^\bullet = R\Gamma(X^\bullet,\calO_{X^\bullet})$; this lemma applies because $R\lim N^\bullet \simeq A$ (as the structure sheaf is truncated, and hence hypercomplete) is bounded.
\end{proof}

In order to get effectivity of hyperdescent for quasi-coherent complexes, it is convenient to use the following criterion for hypercompleteness (in the special case $n=-1$):

\begin{lemma}
\label{lem:HypercompletionTruncated}
Let $\calX$ be an $\infty$-topos. Let $\calF \to \calG$ be a map in $\calX$ which is relatively $n$-truncated for some integer $n$. If $\calG$ is hypercomplete, so is $\calF$.
\end{lemma}
\begin{proof}
	When $\calG = \ast$, this follows from \cite[Lemma 6.5.2.9]{LurieHTT}. In general, one reduces to this case by working in the slice $\infty$-topos $\calX_{/\calG}$ (since a map in $\calX_{/\calG}$ is $\infty$-connective if and only if its image under the forgetful functor $\calX_{/\calG} \to \calX$ is $\infty$-connective by \cite[Proposition 6.5.1.18]{LurieHTT}).
\end{proof}

Hyperdescent for quasi-coherent complexes follows relatively formally from everything so far:

\begin{proof}[Proof of Theorem \ref{thm:hDescentDerived} (4)]
	Let $\calX$ be the $\infty$-topos defined by the $h$-topology on $\Perf^\fp_{/k}$. Write $\calO$ for the structure sheaf, viewed as a sheaf of $E_\infty$-rings on $\calX$, i.e., the $h$-sheaf of $E_\infty$-rings on $\Perf^\fp_{/k}$ defined by $X \mapsto R\Gamma(X,\calO_X)$. Note that any $X \in \Perf^\fp_{/k}$ defines a hypercomplete object of $\calX$ (since it is $0$-truncated); likewise, $\calO$ is a hypercomplete $h$-sheaf of spectra since it is valued in coconnective spectra. Now for any $X \in \Perf^\fp_{/k}$, let $\calG(X)$ be the $\infty$-category of hypercomplete $h$-sheaves of $\calO$-module spectra on $\Perf^\fp_{/X}$, i.e., the $\infty$-category of sheaves $\calO$-module spectra on the hypercompletion of $\calX_{/X}$. By \cite[Remark 2.1.11]{LurieDAGVIII} applied to the hypercompletion of $\calX$, one deduces that $\calG$ is itself a hypercomplete $h$-sheaf. Moreover,  Lemma \ref{lem:QCohComplexHypercomplete} gives a map $D_{qc}(X) \to \calG(X)$ which is fully faithful: the right adjoint is given by the global sections functor when $X$ is affine, so full faithfulness again follows from Lemma \ref{lem:QCohComplexHypercomplete}.  Now, by Theorem \ref{thm:hDescentDerived} (1), the functor $D_{qc}(-)$ is also an $h$-sheaf. Thus, by varying $X$, we obtain a map $D_{qc}(-) \to \calG(-)$ in $\calX$ which is relatively $(-1)$-truncated\footnote{This means that the fibres are either contractible or empty.} (by full faithfulness), and whose target is hypercomplete. By Lemma \ref{lem:HypercompletionTruncated}, it follows that $D_{qc}(-)$ is a hypercomplete $h$-sheaf.
\end{proof}

\begin{remark}
	The proof of effectivity of hyperdescent given above is a bit abstract. Concretely, one may argue as follows: if $X^\bullet \to Y$ is a $h$-hypercover in $\Perf^\fp_{/k}$, and $K^\bullet \in \lim D_{qc}(X^\bullet)$ is a hyperdescent datum, then $K^\bullet$ defines a unique hypercomplete $h$-sheaf $L$ of $\calO$-module spectra on $Y$ by the formula
	\[ (g:U \to Y) \mapsto L(U) := R\Gamma(U \times_{Y} X^\bullet, g^* K^\bullet).\]
	The formation of $L$ commutes with base change, by construction. It remains to check that $L$ is quasi-coherent. By general nonsense and Lemma~\ref{lem:QCohComplexHypercomplete}, $L|_{X^i} \simeq K^i$, and thus this restriction is quasi-coherent for all $i$. In particular, for any $g:U \to Y$ that factors through $X^0$, we know that $L|_U$ is quasi-coherent. But then $M := L|_{X^0}$ is an object of $D_{qc}(X^0)$ equipped with canonical Cech descent data for the map $X^0 \to Y$ (since $L$ comes from $Y$). Since the latter is an $h$-cover, we finish by invoking Theorem \ref{thm:hDescentDerived} (1).
\end{remark}

	\subsection{Descent for Witt vector cohomology}
	\label{ss:WittVectorCoh}
Fix a noetherian $\F_p$-scheme $S$ of finite Krull dimension. The next result of Berthelot-Bloch-Esnault (see \cite[Theorem 2.4 \& Proposition 3.2]{BBE}) identifies a descent property for the cohomology of this sheaf on $\Sch^\fp_{/S}$ after inverting $p$:

\begin{theorem}
	The functor $X \mapsto R\Gamma(X,W\calO_X)[\frac{1}{p}]$ is an $h$-sheaf of spectra on $\Sch^{\fp}_{/S}$.
\end{theorem}

We can improve this to a descent property where only $F$ is inverted:

\begin{proposition}
	\label{prop:BBEWittDescent}
	The functor $X \mapsto R\Gamma(X, W\calO_X)[\frac{1}{F}]$ is an $h$-sheaf of spectra on $\Sch^{\fp}_{/S}$.
\end{proposition}

It is convenient to use some derived algebraic geometry to establish this result. For this, we use the following notation.

\begin{notation}
All occurrences of derived geometry in this section are meant with respect to simplicial commutative rings. For any such derived scheme $X$, write $X^{cl}$ for its underlying classical scheme. Fibre products of ordinary schemes are computed in the ordinary sense unless the adjective `derived' is present. By evaluating the Witt vector construction termwise on a simplicial commutative ring, we obtain a sheaf $W(\calO_X)$ of simplicial commutative rings on any derived $\F_p$-scheme $X$. The Frobenius on $X$ induces an endomorphism $F:W(\calO_X) \to W(\calO_X)$, and we will study the sheaf $W(\calO_X)[\frac{1}{F}]$. There is also a $V$ operator $W(\calO_X) \to W(\calO_X)$ such that $FV = VF = p$, but we do not use $V$.

\end{notation}

The key observation is that the presheaf in Proposition~\ref{prop:BBEWittDescent} is insensitive to the derived structure.

\begin{lemma}
\label{lem:WittSCR}
Let $X$ be a derived $\F_p$-scheme. Then pullback along $X^{cl} \hookrightarrow X$ induces an isomorphism
\[ W(\calO_X)[\frac{1}{F}] \simeq W(\calO_{X^{cl}})[\frac{1}{F}].\]
\end{lemma}
\begin{proof}
We may assume $X = \Spec(A)$ is affine. We must show that $W(A)[\frac{1}{F}] \simeq W(\pi_0(A))[\frac{1}{F}]$. The definition of $W(-)$ identifies $W(A)$ with the simplicial set $\prod_{i \in \N} A$ in a manner that is compatible with $F$, functorial in the map $A \to \pi_0(A)$, and sends $0 \in W(A)$ to $0 = (0,0,...) \in \prod_{i \in \N} A$. Applying $\pi_0$ then immediately gives the $\pi_0$-version of the desired statement. It remains to show that $W(A)[\frac{1}{F}]$ is discrete. The preceding description of $W(A)$ also gives an isomorphism $\pi_i(W(A),0) \simeq \prod_{i \in \N} \pi_i(A,0)$ of groups compatible with $F$; here we use the Eckman-Hilton observation that, given a simplicial abelian group $K$, the two induced group structures on $\pi_i(K,0)$ for $i > 0$ (one coming from the usual group structure on higher homotopy groups which only depends on the simplicial set underlying $K$, the other coming from the group structure on $K$) are identical.  As $F$ kills $\pi_i(A,0)$ for $i > 0$ by (the proof of) Proposition~\ref{prop:FrobSCR}, it follows that $F$ must kill $\pi_i(W(A),0)$ for $i > 0$, and thus $W(A)[\frac{1}{F}]$ is discrete, as wanted.
\end{proof}

\begin{remark}
As a consequence of Lemma~\ref{lem:WittSCR} and the formula $VF = p$, one has: if $A$ is a simplicial commutative $\F_p$-algebra, then the cone of the canonical map $W(A) \to W(\pi_0(A))$ is killed by $p$. 
\end{remark}

The next lemma is the crucial ingredient in the proof of Proposition~\ref{prop:BBEWittDescent}, and is the only spot in the proof where derived schemes arise (via derived fibre products of non-flat maps).

\begin{lemma}
\label{lem:FibreSeqFibreProducts}
Consider a cartesian square 
\[ \xymatrix{ E \ar[r] \ar[d] & X \ar[d]^-f \\
			Z \ar[r]^-i & Y }\]
of qcqs $\F_p$-schemes with $Y$ affine, and $Z \hookrightarrow Y$ a constructible closed immersion. Assume that pullback induces a fibre sequence
\begin{equation}
\label{eq:FibreSeqBlowup}
R\Gamma(Y,\calO_Y) \to R\Gamma(X,\calO_X) \oplus R\Gamma(Z,\calO_Z) \to R\Gamma(E,\calO_E).
\end{equation}
For any integer $n \geq 1$, let $X^{n/Y}$ be the $n$-fold self fibre product of $X$ over $Y$, and set $E^{n/Y} = X^{n/Y} \times_Y Z$, so $E^{n/Y}$ is the $n$-fold self fibre product of $E$ over $Y$ (and thus also over $Z$, since $Z \to Y$ is a monomorphism). Then pullback induces a fibre sequence
\[ R\Gamma(Y, W(\calO_Y)[\frac{1}{F}]) \to R\Gamma(X^{n/Y}, W(\calO_{X^{n/Y}})[\frac{1}{F}]) \oplus R\Gamma(Z, W(\calO_Z)[\frac{1}{F}]) \to R\Gamma(E^{n/Y}, W(\calO_{E^{n/Y}})[\frac{1}{F}]).\]
\end{lemma}
\begin{proof}
As the Witt vector functor $W(-)$ is a derived inverse limit of the truncated Witt vector functors $W_m(-)$, and because each $W_m(-)$ is an $m$-fold iterated extension of copies of $\G_a$ (with Frobenius twists), the case $n=1$ is immediate from the fibre sequence \eqref{eq:FibreSeqBlowup}, giving the desired fibre sequence
\begin{equation}
\label{eq:FibreSeqn1}
R\Gamma(Y, W(\calO_Y)[\frac{1}{F}]) \to R\Gamma(X, W(\calO_X)[\frac{1}{F}]) \oplus R\Gamma(Z, W(\calO_Z)[\frac{1}{F}]) \to R\Gamma(E, W(\calO_E)[\frac{1}{F}]).
\end{equation}
Note that we did not need to invert $F$ to get the above sequence; this will not be true for higher $n$.

Assume now that $n=2$. Since $Y$ is affine, applying $- \otimes_{R\Gamma(Y,\calO_Y)} R\Gamma(X,\calO_X)$ to the fibre sequence \eqref{eq:FibreSeqBlowup}, and using that coherent cohomology commutes with base change in the derived setting, we get a fibre sequence
\[ R\Gamma(X,\calO_X) \to R\Gamma(X \times_Y^L X, \calO_{X \times_Y^L X}) \oplus R\Gamma(Z \times_Y^L X, \calO_{Z \times_Y^L X}) \to R\Gamma(E \times_Y^L X, \calO_{E \times_Y^L X}).\]
Arguing as above for $n=1$, this gives a fibre sequence
\begin{align*}
R\Gamma(X,W(\calO_X)[\frac{1}{F}]) \to R\Gamma(X \times_Y^L X, W(\calO_{X \times_Y^L X})[\frac{1}{F}]) \oplus & R\Gamma(Z \times_Y^L X, W(\calO_{Z \times_Y^L X})[\frac{1}{F}]) \to \\
 & R\Gamma(E \times_Y^L X, W(\calO_{E \times_Y^L X})[\frac{1}{F}]).
 \end{align*}
By Lemma~\ref{lem:WittSCR}, we can ignore the derived structure, so this sequence is identified with
\[ R\Gamma(X,W(\calO_X)[\frac{1}{F}]) \to R\Gamma(X^{2/Y}, W(\calO_{X^{2/Y}})[\frac{1}{F}]) \oplus R\Gamma(E, W(\calO_E)[\frac{1}{F}]) \to R\Gamma(E^{2/Y}, W(\calO_{E^{2/Y}})[\frac{1}{F}]).\]
Comparing this last sequence with \eqref{eq:FibreSeqn1} and simplifying, we get
\[ R\Gamma(Y, W(\calO_Y)[\frac{1}{F}]) \to R\Gamma(X^{2/Y}, W(\calO_{X^{2/Y}})[\frac{1}{F}]) \oplus R\Gamma(Z, W(\calO_Z)[\frac{1}{F}]) \to R\Gamma(E^{2/Y}, W(\calO_{E^{2/Y}})[\frac{1}{F}]),\]
as wanted for $n=2$. For $n \geq 3$, one argues similarly by induction.
\end{proof}

We now prove the $h$-descent result, essentially using the criterion in Theorem~\ref{CritHSheaf}. As it is not clear to us how to check criterion (ii) in Theorem~\ref{CritHSheaf} directly, we instead follow the {\em proof} of Theorem~\ref{CritHSheaf} below.

\begin{proof}[Proof of Proposition~\ref{prop:BBEWittDescent}]
For simplicity, write $\calF(X) = R\Gamma(X, W(\calO_X)[\frac{1}{F}])$. If $f:X \to Y$ is a finite universal homeomorphism, then $\calF(X) \simeq \calF(Y)$ since a power of Frobenius on either $X$ or $Y$ factors over $f$. In particular, $\calF(-)$ carries nilimmersions to isomorphisms. This fact will be used without further comment.

We begin by checking that $\calF(-)$ is an fppf sheaf. Since the functor $R\Gamma(X,W\calO_X)$ takes on coconnective values, it is enough to check this on affines and before inverting $F$, so we want: if $A \to B$ is an fppf map of $\F_p$-algebras with Cech nerve $B^\bullet$, then $W(A) \simeq \lim W(B^\bullet)$. Exchanging limits, we reduce to the analogous statement for $W_n(-)$, which follows from standard exact sequence expressing $W_n(-)$ as an iterated extension of (Frobenius-twisted) copies of $W_1(-)$.

To check $h$-descent, we must check the following: if $f:X \to Y$ is a proper surjective map of finitely presented $S$-schemes with Cech nerve $X^{\bullet/Y}$ (in the classical sense), then $\calF(Y) \simeq \lim \calF(X^{\bullet/Y})$. We prove this by induction on $\dim(Y)$. If $\dim(Y) = 0$, then we are reduced to the case of fppf descent as in Lemma~\ref{ProperGood}. In general, using Lemma~\ref{ProperGood} and Lemma~\ref{DevissageLemma} (and the fact that $\calF(-)$ is an fppf sheaf that converts nilimmersions into isomorphisms), we may assume that $f$ is an isomorphism outside some $Z \subset Y$ with $\dim(Z) < \dim(Y)$; in fact, by the proof of Lemma~\ref{ProperGood}, we can even arrange for $X$ to be the blowup of $Y$ along $Z$. Let $E = X \times_Y Z$ be the exceptional divisor, so $H^i(X, \calO_X(-nE)) = 0$ for all $i > 0$ and $n \geq n_0$ by Serre vanishing (as $\calO_X(-E)$ is relatively ample for $X \to Y$ by the construction of blowups). Replacing $E$ with a suitable thickening (and similarly for $Z$, to ensure $Z$ still receives a map from $E$), we may assume that $H^i(X, \calO_X(-E)) = 0$ for $i > 0$. Now set $Y' = X \sqcup_E Z$ be the pushout of $E \hookrightarrow X$ along $E \to Z$. This gives a square
\[ \xymatrix{ E \ar[r] \ar[d] & X \ar[d] \\
			Z \ar[r] & Y' }\]
			which is a Cartesian square up to universal homeomorphisms, and the induced map $Y' \to Y$ is a finite universal homeomorphism. In particular, $\calF(Y) \simeq \calF(Y')$ and $\calF(X^{\bullet/Y}) \simeq \calF(X^{\bullet/Y'})$, so we may replace $Y$ with $Y'$. (Note that in the process of making these replacements, we have destroyed the property of $f$ being a blowup or that $E$ is the scheme-theoretic preimage of $Z$, but have preserved the crucial consequence that $H^i(X, \calO_X(-E)) = 0$ for $i > 0$.) 
			
Now, by induction on dimension, we have $\calF(Z) \simeq \lim \calF(E^{\bullet/Z})$. As in the proof of Theorem~\ref{CritHSheaf}, we are reduced to showing that for each $n > 0$, applying $\calF(-)$ to the square
\[ \xymatrix{ E^{n/Y} \ar[r] \ar[d] & X^{n/Y} \ar[d] \\
			Z \ar[r] & Y}\]
			induces a fibre sequence
			\[ \calF(Y) \to \calF(X^{n/Y}) \oplus \calF(Z) \to \calF(E^{n/Y})\]
			of spectra. Using Lemma~\ref{lem:FibreSeqFibreProducts}, it is enough to show that
\[ R\Gamma(Y,\calO_Y) \to R\Gamma(X, \calO_X) \oplus R\Gamma(Z,\calO_Z) \to R\Gamma(E,\calO_E)\]
is a fibre sequence. This follows from the conjunction of the following facts: the first term has no $H^i$ for $i > 0$ as $Y$ is affine, the second map induces an isomorphism on applying $H^i$ for $i > 0$ as $H^i(X, \calO_X(-E)) = 0$ for $i > 0$ (and because $Z$ are affine), and applying $H^0$ gives a sequence
\[ H^0(Y,\calO_Y) \to H^0(X,\calO_X) \oplus H^0(Z,\calO_Z) \to H^0(E,\calO_E)\]
which is left-exact as $Y$ is the pushout $X \sqcup_E Z$, and right exact because $H^0(X,\calO_X) \to H^0(E,\calO_E)$ is surjective as $H^1(X, \calO_X(-E)) = 0$.
\end{proof}

\section{Appendix: Determinants}\label{Appendix}

Let $R$ be a commutative ring. Our goal in this section is to recall the definition of the natural map
\[
\det: K(R)\to \gPic^\Z(R)
\]
from the $K$-theory spectrum of $R$ to the Picard groupoid of graded line bundles $\gPic^\Z(R)$. Intuitively, this is just the map sending a finite projective $R$-module $M$ to $(\det M, \rk M)$.

Recall that a symmetric monoidal category is a category $C$ equipped with a functor $\otimes: C\times C\to C$, a unit object $1\in C$, as well as a unitality constraint $\eta_X: 1\otimes X\cong X$, a commutativity constraint $c_{X,Y}: X\otimes Y\cong Y\otimes X$ and an associativity constraint $a_{X,Y,Z}: X\otimes (Y\otimes Z)\cong (X\otimes Y)\otimes Z$ functorially in $X,Y,Z\in C$, satisfying certain compatibility conditions. Moreover, there is a notion of symmetric monoidal functor between symmetric monoidal categories, which is a functor which commutes with $\otimes$ in the appropriate sense. As usual, there is also a notion of a natural transformation between symmetric monoidal functors. The definitions were originally given by MacLane, \cite{MacLane}.

\begin{definition} A symmetric monoidal category $C$ is called \emph{strict} if $c_{X,X}: X\otimes X\cong X\otimes X$ is the identity for all $X\in C$.
\end{definition}

In general, the axioms of a symmetric monoidal category say that $c_{X,Y}\circ c_{Y,X} = \id$ for all $X,Y\in C$, so in particular $c_{X,X}$ is always an involution.

Note that whenever $C$ is a symmetric monoidal category, the subcategory $C^\simeq\subset C$ consisting of all objects, but only isomorphisms as morphisms, is another symmetric monoidal category (as all extra data concerns isomorphisms). Let us refer to symmetric monoidal categories all of whose morphisms are isomorphisms as symmetric monoidal groupoids.

Below, we will recall that any symmetric monoidal category $C$ admits a $K$-theory spectrum $K(C)$, and any symmetric monoidal functor $F: C\to D$ induces a map of spectra $K(F): K(C)\to K(D)$. Applying this to the inclusion $C^\simeq\subset C$ will give an equivalence $K(C^\simeq)\simeq K(C)$; actually, an equality $K(C^\simeq) = K(C)$. For this reason, we largely restrict to symmetric monoidal groupoids in the following.

The following examples are our main interest.

\begin{example} Fix a commutative ring $R$.
\begin{enumerate}
\item[{\rm (i)}] The groupoid $\gVect(R)$ of finite projective $R$-modules is a symmetric monoidal category with respect to $\oplus$, unit $0$, and the standard unitality, commutativity and associativity constraints. For example, $c_{X,Y}: X\oplus Y\cong Y\oplus X$ sends $(x,y)$ to $(y,x)$. Note that $c_{X,X}$ is \emph{not} the identity map for $X\neq 0$.

We remark that the groupoid of finite projective $R$-modules is also symmetric monoidal with respect to the tensor product $\otimes$; we will not use this symmetric monoidal structure.
\item[{\rm (ii)}] The groupoid $\gPic(R)$ of line bundles over $R$ (i.e., finite projective $R$-modules of rank $1$) is a symmetric monoidal category with respect to $\otimes$, unit $1$, and the standard unitality, commutativity and associativity constraints. For example, $c_{X,Y}: X\otimes Y\cong Y\otimes X$ sends $x\otimes y$ to $y\otimes x$. Note that $c_{X,X}$ \emph{is} the identity map for all $X\in \gPic(R)$.
\item[{\rm (iii)}] The groupoid $\gPic^\Z(R)$ of $\Z$-graded line bundles. Here, an object is given by a pair $(L,f)$ where $L\in \gPic(R)$, and $f:\Spec(R) \to \Z$ is a locally constant function; the set $\Isom( (L,f), (M,g))$ is empty if $f \neq g$, and given by $\Isom(L,M)$ otherwise. The unit is given by $(1,0)$. This groupoid is endowed with a symmetric monoidal structure $\otimes$ where $(L,f) \otimes (M,g) := (L \otimes M, f + g)$, and the commutativity constraint 
\[ (L \otimes M, f+g) =: (L,f) \otimes (M,g) \simeq (M,g) \otimes (L,f) := (M \otimes L, g+f)\]
determined by the rule 
\[ \ell \otimes m \mapsto  (-1)^{f \cdot g} m \otimes \ell,\]
and the obvious associativity constraint. Note that $c_{(L,f),(L,f)}: (L,f)\otimes (L,f)\cong (L,f)\otimes (L,f)$ is given by multiplication with $(-1)^f$, and is thus \emph{not} in general the identity.
\end{enumerate}
\end{example}

The following proposition is central.

\begin{proposition}\label{DetSymmMon} There is a natural symmetric monoidal functor
\[
\det: \gVect(R)\to \gPic^\Z(R)
\]
sending $M\in \gVect(R)$ to $(\det M,\rk M)\in \gPic^\Z(R) = \gPic(R) \times H^0(\Spec(R),\Z)$.
\end{proposition}

Note that this functor does \emph{not} factor through $\gPic(R)$.

\begin{proof} All verifications are automatic, and the only critical observation is the following. If $M, N\in \gVect(R)$, then the commutativity constraint $c_{M,N}: M\oplus N\cong N\oplus M$ swapping $N$ and $M$ induces multiplication by $(-1)^{\rk M \rk N}$ on
\[
\det(M)\otimes \det(N)\cong \det(M\oplus N)\buildrel{\det c_{M,N}}\over\cong \det(N\oplus M)\cong \det(N)\otimes \det(M)\ ,
\]
if one identifies $\det(M)\otimes \det(N)$ and $\det(N)\otimes \det(M)$ using the usual commutativity constraint of $\gPic(R)$.
\end{proof}

Now we recall the construction of the $K$-theory spectrum. Let $C$ be a symmetric monoidal category, with $C^\simeq\subset C$ the underlying symmetric monoidal groupoid. By definition, $K(C) = K(C^\simeq)$, so assume that $C$ is a groupoid to start with. Recall that a groupoid is equivalent to a space whose only nonzero homotopy groups are $\pi_0$ and $\pi_1$. Concretely, this can be realized by the nerve construction, which associates to any category $C$ the simplicial set $N(C)$ whose $n$-simplices are chains of $n-1$ morphisms,
\[
X_0\buildrel f_0\over\to X_1\buildrel f_1\over\to\ldots\buildrel {f_{n-1}}\over\to X_n\ .
\]
In particular, the $0$-simplices $N(C)_0$ are the objects of $C$, and the $1$-simplices $N(C)_1$ are the morphisms of $C$. Higher simplices encode the composition law, and degenerate simplices encode identity morphisms. If $C$ is a groupoid, then $N(C)$ is a Kan complex, whose geometric realization $|N(C)|$ is a (compactly generated) topological space whose only nonzero homotopy groups are $\pi_0$ and $\pi_1$. In fact, $|N(C)|$ is homotopy equivalent to the disjoint union
\[
\bigsqcup_{X\in C/\simeq} B\Aut(X)
\]
of the classifying spaces of the automorphism group $\Aut(X)$, over all isomorphism classes of objects $X\in C$. In the following, we refer to Kan complexes as `spaces'. They are naturally organized into the $\infty$-category $\calS$ of spaces. Here and in the following, we make use of the theory of $\infty$-categories, cf. \cite{LurieHTT}, which were previously defined under the name of weak Kan complexes (by Boardman-Vogt) and quasicategories (by Joyal). Using this language, we can easily state precise definitions and theorems without having to go into the pain of detailed constructions. Giving detailed proofs is more elaborate, and we only try to convey the meaning of the statements below.

As $C$ is symmetric monoidal, the space $N(C)$ is equipped with an addition law $N(\otimes): N(C)\times N(C)=N(C\times C)\to N(C)$. Moreover, this addition law is commutative and associative up to coherent isomorphisms, as expressed by the commutativity and associativity constraints. However, the addition law is not strictly commutative and associative. Such spaces equipped with a coherently commutative and associative (but not necessarily invertible) addition law are known as (special) $\Gamma$-spaces, as defined by Segal, \cite{SegalGammaSpaces}, or as $\E_\infty$-monoids, as defined by May, \cite{MayEInfty}. Although our definition is essentially that of a special $\Gamma$-space, we prefer to call them $\E_\infty$-monoids.\footnote{Lurie in \cite[Remark 2.4.2.2]{LurieHA} calls these commutative monoids. However, with \cite[Notation 5.1.1.6]{LurieHA}, this becomes the same thing as an $\E_\infty$-monoid.}

\begin{definition} Let $\Fin_\ast$ be the category of finite pointed sets.\footnote{This is Segal's category $\Gamma$ (or its opposite, depending on references).} For $n\geq 0$, let $[n] = \{0,1,\ldots,n\}\in \Fin_\ast$ be the object pointed at $0$. An $\E_\infty$-monoid is a functor
\[
X: N(\Fin_\ast)\to \calS
\]
of $\infty$-categories to the $\infty$-category $\calS$ of Kan complexes, such that for all $[n]\in \Fin_\ast$, the natural map
\[
X([n])\to \prod_{i=1}^n X([1])
\]
is a weak equivalence, where the map is induced by the $n$ maps $[n]\to [1]$ contracting everything to $0$ except $i$.
\end{definition}

The $\E_\infty$-monoids are naturally organized into an $\infty$-category (a full subcategory of the $\infty$-category of functors from $N(\Fin_\ast)$ to $\calS$), which we call the $\infty$-category of $\E_\infty$-monoids $\Mon_{\E_\infty}$.

For an $\E_\infty$-monoid $X$, the $0$-th space $X([0])$ is weakly contractible, as it is weakly equivalent to an empty product. We refer to $X([1])$ as the underlying space of the $\E_\infty$-monoid $X$, and will sometimes confuse $X$ with $X([1])$. It is pointed in the sense that it comes equipped with a map $X([0])\to X([1])$ from a weakly contractible space; the space $X([0])$ is the 'unit' of $X([1])$. Moreover, using the map $[2]\to [1]$ sending only $0$ to $0$, we get a natural map
\[
X([1])\times X([1])\simeq X([2])\to X([1])
\]
which gives an addition law on $X([1])$. The higher data precisely ensure that the unit is unital, and that the addition law is commutative and associative 'up to coherent homotopy'. In particular, for an $\E_\infty$-monoid $X$, the set of connected components $\pi_0 X := \pi_0 X([1])$ forms a commutative monoid.

\begin{construction}\label{SymmMonToEInfty} Let $C$ be a symmetric monoidal groupoid. We define an $\E_\infty$-monoid $N(C)$ in the following way. For each finite pointed set $(S,s)\in \Fin_\ast$, let
\[
N(C)(S) = N\left(\{(X_T)_{T\subset S\setminus \{s\}}, X_{T\sqcup T^\prime}\cong X_T\otimes X_{T^\prime}\}\right)
\]
be the nerve of the groupoid of objects $X_T\in C$ for all $T\subset S\setminus \{s\}$ which are equipped with compatible isomorphisms $X_{T\sqcup T^\prime}\cong X_T\otimes X_{T^\prime}$ for any disjoint subsets $T, T^\prime\subset S\setminus \{s\}$. For any morphism $f: (S,s)\to (S^\prime,s^\prime)$, let the map
\[
N(C)(S)\to N(C)(S^\prime)
\]
be given by
\[
(X_T)_{T\subset S\setminus \{s\}}\mapsto (X_{f^{-1}(T^\prime)})_{T^\prime\subset S^\prime\setminus \{s^\prime\}}\ .
\]
\end{construction}

\begin{remark} Here, we have defined an actual functor from $\Fin_\ast$ to the \emph{category} of Kan complexes. The naive definition of $N(C)(S)$ would be $N(C)(S) = N(C)^{S\setminus \{s\}}$, with the map
\[
N(C)^{S\setminus \{s\}}\to N(C)^{S^\prime\setminus \{s^\prime\}}
\]
for a map $f: (S,s)\to (S^\prime,s^\prime)$ of pointed sets being given by
\[
(X_t)_{t\in S\setminus \{s\}}\mapsto (\otimes_{t\in f^{-1}(t^\prime)} X_t)_{t^\prime\in S^\prime\setminus \{s^\prime\}}\ .
\]
This definition is not compatible with composition on the nose, but can be made into a functor \emph{of $\infty$-categories} $N(C): \Fin_\ast\to \calS$. However, the verification of this fact is most easily done by constructing a weakly equivalent strict functor, as above.
\end{remark}

We leave it to the reader to spell out the definition of compatibility of the isomorphisms $X_{T\sqcup T^\prime}\cong X_T\otimes X_{T^\prime}$; verifying that $N(C)([n])\to N(C)^n$ is an equivalence uses the precise axioms of a symmetric monoidal category.

\begin{definition} An $\E_\infty$-monoid $X$ is called grouplike if the commutative monoid $\pi_0(X)$ is a group. Let $\Mon_{\E_\infty}^\gp\subset \Mon_{\E_\infty}$ denote the full subcategory of grouplike $\E_\infty$-monoids.
\end{definition}

Grouplike $\E_\infty$-monoids are equivalent to connective spectra.

\begin{definition} The $\infty$-category of connective spectra $\Sp^{\geq 0}$ is given by the limit of
\[
\calS_\ast\buildrel\Omega\over \leftarrow \calS_\ast^{\geq 1}\buildrel\Omega\over\leftarrow \calS_\ast^{\geq 2}\buildrel\Omega\over\leftarrow\ldots
\]
in the $\infty$-category of $\infty$-categories, where $\calS_\ast$ is the $\infty$-category of pointed spaces, $\calS_\ast^{\geq i}\subset \calS_\ast$ is the full subcategory of $i$-connected spaces\footnote{i.e., of pointed spaces $X\in \calS_\ast$ such that $\pi_j X=0$ for $j<i$.}, and $\Omega: \calS_\ast\to \calS_\ast$, $X\mapsto \ast\times_X \ast$ is the loop space functor.
\end{definition}

The following theorem is due to Segal, \cite[Proposition 3.4]{SegalGammaSpaces}, cf. \cite[Theorem 5.2.6.10, Remark 5.2.6.26]{LurieHA}. The functor $\Mon_{\E_\infty}^\gp\to \Sp^{\geq 0}$ is usually called ``the infinite loop space machine".

\begin{theorem}\label{EMonVsSpectra} There is a natural equivalence of $\infty$-categories
\[
\Mon_{\E_\infty}^\gp\simeq \Sp^{\geq 0}\ .
\]
\end{theorem}

May-Thomason, \cite{MayThomason}, showed that there is essentially only one such equivalence (although there are many constructions).\footnote{For an $\infty$-categorical discussion of these matters, see recent work of Gepner-Groth-Nikolaus, \cite{GepnerGrothNikolaus}.} Let us briefly sketch the definition of the functors in either direction. Let $X\in \Sp^{\geq 0}$ be a connective spectrum. Then the natural map
\[
X\sqcup X\to X\times X
\]
from the coproduct to the product, taken in the $\infty$-category of connective spectra (equivalently, of spectra), is an equivalence, cf. \cite[Lemma 1.1.2.10]{LurieHA}. Thus, the obvious map $X\sqcup X\to X$ extends to a map $X\times X\to X$, which is the 'addition law' of the spectrum. Repeating the same arguments for arbitrary finite (co)products shows that the forgetful functor
\[
\Mon_{\E_\infty}(\Sp^{\geq 0})\to \Sp^{\geq 0}
\]
from the $\infty$-category of $\E_\infty$-monoids in $\Sp^{\geq 0}$ to $\Sp^{\geq 0}$, is an equivalence, i.e. any (connective) spectrum comes equipped with a canonical $\E_\infty$-monoid structure. Composing with the forgetful map
\[
\Mon_{\E_\infty}(\Sp^{\geq 0})\to \Mon_{\E_\infty}(\calS) = \Mon_{\E_\infty}
\]
gives a functor $\Sp^{\geq 0}\to \Mon_{\E_\infty}$, preserving all $\pi_i$. As all $\pi_i$, including $\pi_0$, of a spectrum are groups, it follows that the functor takes values in $\Mon_{\E_\infty}^\gp$.

Conversely, start with an $\E_\infty$-monoid $X\in \Mon_{\E_\infty}$ (not necessarily grouplike). Then $X$ admits a classifying space $BX$. Indeed, there is a natural functor $\Delta^\op\to \Fin_\ast$ (cf. \cite[p. 295]{SegalGammaSpaces}) taking the $m$-simplex $\Delta^m$ to $[m]$. This allows one to view $X$ as a simplicial space $X^\prime: N(\Delta^\op)\to N(\Fin_\ast)\to \calS$, which can be turned into a space $BX = \colim_{\Delta^\op} X^\prime\in \calS$. In fact, one checks directly that $BX$ still carries a canonical $\E_\infty$-monoid structure, cf. \cite[Definition 1.3]{SegalGammaSpaces}. This gives a functor
\[
B: \Mon_{\E_\infty}\to \Mon_{\E_\infty}\ .
\]
There is a natural map $X\to \Omega BX$ of $\E_\infty$-monoids.\footnote{The $\infty$-category $\Mon_{\E_\infty}$ admits all small limits, and $\Mon_{\E_\infty}\to \calS_\ast$, $X\mapsto X([1])$, preserves all small limits. In particular, this applies to the loop space $X\mapsto \Omega X = \ast\times_X \ast$, giving a functor $\Omega: \Mon_{\E_\infty}\to \Mon_{\E_\infty}$.} The following is \cite[Proposition 1.4]{SegalGammaSpaces}.

\begin{proposition}\label{CharGrouplike} If $X\in \Mon_{\E_\infty}$ is $k$-connected, then $BX$ is $(k+1)$-connected. Moreover, the natural map
\[
X\to \Omega BX
\]
is an equivalence if and only if $X$ is grouplike.
\end{proposition}

\begin{remark} As $\pi_0(\Omega BX) = \pi_1(X)$ is a group, $\Omega BX$ is always a grouplike $\E_\infty$-monoid. Also, $BX$ is always $1$-connected, and thus grouplike (as $\pi_0 BX=0$ is a group).
\end{remark}

Thus, if $X\in \Mon_{\E_\infty}^\gp$ is a grouplike $\E_\infty$-monoid, the sequence of spaces $X,BX,B^2X,\ldots$ forms a connective spectrum, giving a functor
\[
\Mon_{\E_\infty}^\gp\to \Sp^{\geq 0}\ .
\]

In particular, the construction shows that as with classical monoids, one can form the group completion.

\begin{corollary} The full subcategory $\Mon_{\E_\infty}^\gp\subset \Mon_{\E_\infty}$ is a reflective subcategory, cf. \cite[Remark 5.2.7.9]{LurieHTT}, i.e. it admits a left adjoint, called the group completion,
\[
\Mon_{\E_\infty}\to \Mon_{\E_\infty}^\gp: X\mapsto X^\gp=\Omega BX\ .
\]
\end{corollary}

Recall that $K_0(C)$ is the group completion of the commutative monoid of objects of $C$ \emph{up to isomorphism}. Using the above machinery, we can erase the words ``up to isomorphism", and arrive at the definition of higher algebraic $K$-theory:

\begin{definition} Let $C$ be a symmetric monoidal category. The $K$-theory spectrum $K(C)\in \Sp^{\geq 0}$ is defined to be the image under $\Mon_{\E_\infty}^\gp\simeq \Sp^{\geq 0}$ of the group completion $N(C^\simeq)^\gp$ of the $\E_\infty$-monoid $N(C^\simeq)$ from Construction \ref{SymmMonToEInfty}.
\end{definition}

Clearly, this construction is functorial in $C$. There is a certain situation in which the group completion is unnecessary.

\begin{definition} A symmetric monoidal category $C$ is called a Picard groupoid if $C$ is a groupoid, and any object $X\in C$ admits an inverse $X^{-1}\in C$ such that $X\otimes X^{-1}\cong 1$.
\end{definition}

The groupoids $\gPic(R)$ and $\gPic^\Z(R)$ are examples of Picard groupoids (explaining the name), as one can form the inverse of a line bundle.

\begin{proposition} A symmetric monoidal groupoid $C$ is a Picard groupoid if and only if the $\E_\infty$-monoid $N(C)$ is grouplike.
\end{proposition}

\begin{proof} One has an identification of $\pi_0 N(C) = C/\simeq$ with the monoid (under $\otimes$) of objects of $C$ up to isomorphism. But, by definition, $C$ is a Picard groupoid if and only if $C/\simeq$ is a group.
\end{proof}

In particular, for a Picard groupoid, the space underlying $K(C)$ is just the nerve $N(C)$, which is $1$-truncated. One can show that this induces an equivalence between Picard groupoids and $1$-truncated connective spectra, cf. \cite[\S 3]{Patel}. The idea is that both can be identified with grouplike $\E_\infty$-monoids in groupoids.

Because of this equivalence, we will often confuse a Picard groupoid with its $K$-theory spectrum, and in particular we continue to write $\gPic^\Z(R)$ for the corresponding ($1$-truncated) connective spectrum.

\begin{definition} Let $R$ be a commutative ring. The $K$-theory spectrum of $R$ is $K(R) = K(\gVect(R))$.
\end{definition}

\begin{corollary} There is a natural functorial map
\[
\det: K(R)\to \gPic^\Z(R)
\]
of connective spectra.
\end{corollary}

\begin{proof} Apply the $K$-theory functor to
\[
\det: \gVect(R)\to \gPic^\Z(R)\ .
\]
\end{proof}

As $\gPic^\Z(R)$ is $1$-truncated, the map $K(R)\to \gPic^\Z(R)$ factors canonically over the $1$-truncation $\tau_{\leq 1}K(R)$. In fact, $\gPic^\Z(R)$ is, at least after Zariski sheafification, precisely the $1$-truncation of $K(R)$.

\begin{proposition}\label{prop:Pic1TruncationofK} The natural map
\[
\det: \tau_{\leq 1} K(R)\to \gPic^\Z(R)
\]
of presheaves of groupoids on the category of affine schemes becomes an isomorphism after Zariski sheafification.
\end{proposition}

\begin{proof} We need to prove that the maps
\[
K_1(R) = \pi_1 K(R)\to \pi_1 \gPic^\Z(R)
\]
and
\[
K_0(R) = \pi_0 K(R)\to \pi_0 \gPic^\Z(R)
\]
are isomorphisms after Zariski sheafification; equivalently, for local rings $R$. But if $R$ is local, $K_1(R) = R^\times = \pi_1 \gPic^\Z(R)$, and $K_0(R) = \Z = \pi_0 \gPic^\Z(R)$.
\end{proof}

\bibliography{witt-affine-grassmannian}
\end{document}